\documentclass[12pt]{amsart}

\usepackage{enumerate}
\usepackage{fancyvrb}% it creates fancy verbatim text
\usepackage{tcolorbox}% it creates nice colour boxes

\usepackage[final]{showlabels}
\usepackage{setspace}
\usepackage{url}

\usepackage{amsmath}

\usepackage{amssymb,amsthm,graphicx}

\usepackage[normalem]{ulem} %crossout text

\graphicspath{{pictures/}{../pictures/}}
\usepackage{xcolor}
\usepackage{mathtools}

\newcommand{\cvec}[1]{\begin{pmatrix}#1\end{pmatrix}}

\DeclarePairedDelimiter\floor{\lfloor}{\rfloor}
\DeclarePairedDelimiter\set{\{}{\}}
\DeclarePairedDelimiter\paren{(}{)}
\DeclarePairedDelimiter\brac{[}{]}

\newcommand{\card}[1]{\ensuremath{\# #1}}  %affine hull

\DeclareMathOperator{\conv}{conv} %convex hull

\newcommand{\R}{\mathbb{R}}

\newcommand{\V}{\mathcal{V}}

\newcommand{\f}{\varphi}

\newcommand{\ff}{\eta} 
\newcommand{\hh}{\tau} 
\newcommand{\ii}{\theta}
\newcommand{\jj}{\rho}

\renewcommand{\ge}{\geqslant} 
\renewcommand{\le}{\leqslant}

\DeclareMathOperator{\wed}{W}

\usepackage[capitalise]{cleveref} % cleveref needs be loaded between amsmath and amsthm before the definitions of the environments theorems, lemmas, etc

\newtheorem{theorem}{Theorem}
\newtheorem{corollary}[theorem]{Corollary}
\newtheorem{lemma}[theorem]{Lemma}
\newtheorem{proposition}[theorem]{Proposition}

\newtheorem{conjecture}[theorem]{Conjecture}

\theoremstyle{definition}

\newtheorem{fact}{Fact}
\AtEndEnvironment{proof}{\setcounter{fact}{0}}
\newenvironment{factproof}{\noindent\emph{Proof of fact.}}{\hfill$\qed$}

\theoremstyle{remark}
\newtheorem{remark}[theorem]{Remark}

\newtheorem{claim}{Claim}
\AtEndEnvironment{proof}{\setcounter{claim}{0}}
\newenvironment{claimproof}{\noindent\emph{Proof of claim.}}{\hfill$\qed$}

\def\tinyskip{\vspace{2pt}}

\theoremstyle{definition}

\newtheorem{case}{Case}
\AtEndEnvironment{proof}{\setcounter{case}{0}}
\AtEndEnvironment{claimproof}{\setcounter{case}{0}}
\newcounter{subcases}[case]

\renewcommand*\thesubcases{\arabic{subcases}}
\newenvironment{subcases}
{%
 \setcounter{subcases}{0}%
    \def\subcase
      {% 
        \tinyskip
        \par\noindent
        \refstepcounter{subcases}%
        \emph{Subcase \thecase.\thesubcases.}              
      }%
}
{%
}

\usepackage{anysize}
\marginsize{2cm}{2cm}{2cm}{2cm}%\marginsize{left}{right}{top}{bottom}

\crefname{lemma}{Lemma}{Lemmas}
\crefname{theorem}{Theorem}{Theorems}
\crefname{corollary}{Corollary}{Corollaries}
\crefname{proposition}{Proposition}{Propositions}
\crefname{conjecture}{Conjecture}{Conjectures}
\crefname{claim}{Claim}{Claims}
\crefname{case}{Case}{Cases}
\crefname{subcase}{Subcase}{Subcases}
\crefname{fact}{Fact}{Facts}
\begin{document}

\date{\today}
\title{A lower bound theorem for  $d$-polytopes with at most $3d-1$  vertices}
\author{Guillermo Pineda-Villavicencio}
\author{Jie Wang}

\address{School of Information Technology, Deakin University, Locked Bag 20000, Geelong VIC. 3220, Australia}
\email{\texttt{guillermo.pineda@deakin.edu.au},\texttt{wangjiebulingbuling@foxmail.com}}

%\author{David Yost}
%\address{Federation University, Mt. Helen, Vic. 3350, Australia}
%\email{\texttt{d.yost@federation.edu.au}}

\begin{abstract} We prove a lower bound theorem for the number of $k$-faces ($1\le k\le d-2$) in a $d$-dimensional polytope $P$ (or \emph{$d$-polytope}) with up to $3d-1$ vertices. Previous lower bound theorems for $d$-polytopes with few vertices concern those with at most $2d$ vertices~\cite{Xue21}, $2d+1$ vertices~\cite{Xue24,PinYos22}, and $2d+2$ vertices~\cite{PinTriYos24}.
%Previous   lower bound theorems for $d$-polytopes with relatively few vertices consist of those with at most $2d$ vertices, as shown  by Xue~\cite{Xue21}; those with  $2d+1$ vertices, obtained independently  by Xue~\cite{Xue24}  and Pineda-Villavicencio and Yost~\cite{PinYos22}; and those with  $2d+2$ vertices, established by Pineda-Villavicencio et al.~ \cite{PinTriYos24}.

If $P$ has exactly $d+2$ facets and $2d+\ell$ vertices ($\ell\ge 1$), the lower bound is tight for certain combinations of $d$ and $\ell$. When  $P$ has at least $d+3$ facets and $2d+\ell$ vertices ($\ell\ge 1$), the lower bound remains  tight up to $\ell=d-1$, and  equality  for some $1\le k\le d-2$ is attained only when $P$ has precisely  $d+3$ facets. 

We exhibit at least one minimiser for each number of vertices between $2d+1$ and $3d-1$, including two distinct minimisers with $2d+2$ vertices and three with $3d-2$ vertices.
\end{abstract}

\maketitle

\section{Introduction}
 
Two (convex) polytopes are combinatorially isomorphic if their face lattices are isomorphic. In this paper, we do not distinguish between combinatorially isomorphic polytopes. 

Gr\"unbaum~\cite[Sec.~10.2]{Gru03} conjectured that the function 
\begin{equation}\label{eq:function-at-most-2d}      
   \ii_k(d+s,d):=\binom{d+1}{k+1}+\binom{d}{k+1}-\binom{d+1-s}{k+1},\; \text{for  $1\le s\le d$},
\end{equation}
gives the minimum number of $k$-faces of a $d$-polytope with $d+s$ vertices. He proved the conjecture for $s=2,3,4$, and Xue~\cite{Xue21} later completed the proof, including a characterisation of the unique minimisers for $k\in[1\ldots d-2]$. Earlier, partial results had been obtained by Pineda-Villavicencio et~al.~\cite{PinUgoYos15}  for certain values of $k$.

For $k\in [1\ldots d-2]$ (the set of integers between $1$ and $d-2$), each minimiser of \eqref{eq:function-at-most-2d} is a  $(d-s)$-fold pyramid over a  simplicial $s$-prism for $s\in [1\ldots d]$; we refer to such a polytope as an \emph{($s$,$d-s$)-triplex}.

Let $P$ be a $d$-polytope in $\R^{d}$, and let $F$ be a face of $P$. If $K$ is a closed halfspace in $\mathbb{R}^d$ such that the vertices of $P$ not in $K$ are precisely the vertices of $F$, then the polytope $P':=P\cap K$ is said to be obtained by \emph{truncating the face} $F$ of $P$. 

For $d\ge 2$ and $\ell\in [1\ldots d-1]$,  truncating a \emph{simple} vertex (i.e., one that lies in exactly $d$ edges) from the $(\ell+1,d-\ell-1)$-triplex yields a $d$-polytope $J(\ell+1,d)$ with $2d+\ell$ vertices and $d+3$ facets. The  number of $k$-faces of $J(\ell+1,d)$ for $k\in [1\ldots d-1]$ is 
\begin{equation*}
  \begin{aligned}
  \eta_k(2d+\ell,d)&:=\binom{d+1}{k+1}+2\binom{d}{k+1}-\binom{d-\ell}{k+1}-\binom{1}{k+1}\\
  &\;=\binom{d}{k}+3\binom{d}{k+1}-\binom{d-\ell}{k+1}-\binom{1}{k+1}.
  \end{aligned}
%\label{eq:3minus1-function}
\end{equation*}
The paper~\cite{PinUgoYos16a} refers to $J(2,d)$ as the \emph{$d$-pentasm}, $J(3,d)$ as $B(d)$, and $J(d,d)$ as $J(d)$. 

The \emph{Cartesian product} of a $d$-polytope $P\subset \R^{d}$ and a $d'$-polytope $P'\subset \R^{d'}$ is the Cartesian product of the sets $P$ and $P'$: $P\times P'=\set*{(p, p')^{t}\in \R^{d+d'}\mid p\in P,\, p'\in P'}.$
% \begin{equation*}
% P\times P'=\set*{(p, p')^{t}\in \R^{d+d'}\mid p\in P,\, p'\in P'}.
% \end{equation*}

We denote the $d$-simplex by $T(d)$. Then the ($s$,$d-s$)-triplex is a $(d-s)$-fold pyramid over $T(1)\times T(s-1)$.   Let $f_{k}$ denote the number of $k$-faces in a polytope. 

\begin{lemma}[McMullen 1971, {\cite[Sec.~3]{McM71}}]\label[lemma]{lem:dplus2facets} 
Let  $P$ be a $d$-polytope with $d+2$ facets, where $d\ge 2$. Then there exist integers $2\le a\le d$ and $1\le m\le \floor{a/2}$ such that $P$ is a $(d-a)$-fold pyramid over $T(m)\times T(a-m)$. The number of  $k$-faces of $P$ is
\begin{equation}\label{eq:dplus2facets}
\jj_k(a,m,d):= \binom{d+2}{k+2} -\binom{d-a+m+1}{k+2}-\binom{d-m+1}{k+2} +\binom{d-a+1}{k+2}.
\end{equation}
In particular,  $f_{0}(P)=d+1+m(a-m)$.
\end{lemma}

 Pineda-Villavicencio~\cite[Prob.~8.7.11]{Pin24} conjectured that the polytope $J(\ell+1,d)$ and certain $d$-polytopes with $d+2$ facets minimise the number of faces among all $d$-polytopes with at most  $3d-1$ vertices. In a recent paper by Yost and the second author \cite{WanYos_conditional}, the conjecture is confirmed for decomposable $d$-polytopes with $2d+\ell$ vertices, where $\ell \le d-4$.
 
 \begin{conjecture}[{\cite[Prob.~8.7.11]{Pin24}}]
 \label[conjecture]{conj:3minus1}
   Let $d\ge 3$ and $\ell\in [1\ldots d-1]$. Let $P$ be a $d$-polytope with $2d+\ell$ vertices.
\begin{enumerate}[{\rm (i)}]
\item If $P$ has at least $d+3$ facets, then $f_{k}(P)\ge f_{k}(J(\ell+1,d))$ for each $k\in [1\ldots d-2]$.
\item  If $P$ has $d+2$ facets, then it is a $(d-a)$-fold pyramid over $T(m)\times T(a-m)$ for some $a\in [2\ldots d]$ and $2\le m\le \floor{a/2}$.
\end{enumerate}
 \end{conjecture}

We resolve \cref{conj:3minus1} in full. The case $\ell=1$   was previously settled independently by Pineda-Villavicencio and Yost~\cite{PinYos22} and by Xue~\cite{Xue24}; see \cref{thm:2dplus1}. The case $\ell=2$ was verified by Pineda-Villavicencio et al.~\cite{PinTriYos24}; see \cref{thm:2dplus2}. In addition, we show that any $d$-dimensional minimiser of \cref{conj:3minus1}(i) must have exactly $d+3$ facets.

Beyond the polytope $J(\ell+1,d)$, we present other minimisers in \cref{sec:additional-minimisers}.

The proof of our lower bound theorem (\cref{thm:3minus1}) is lengthy and combines case analysis, induction, a versatile counting result of Xue~\cite[Prop.~3.1]{Xue21} (see \cref{cor:number-faces-outside-facet-practical}), and a shelling argument inspired by Blind and Blind~\cite{BliBli99} (see \cref{cl:3minus1-less-2dplus1}). Xue's result, which estimates the number of faces containing at least one vertex from a  set of at most $d$ vertices (see \cref{cl:3minus1-minus-2,cl:3minus1-more-2dplus1,cl:3minus1-less-2dplus1}), is applied in several contexts, underscoring its versatility. The shelling argument  may also be applicable to other face-bounding problems involving $d$-polytopes whose facets have a moderate number of vertices---say, at most $2(d-1)$—and may be of independent interest.

In the conclusions (\cref{sec:conclusions}), we explore possible lower bounds for $d$-polytopes with at least $d+3$ facets and between $3d$ and $4d-4$ vertices. In brief, we conjecture that for $d\ge 5$, the minimisers alternate between two constructions:
\begin{itemize}
    \item a $d$-polytope obtained by truncating a \emph{simple edge} (an edge whose endpoints are both simple vertices) from an $(s_1,d-s_1)$-triplex ($2\le s_1\le d$), and
    \item a $d$-polytope obtained by truncating a simple vertex from $J(s_3+1,d)$ ($1\le s_3\le d-3$).
\end{itemize}

\section{Polytopes with exactly $d+2$ facets}
\label{sec:dplus2-facets} 

 The \emph{direct sum} $P\oplus P'$ of a $d$-polytope $P\subset \R^{d}$ and a $d'$-polytope $P'\subset\R^{d'}$, both containing the origin in their relative interiors, is the $(d+d')$-polytope: 
\begin{equation*}
\label{eq:direct-sums}
P\oplus P'=\conv \left(\left\{\cvec{ p\\ 0_{d'}}\in \R^{d+d'}\middle|\;  p\in P\right\}\bigcup \left\{\cvec{ 0_{d}\\p'}\in \R^{d+d'}\middle|\; p'\in P'\right\}\right),    
\end{equation*} 
where $0_{d}$ denotes the zero vector in $\R^{d}$.

% \begin{lemma}\label{lem:dplus2vertices}\cite[Sec.~6.1]{Gru03}
% Let  $P$ be a $d$-dimensional polytope with $d+2$ vertices, where $d\ge 2$. Then there exist  integers $2\le a\le d$ and $1\le m\le \floor{a/2}$ such that $P$ is a $(d-a)$-fold pyramid over $T(m)\oplus T(a-m)$. The number of  $k$-dimensional faces of $P$ is
% \begin{equation*}
% \binom{d+2}{d-k+1} -\binom{d-a+m+1}{d-k+1}-\binom{d-m+1}{d-k+1} +\binom{d-a+1}{d-k+1}.
% \end{equation*}
% In particular,  $f_{d-1}(P)=d+1+m(a-m)$.
% \end{lemma}

Following Gr{{\"u}}nbaum~\cite[Thm.~6.1.4]{Gru03}, we denote the $(d-a)$-fold pyramid over $T(m)\oplus T(a-m)$ by $T_{m}^{d,d-a}$, where $2\le a\le d$ and $1\le m\le \floor{a/2}$; these are precisely the $d$-polytopes with $d+2$ vertices.  Gr{{\"u}}nbaum~\cite[p.~101]{Gru03} also established inequalities for the number of $k$-faces of these polytopes.

\begin{lemma} For $k\in [0\ldots d-1]$, the following hold:
\begin{enumerate}[{\rm (i)}]
\item If $2\le a\le d$ and $1\le m\le \floor{a/2}-1$, then  $f_{k}(T_{m}^{d,d-a})\le f_{k}(T_{m+1}^{d,d-a})$, with strict inequality if and only if $m\le k$.
\item If $2\le a\le d-1$ and $1\le m\le \floor{a/2}$, then  $f_{k}(T_{m}^{d,d-a})\le f_{k}(T_{m}^{d,d-(a+1)})$, with strict inequality if and only if $a-m\le k$.
\end{enumerate}
\label[lemma]{lem:dplus2-vertices-inequalities}
\end{lemma}

The dual polytope $(T_{m}^{d,d-a})^{*}$ of $T^{d,d-a}_{m}$ is a $(d-a)$-fold pyramid over $T(m)\times T(a-m)$. The $d+2$ facets of $(T_{m}^{d,d-a})^{*}$ are described below.

\begin{remark}\label{rmk:dplus2facets-facets}
For $d\ge 2$, $2\le a\le d$, and $1\le m\le \floor{a/2}$,  the $d+2$ facets of the $(d-a)$-fold pyramid $P$ over $T(m)\times T(a-m)$ are:
\begin{enumerate}[{\rm (i)}]
\item $m+1$ facets that are  $(d-a)$-fold pyramids over $T(m-1)\times T(a-m)$, with $f_{0}(P)-(a-m+1)=(a-m)(m-1)+d$ vertices;

\item $a-m+1$ facets that are $(d-a)$-fold pyramids over $T(m)\times T(a-m-1)$, with $f_{0}(P)-(m+1)=m(a-m-1)+d$ vertices;

\item $d-a$ facets that are $(d-a-1)$-fold pyramids over $T(m)\times T(a-m)$, with $f_{0}(P)-1$ vertices.
\end{enumerate}
\end{remark}

\textbf{\cref{conj:3minus1}(ii) is settled in \cref{lem:lower-bound-dplus2-facets,lem:lower-bound-dplus2-facets-extra}}. The proof of \cref{lem:lower-bound-dplus2-facets} is embedded in the proof of \cite[Thm.~4.1]{Xue24} and  also appears as \cite[Lem.~7]{PinTriYos24}.
%, settles \cref{conj:3minus1}(ii) for up to $3d-3$ vertices. 
 
\begin{lemma}[{\cite[Lem.~7]{PinTriYos24}}] For $d\ge 4$, $\ell\in [1\ldots d-3]$, $a=\floor{(d+\ell+1)/2}+1$, and $1\le k\le d-2$, the number \begin{align*}
\hh_{k}(2d+\ell,d)&:=\binom{d+2}{k+2} -\binom{d-a+3}{k+2}-\binom{d-1}{k+2} +\binom{d-a+1}{k+2}\\
&\,=\binom{d+1}{k+1}+\binom{d}{k+1}+\binom{d-1}{k+1} -\binom{d-a+2}{k+1}-\binom{d-a+1}{k+1}
\end{align*}
of $k$-faces  of the $d$-polytope $(T_{2}^{d,d-a})^{*}$ is a lower bound on the number of $k$-faces of $d$-polytopes with $d+2$ facets and \textbf{at least} $2d+\ell$ vertices. 
\label[lemma]{lem:lower-bound-dplus2-facets}
\end{lemma}
The function $\tau_{k}(2d+\ell,d)$ in \cref{lem:lower-bound-dplus2-facets} is obtained by setting $a=\floor{(d+\ell+1)/2}+1$ and $m=2$ in the function $\rho_{k} (a,m,d)$ from \cref{lem:dplus2facets}. 

We now address the cases of $3d-2$ and $3d-1$ vertices in \cref{conj:3minus1}(ii). For $d=4,5$, since $2\le a\le d$ and $1\le m\le \floor{a/2}$, a $d$-polytope $(T_{m}^{d,d-a})^{*}$ cannot have more than $3d-3$ vertices. Additionally, for $d=6$, the polytope $(T_{m}^{d,d-a})^{*}$ cannot exceed  $3d-2$ vertices.
% \begin{remark}
% \label{rmk:dplus2-facets-d45}
% \end{remark}

\begin{lemma} Let $d\ge 6$ when $\alpha=0$ and  $d\ge 7$ when $\alpha=1$, and let $1\le k\le d-2$. If $a=d+2-\floor{(d-\alpha)/3}$, then the number \begin{align*}
\hh_{k}(3d-2+\alpha,d)&:=\binom{d+1}{k+1}+\binom{d}{k+1}+\binom{d-1}{k+1}+\binom{d-2}{k+1} -\binom{d-a+3}{k+1}\\
&\quad -\binom{d-a+2}{k+1}-\binom{d-a+1}{k+1}
\end{align*}
of $k$-faces  of the $d$-polytope $(T_{3}^{d,d-a})^{*}$ is a lower bound on the number of $k$-faces of $d$-polytopes with $d+2$ facets and \textbf{at least} $3d-2+\alpha$ vertices. 
\label[lemma]{lem:lower-bound-dplus2-facets-extra}
\end{lemma}
\begin{proof} The argument proceeds similarly to the proof of \cite[Lem.~7]{PinTriYos24}. We provide a sketch here. 
  
  Let $P:=(T_{3}^{d,d-a})^{*}$. Then $f_{0}(P)=4d-2-3\floor{(d-\alpha)/3}$ vertices. Consider a $d$-polytope  $Q$, distinct from $P$, with $d+2$ facets and at least $3d-2+\alpha$ vertices. By \cref{lem:dplus2facets}, $Q$ has the form $Q=(T_{m}^{d,d-b})^{*}$,  where $2\le b\le d$ and $1\le m\le \floor{b/2}$. Since $f_0(Q)=d+1+m(b-m)\ge 3d-2+\alpha$ and $b\le d$, we must have $m\ge 3$.
  
  Suppose  $f_{0}(Q)\ge f_{0}(P)$. We aim to show that $f_{k}(P)\le f_{k}(Q)$ for $1\le k\le d-2$.  We work in the dual setting, proving that $f_{k}(P^{*})\le f_{k}(Q^{*})$ for $1\le k\le d-2$. From $f_{d-1}(Q^{*})\ge f_{d-1}(P^{*})$, we obtain
\begin{equation*}
m(b-m)\ge 3(a-3).
%\label{eq:lower-bound-dplus2-facets-1}
\end{equation*}
 Suppose $m=3$. Then $b\ge a$, and \cref{lem:dplus2-vertices-inequalities}(ii) implies  $f_{k}(P^{*})\le f_{k}(Q^{*})$ for $1\le k\le d-2$, as desired. Thus, we assume  $m\ge 4$, so $b\ge 8$.

If $b\ge a$, then again \cref{lem:dplus2-vertices-inequalities}(ii)  gives $f_{k}(T_{m}^{d,d-a})\le f_{k}(Q^{*})$ for $0\le k\le d-2$. Additionally, \cref{lem:dplus2-vertices-inequalities}(i) implies $f_{k}(T_{m}^{d,d-a})\ge f_{k}(P^{*})$ for $1\le k\le d-2$, and this case is resolved. We may therefore assume that $a> b$.

In this case, we compute
\begin{equation*}
\label{eq:lower-bound-dplus2-facets-1a}
\begin{aligned}
f_{k}(Q^{*})-f_{k}(P^{*})&=\underbrace{\brac*{-\binom{d-b+m+1}{d-k+1}+\binom{d-a+4}{d-k+1}}}_{=A}+\underbrace{\brac*{-\binom{d-m+1}{d-k+1}+\binom{d-2}{d-k+1}}}_{=B} \\
&+\underbrace{\brac*{\binom{d-b+1}{d-k+1}-\binom{d-a+1}{d-k+1}}}_{=C}.
\end{aligned}
\end{equation*}
Continuing as in the proof of \cite[Lem.~7]{PinTriYos24}, we find that $A+B+C\ge 0$. 
\end{proof}

The function $\tau_{k}(3d-2+\alpha,d)$ in \cref{lem:lower-bound-dplus2-facets-extra} is obtained by setting $a=d+2-\floor{(d-\alpha)/3}$ and $m=3$ in the function $\rho_{k} (a,m,d)$ from \cref{lem:dplus2facets}. 
		
%\begin{align*}
%\hh_{k}(3d-1,d)&:=\binom{d+1}{k+1}+\binom{d}{k+1}+\binom{d-1}{k+1} -\binom{d-a+2}{k+1}-\binom{d-a+1}{k+1}
%\end{align*}
	
\begin{remark} Comparing the functions from \cref{lem:lower-bound-dplus2-facets} and \cref{lem:lower-bound-dplus2-facets-extra}, we find   that,  for $d\ge 4$ and each $1\le k\le d-2$, 

\begin{equation*}
\tau_{k}(3d-3,d)\le\tau_{k}(3d-2,d)\le \tau_{k}(3d-1,d).
\end{equation*}
The first inequality is settled in \cref{lem:combinatorial-ineq}(i), while the second inequality  follows from the definitions of $\tau_{k}(3d-2,d)$ and $\tau_{k}(3d-1,d)$.
\label{rmk:functions-dplus2-facets}
\end{remark}

We conclude this section by extending the function $\tau_{k}(2d+\ell,d)$ beyond $\ell=d-1$:
\begin{equation*}
\label{eq:3minus1-tau-function-extension}
\tau_{k}(2d+\ell,d)=\tau_{k}(3d-1,d),\;\text{for $\ell\ge d$}.
\end{equation*} 

Henceforth, we concentrate on $d$-polytopes with at least $d+3$ facets.

\section{Additional minimisers with $d+3$ facets}
\label{sec:additional-minimisers} 

For the case of $2d+2$ vertices (that is, $\ell=2$), truncating a nonsimple vertex of  a $(2,d-2)$-triplex results in a $d$-polytope $A(d)$~\cite{PinUgoYos16a}, which can also be realised as a prism over a copy of $(2,d-3)$-triplex. Thus  $A(d)$  has  the same $f$-vector as $J(3,d)$, the polytope obtained by truncating  a simple vertex in a $(3,d-3)$-triplex. Recall that the \emph{$f$-vector} of a $d$-polytope $P$, denoted  $f(P)$,  is the sequence
$(f_0, f_1, \dots, f_{d-1})$.

For the case of $3d-2$ vertices (that is, $\ell=d-2$), we present two additional minimisers. First we define the \emph{vertex figure} of a polytope $P$ at a vertex $v$  as the polytope $P/v:=H\cap P$, where $H$ is a hyperplane that separates $v$ from the other vertices of $P$. There is a bijection between the $k$-faces of $P$ that contain $v$ and the $(k-1)$-faces of $P/v$.

The paper \cite{PinUgoYos16a} introduces the $d$-polytope $\Sigma(d)$ as the convex hull of \[\set*{0,e_1,e_1+e_k,e_2,e_2+e_k,e_1+e_2,e_1+e_2+2e_k: 3\le k\le d},\]
where $\{e_i\}$ is the standard basis of $\mathbb{R}^d$.    
The polytope $\Sigma(d)$ has $d+3$ facets, including $d+1$ facets that contain its unique nonsimple vertex $v_1$, whose vertex figure is a simplicial $(d-1)$-prism, and two simplicial prisms $F_1$ and  $F_2$ that do not contain $v_1$ and intersect along a simplex ridge. Thus, for $d\ge 3$ and $0\le k \le d-2$, the number of $k$-faces of $\Sigma(d)$ is 
\begin{equation}
\label{eq:Sigma_d}
    \begin{aligned}
    f_k(\Sigma(d))&=f_k(F_1)+ f_k(F_2) - f_k(F_1 \cap F_2) + f_{k-1}(\Sigma(d)/v_1) \\
    & = 2\brac*{\binom{d}{k+1}+\binom{d-1}{k+1}-\binom{1}{k+1}}- \binom{d-1}{k+1}\\
    & +\brac*{\binom{d}{k}+\binom{d-1}{k}-\binom{1}{k}} \\
    & = \binom{d+1}{k+1}+2\binom{d}{k+1}-\binom{2}{k+1}-\binom{1}{k+1}.
    \end{aligned}
\end{equation}
Hence, $\Sigma(d)$ has the same $f$-vector as $J(d-1,d)$, the polytope obtained by truncating  a simple vertex in a $(d-1,1)$-triplex. 

% \begin{remark}[Facets of $\Sigma(d)$]\label{rmk:Sigmad-Facets} The $d+3$ facets of  $\Sigma(d)$ are as follows.
% \begin{enumerate}[{\rm (i)}]
% \item $d-1$ copies of $\Sigma({d-1})$,
% \item two simplicial prisms,
%  \item two simplices.
% \end{enumerate}
% \end{remark}

We now introduce the wedge construction. Let $P$ be a $d$-polytope embedded in the hyperplane $x_{d+1}=0$ of $\R^{d+1}$, and let $F$ be a proper face of $P$. Consider the halfcylinder $C:=P\times [0,\infty)\subset \R^{d+1}$. Let $H'$ be a hyperplane passing through $F \times \{0\}$ that cuts $C$ into a bounded and an unbounded region. The \emph{wedge} of $P$ at $F$, denoted $\wed_{F}(P)$, is defined as the bounded part. The sets $P$ and $H'\cap C$, the \emph{bases} of $\wed_{F} (P)$, define facets of $\wed_{F} (P)$ that are combinatorially isomorphic to $P$ and intersect at the face $F\times \set{0}$. The wedge $W$  over a $d$-polytope $P\times \set{0}\subseteq \R^{d+1}$ at a face $F\times \set{0}$ of $P\times \set{0}$ is combinatorially isomorphic to a prism $Q$ over $P\times \set{0}$ in which the face prism $(F\times \set{0})$ of $Q$ has collapsed to $F\times \set{0}$. For further information, refer to \cite[Sec.~2.6]{Pin24}. We will use the following basic facts: 

\begin{lemma}\label[lemma]{wedge}
    Let $\wed_{F}(P)$ be a wedge of a $d$-polytope $P$ at a face $F$ of $P$.
    \begin{enumerate}[{\rm (i)}]
        \item A $k$-face of $\wed_{F}(P)$ is  a $k$-face of one of the bases of $\wed_{F}(P)$,
        or the wedge of a $(k-1)$-face $F'$ of P at the proper face $F\cap F'$, or a prism over a $(k-1)$-face of $P$ disjoint from $F$.
        \item If $F$ is a facet of $P$, then for each value of $k\in [0\ldots d+1]$ we have
        $$f_k(\wed_{F}(P))=2f_k(P)+f_{k-1}(P)-f_k(F)-f_{k-1}(F).$$
    \end{enumerate}
\end{lemma}

Let $C(d)$ denote the polytope obtained by truncating one simple edge of a  $(2,d-2)$-triplex. It has $3d-2$ vertices and $d+3$ facets. 

% \begin{remark}[Facets of $C(d)$]\label{rmk:Cd-Facets} The $d+3$ facets of $C(d)$ are as follows.
%     \begin{enumerate}[{\rm (i)}]
% \item $d-2$ copies of $C({d-1})$,
% \item three simplicial prisms,
%  \item one copy of $T(2)\times T(d-3)$, and
%  \item one simplex.
%     \end{enumerate} 
% \end{remark}

The polytopes $C(2)$ and $\Sigma(2)$ are both quadrilaterals, and similarly $C(3)$ and $\Sigma(3)$ are the same polytope. However, for $d\ge 4$, the polytopes $C(d)$ and $\Sigma(d)$ are distinct. Furthermore, for $d\ge 4$, each of $C(d)$ and $\Sigma(d)$ is a wedge of  a $(d-1)$-dimensional version of the polytope at a face that is $(d-2)$-dimensional version of the polytope: $C(d)=W_{C(d-2)}(C(d-1))$ and $\Sigma(d)=W_{\Sigma(d-2)}(\Sigma(d-1))$. See \cref{wedge}. Hence, both $C(d)$ and $\Sigma(d)$  share the same $f$-vector. 

\begin{figure}             
\begin{center}     
\includegraphics[scale=.9]{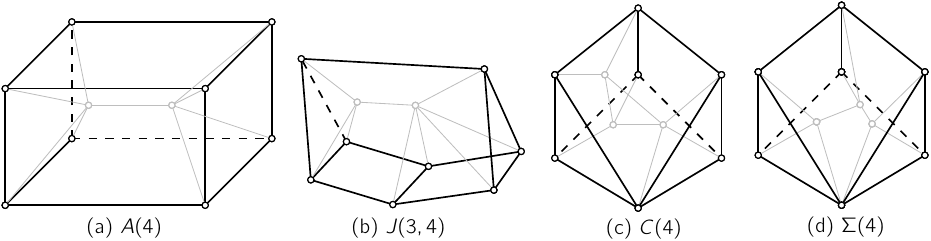}  
\end{center}
\caption{Schlegel diagrams of polytopes. {(a)}  The polytope $A(4)$. {(b)}  The polytope $J(3,4)$. (c) The polytope $C(4)$. (d) The polytope $\Sigma(4)$.}
\label{fig:additional-minimisers}   
\end{figure}

\Cref{fig:additional-minimisers} depicts 4-dimensional versions of these additional minimisers.

\subsection{Multifold pyramids with at least $d+3$ facets whose base is a simple polytope}

 We address the case of pyramids over simple polytopes in \cref{conj:3minus1}(i). Our approach  relies on the lower bound theorem for simple polytopes~\cite{Bar71,Bar73} (\cref{thm:simple-lbt}) and on the notion of \emph{truncation polytopes}---polytopes obtained from simplices by successive truncation of vertices. %It is clear that

% \begin{remark}  The smallest possible vertex counts of truncation $d$-polytopes are $d + 1$, $2d$, $3d-1$, and $4d-2$.
% \label{rmk:truncation-polytopes-vertices}
% \end{remark}

\begin{theorem}[Simple polytopes, Barnette (1971--73)]
\label{thm:simple-lbt} 
Let $d\ge 2$ and let $P$ be a simple $d$-polytope with $f_{d-1}$ facets. Then
\begin{equation*}
%\label{eq:simple}
 f_k(P)\ge\begin{cases}
(d-1)f_{d-1}-(d+1)(d-2),& \text{if $k=0$}; \\
\binom{d}{k+1}f_{d-1}-\binom{d+1}{k+1}(d-1-k),& \text{if $k\in [1\ldots d-2]$}.
\end{cases}
\end{equation*}
If, for $d\ge 4$,  $f_{k}(P)$ achieves equality for some $k\in [0\ldots d-2]$, then $P$ must be a truncation polytope. For $d=2,3$, equality holds for every simple $d$-polytope.
\end{theorem}

 The next lemma extends \cite[Lem.~17]{PinTriYos24}. Before proving the lemma, we extend the function $\eta_{k}(2d+\ell,d)$ beyond $\ell=d-1$:
\begin{equation}
\label{eq:3minus1-function-extension}
\eta_{k}(2d+\ell,d)=\eta_{k}(3d-1,d),\;\text{for $\ell\ge d$}.
\end{equation}

 \begin{lemma} Let $d\ge 3$, $0\le t\le d-2$,  and $\ell\ge 1$. Additionally, let $P$ be a $t$-fold pyramid over a simple $(d-t)$-polytope  with  $2d+\ell$ vertices and at least  $d+3$ facets. Then, for  each $k\in [1\ldots d-2]$, the following hold:  
 \begin{enumerate}[{\rm (i)}]
 \item If $t=0$, then $f_{0}(P)\ge 3d-1$ and $f_{k}(P)\ge \eta_{k}(3d-1,d)$. Equality  holds for some $k\in [1\ldots d-2]$ if and only if $P$ is either a cube or the simple $d$-polytope $J(d,d)$ ($d\ge 3$).  
 %\green{If $t=0$, then $f_{0}(P)\ge 3d-1$ and $f_{k}(P)\ge \eta_{k}(3d-1,d)$. Equality holds for $d = 3$ and $k = 1$ when $P$ is a simple 3-polytope with eight vertices and six facets,  and for $d\ge 4$ and some $k\in [1\ldots d-2]$ if only if $P$ is a truncation $d$-polytope with $3d-1$ vertices}.
 %\item $f_{k}(P)> \eta_{k}(3d-1,d)$ when $t=0$ and $f_{0}(P)>3d-1$;
  \item If $t\ge 1$, then $f_{k}(P)\ge \eta_{k}(2d+\ell,d)$ for $k\in [1\ldots d-2]$. Equality    holds for some $k\in [1\ldots d-2]$ if and only if $k\ge d-\ell$ and  $P$ is a $t$-fold pyramid over $J(d-t, d-t)$ with $t$ satisfying $f_0(P)=3d-2t-1\ge 2d+\ell$.
 \end{enumerate}
% satisfies $f_{0}(P)\ge 3d-1$ and $f_{k}(P)\ge \eta_k(3d-1,d)$ .
 \label[lemma]{lem:2d+s-dplus3-simple-pyramid}
 Also, equality in either case holds for some $k\in [1\ldots d-2]$ only if $P$ has exactly $d+3$ facets.
 \end{lemma}
\begin{proof} 
Part (i) was shown in \cite[Lem.~17]{PinTriYos24}.
%ing only when $P$ has exactly $d+3$ facets. If $d=3$, then by \cref{thm:simple-lbt}, we find that $f_0(P)\ge 3d-1$ and $f_{1}(P)\ge 12=\eta_{1}(8,3)$, with equality only if $P$ has exactly $d+3$ facets, \blue{i.e. $P$ is either a cube or $J(3,3)$. }

(ii) The cases $\ell=1$ and $\ell=2$ were established in \cite[Lem.~18]{PinYos22} and \cite[Lem.~17]{PinTriYos24}, respectively,  for all $d\ge 3$ and all $k\in [1\ldots d-2]$. So assume $\ell\ge 3$.

First consider $d=3$. If $\ell\ge 3$, then $f_{0}(P)\ge 2\times 3+3=9$. If $t=2$ or $3$, then $P$ would be a 3-simplex, which has fewer than 8 vertices, contradicting the assumption. If $t=1$, then $P$ is a pyramid over an $n$-gon with $n\ge 8$, in which case $f_{1}(P)=2n>12=\eta_{1}(8,3)$ and $f_{2}(P)=n+1>d+3$. Thus,  the case $d=3$ is settled, and we proceed by induction on $d$  for all $t\ge 0$ and all $\ell\ge 1$.

Assume now $d\ge 4$. Since $t\ge 1$, $P$ is a pyramid over a $(d-1)$-polytope $F$, which itself is a $(t-1)$-fold pyramid over a simple polytope. Moreover, $F$ has $2(d-1)+\ell+1$ vertices and at least $(d-1)+3$ facets, so the induction hypothesis applies to $F$. 
%Consequently 
%\begin{equation}
%\label{eq:2d+s-dplus3-simple-pyramid-2}
%f_{k}(P)=f_{k}(F)+f_{k-1}(F).
%\end{equation}
We also observe that $f_{d-2}(F)\ge d+2=\eta_{d-2}(2(d-1)+\ell+1,d-1)$.

The following result extends the claim in \cite[Lem.~17]{PinTriYos24}; the proof is essentially unchanged. 

\begin{claim} If $Q$ is a $t$-fold pyramid over  $J(d-t,d-t)$, then 
$f_k(Q)=
    \binom{d+1}{k+1}+2\binom{d}{k+1}-2\binom{t+1}{k+1}$.
Furthermore, if  $f_0(Q)=3d-2t-1\ge 2d+\ell$, then $f_{k}(Q)=\eta_{k}(2d+\ell,d)$ for all $k\ge d-\ell$, while $f_{k}(Q)>\eta_{k}(2d+\ell,d)$ for  $1\le k< d-\ell$.
    \label{cl:faces-multifold-pyramid}
\end{claim}

\begin{claimproof} The face numbers of $Q$ follow directly from the multifold-pyramid formula; see, for instance, \cite[Thm.~4.2.2]{Gru03}. Hence, if  $k>t$, then $f_{k}(Q)=\eta_{k}(2d+\ell,d)$ for each $k\ge d-\ell$; moreover,   $f_{k}(Q)>\eta_{k}(2d+\ell,d)$ for each $ k+1\le d-\ell$. Suppose $1\le k\le t$. Since $2t+1\le d-\ell$, we must have $1\le k< d-\ell-1$, in which case
 \begin{equation*}
     f_k(Q)-\eta_k(2d+\ell,d)=\binom{d-\ell}{k+1}-2\binom{t+1}{k+1}\ge \binom{2t+1}{k+1}-2\binom{t+1}{k+1}>0.
 \end{equation*}
Applying the Vandermonde's identity to $\binom{2t+1}{k+1}$ (\cref{lem:combinatorial-identities}(v)) gives the last inequality.
\end{claimproof}

\textbf{Suppose $t\ge 2$.} By the induction hypothesis (for $\ell\in [1\ldots d-2]$), we get 
\begin{equation*}
\label{eq:2d+s-dplus3-simple-pyramid-3}
f_{k}(F)\ge\eta_{k}(2(d-1)+\ell+1,d-1)=\binom{d}{k+1}+2\binom{d-1}{k+1}-\binom{d-\ell-2}{k+1},\; \text{$k\in [1\ldots d-3]$},
\end{equation*}
where equality  holds for some  $k\in [1\ldots d-3]$ if and only if $k\ge (d-1)-(\ell+1)=d-2-\ell$,  $F$ is a $(t-1)$-fold pyramid over $J(d-t, d-t)$, and $f_0(F)=3(d-1)-2(t-1)-1\ge 2(d-1)+\ell+1$ (this condition implies that $\ell \le d-5$). 

Since $f_{0}(F)=f_{0}(P)-1$, 
\begin{equation*}
  \begin{aligned}
    f_{1}(P)&=f_{1}(F)+f_{0}(F)\\
      &\ge \eta_{1}(2(d-1)+\ell+1,d-1)+2(d-1)+\ell+1 =\eta_{1}(2d+\ell+1,d)-1 \; \ge \eta_{1}(2d+\ell,d).
  \end{aligned}
\end{equation*}
We have equality in $\eta_{1}(2d+\ell+1,d)-1\ge \eta_{1}(2d+\ell,d)$  if and only if $\ell=d-2$. However, by the induction hypothesis, if we have equality in $f_{1}(F)\ge \eta_{1}(2(d-1)+\ell+1,d-1)$, then $\ell\le d-5$. Consequently, $f_1(P)> \eta_{1}(2d+\ell,d)$.  
% in this case, $2f_1(F)<(d-1)f_0(F)$:
% \begin{align*}
%  2f_1(F)=2\eta_1(3d-4, d-1)=3d^2-7d+4<3d^2-6d+3=(d-1)f_0(F),
% \end{align*}
% which is a contradiction.

Additionally, for $k\in [2\ldots d-2]$, we have
%\blue{(The case $k=1$ can be combined into $k\in [2\ldots d-2]$ if we add one extra term $-\binom{1}{k+1}$ in the expression of $\eta_k$)}
\begin{align*}
f_{k}(P)=f_{k}(F)+f_{k-1}(F)&\ge \eta_{k}(2(d-1)+\ell+1,d-1)+\eta_{k-1}(2(d-1)+\ell+1,d-1)\\
            &=\eta_{k}(2d+\ell+1,d)\ge \eta_{k}(2d+\ell,d).
\end{align*}
For $2\le k\le d-2-\ell$, $f_{k-1}(F)>\eta_{k-1}(2(d-1)+\ell+1,d-1)$ by the induction hypothesis, yielding $f_k(P)>\eta_{k}(2d+\ell,d)$. For $k= d-1-\ell$, if $f_{k-1}(F)=\eta_{k-1}(2(d-1)+\ell+1,d-1)$ and $f_{k}(F)=\eta_{k}(2(d-1)+\ell+1,d-1)$, then $f_k(P)=\eta_{k}(2d+\ell+1,d)$ and $\ell\le d-5$. However, in this scenario,  $f_k(P)=\eta_{k}(2d+\ell+1,d) > \eta_{k}(2d+\ell,d)$. If $k\ge d-\ell$, then, also by the induction hypothesis, $f_k(F)=\eta_k(2(d-1)+\ell+1,d-1)$ and $f_{k-1}(F)=\eta_{k-1}(2(d-1)+\ell+1,d-1)$  if and only if $F$ is a $(t-1)$-fold pyramid over $J(d-t, d-t)$ and $f_0(F)\ge 2(d-1)+\ell+1$. Furthermore, if $k\ge d-\ell$, then $\eta_{k}(2d+\ell+1,d)=\eta_{k}(2d+\ell,d)$. Consequently,  $f_{k}(P)=\eta_{k}(2d+\ell,d)$ only when $k\ge d-\ell$ and $P$ is a $t$-fold pyramid over $J(d-t,d-t)$.

%Equality for \blue{some $k\ge d-\ell-2$} is achieved  when $, and by the induction hypothesis this occurs  satisfying....}
%$F$ has precisely $d+2$ facets, concluding that $P$ has $d+3$ facets. 

Suppose that $\ell\ge d-1$, we show that $f_k(P) > \eta_k(2d+\ell, d)$ for all $k\in [1 \ldots d-2]$.  Here $f_{0}(P)\ge 3d-1$ and $f_{0}(F)\ge 3d-2>3(d-1)-1$. By the induction hypothesis, $f_{k}(F)\ge \eta_{k}(2(d-1)+d-2,d-1)$ for $k\in [1\ldots d-3]$, where equality  holds for some  $k\in [1\ldots d-3]$ if and only if $k\ge (d-1)-(\ell+1)=d-2-\ell$,  $F$ is a $(t-1)$-fold pyramid over $J(d-t, d-t)$, and $f_0(F)=3(d-1)-2(t-1)-1\ge 2(d-1)+\ell+1$ (again this condition implies that $\ell \le d-5$). Hence, for $\ell\ge d-1$, 
\begin{equation*}
    f_{k}(F)> \eta_{k}(2(d-1)+d-2,d-1),\; \text{for $k\in [1\ldots d-3]$}.
\end{equation*}
Since $f_0(F)=f_0(P)-1$, 
%\blue{(Similarly, the case $k=1$ can be combined into $k\in [2\ldots d-2]$ if we add one extra term $-\binom{1}{k+1}$ in the expression of $\eta_k$)}
\begin{align*}
f_{1}(P)&=f_{1}(F)+f_{0}(F)>\eta_{1}(3(d-1)-1,d-1)+3d-2=\eta_{1}(3d-1,d).
%&=\eta_{1}(2d+d-1-1,d)+2=\eta_{1}(2d+d-1,d)+1\\
\end{align*}
Additionally, for $k\in [2\ldots d-2]$,  \begin{align*}
f_{k}(P)&=f_{k}(F)+f_{k-1}(F)\\
&> \eta_{k}(3(d-1)-1,d-1)+\eta_{k-1}(3(d-1)-1,d-1)=\eta_{k}(3d-2,d)=\eta_{k}(3d-1,d).
\end{align*}

% Equality for $k=1$ holds if and only if $\ell=d-1$, $f_0(F)=3d-2$, and $f_1(F)=\eta_1(3d-4, d-1)$. However, $2f_1(F)<(d-1)f_0(F)$:
% \begin{align*}
%  2f_1(F)=2\eta_1(3d-4, d-1)=3d^2-7d+4<3d^2-5d+2=(d-1)f_0(F),
% \end{align*}
% which is a contradiction.

 % The equality $f_k(P)=\eta_k(3d-2, d)=\eta_k(3d-1, d)$ implies $f_k(F)=\eta_{k}(3d-4,d-1)$ for some $k\in [2\ldots d-2]$. By induction hypothesis, equality for some $k\in [2\ldots d-2]$ implies that \blue{$F$ is a $(t-1)$-fold pyramid over $J(d-t, d-t)$, which means $P$ has at most $3d-5$ vertices, contradicting to our assumption that $\ell \ge d-1$.}
 %has exactly $d+2$ $(d-2)$-faces, and so  $P$ has exactly $d+3$ facets. 
\textbf{Finally, assume that $t=1$}. Then $F$ is a simple $(d-1)$-polytope, in which case Part (i) implies  $f_{0}(F)\ge 3(d-1)-1$, so $f_{0}(P)\ge 3d-3$.  For $f_{0}(P)\ge 3d-2$, it follows that $f_{0}(F)\ge 3d-3$. By Part (i),  $f_{k}(F)> \eta_{k}(3(d-1)-1,d-1)$ for $k\in [2\ldots d-2]$. Thus 
\begin{align*}
f_{k}(P)&=f_{k}(F)+f_{k-1}(F)\\
&> \eta_{k}(3(d-1)-1,d-1)+\eta_{k-1}(3(d-1)-1,d-1) = \eta_{k}(3d-2,d)= \eta_{k}(3d-1,d).
\end{align*}
% since, by Part (i), $f_k(F)=\eta_k(3(d-1)-1,d-1)$ for $k\in [1\ldots d-3]$ only if $F$ has precisely $3d-4$ vertices, which is not the case. 
For $f_{0}(P)\ge 3d-2$ and $k=1$, as  $F$ is a simple $(d-1)$-polytope, 
%\blue{(It seems that the case $k=1$ can be combined as well into $k\in [2\ldots d-2]$ if we add one extra term $-\binom{1}{k+1}$ in the expression of $\eta_k$)}
\begin{align*}
   f_1(P) & \ge \frac{(3d-3)(d-1)}{2}+3d-3  = \eta_1(3d-1,d)+\frac{d-1}{2}-1  > \eta_1(3d-1,d).
\end{align*}

Suppose that $f_0(P)=3d-3$ (that is, $\ell=d-3$). By the induction hypothesis, 
\begin{equation*}
f_{k}(F)\ge\eta_{k}(3(d-1)-1,d-1),\; \text{$k\in [1\ldots d-3]$},
\end{equation*}
where equality  holds for some  $k\in [1\ldots d-3]$ if and only if $k\ge d-2-\ell=1$ and $F$ is  $J(d-1, d-1)$. We  start with the cases $k=1,2$, \begin{align*}
f_{1}(P)&=f_{1}(F)+f_{0}(F)\ge \eta_{1}(3(d-1)-1,d-1)+3(d-1)-1=\eta_{1}(3d-3,d)+1,
\end{align*}
and
\begin{align*}
f_{2}(P)&=f_{2}(F)+f_{1}(F)\ge \eta_{2}(3(d-1)-1,d-1)+\eta_{1}(3(d-1)-1,d-1)=\eta_{2}(3d-3,d)+1.
\end{align*}
For $k\in [3\ldots d-2]$, 
\begin{align*}
f_{k}(P)&=f_{k}(F)+f_{k-1}(F)\\
&\ge \eta_{k}(3(d-1)-1,d-1)+\eta_{k-1}(3(d-1)-1,d-1)= \eta_{k}(3d-3,d)
\end{align*}	
and equality  in $f_{k}(P)\ge  \eta_{k}(3d-3,d)$ holds for some  $k\in [1\ldots d-2]$ if and only if $k\ge 3=d-\ell$ and $P$ is a pyramid over $J(d-1, d-1)$ satisfying $f_0(P)\ge 3d-3$, as required. This completes the proof.
%$F$ has precisely $d+2$ facets, concluding that  $P$ has $d+3$ facets.   For $k=1,2$, 

\end{proof}
% \blue{with equality holds if and only if $F$ has exactly $d+2$ facets, in which case $P$ has exactly $d+3$ facets. }
% %If $f_0(F)\ge 3d-3$, (since $F$ is simple) we have 
% \begin{align*}
%    f_1(P) & \ge \frac{(3d-3)(d-1)}{2}+3d-3 \\
%           & = \eta_1(3d-1,d)+\frac{d-1}{2}-1 \\
%           & > \eta_1(3d-1,d).
% \end{align*}
% }

%So if $f_0(F)\ge 3d-3$, we have $f_{1}(P)\ge \eta_{1}(3d-3,d)+3=\eta_{1}(3d-1,d)$, with equality holds only if $F$ has $3d-3$ vertices and with $\eta_{1}(3d-3,d)+1$ edges........}

%\begin{proof} For $k\ge 0$, each $(k-1)$-face $F_{0}$ in $F$  is contained in a  $k$-face $F_{1}$ not in $F$ such that $F_{0}=F_{1}\cap F$. The result now follows.
%\end{proof}
	
\section{Existing lower bounds  for polytopes with few vertices and auxiliary results}
\label{sec:existing-theorems}
  We use the following result by Xue~\cite[Prop.~3.1]{Xue21}; see also \cite[Prop.~12]{PinYos22}.

\begin{proposition}[Xue 2021]\label[proposition]{prop:number-faces-outside-facet}
Let $d\ge 2$ and let $P$ be a $d$-polytope. Suppose  $r\le d+1$ and let $S:=(v_1,v_2,\ldots,v_r)$ be a sequence of distinct vertices in $P$. Then 
\begin{enumerate}[{\rm (i)}]
    \item  There is a sequence $F_1, F_2,\ldots, F_r$ of faces of $P$ such that each $F_i$ has dimension $d-i+1$ and contains $v_i$, but does not contain any $v_j$ with $j<i$.
    \item For each $k\ge1$, the number of $k$-faces of $P$ that contain at least one of the vertices in $S$ is bounded from below by \[\sum_{i=1}^{r}f_{k-1}(F_i/v_i).\]
\end{enumerate}
\end{proposition}

	Xue's lower bound theorem for $d$-polytopes with at most $2d$ vertices reads as follows:
  
\begin{theorem}[$d$-polytopes with at most $2d$ vertices,  {\cite{Xue21}}]\label{thm:at-most-2d}  
Let $d\ge 2$ and $1\le s\le d$. If $P$ is a  $d$-polytope with  $d+s$ vertices, then
\[f_k(P)\ge \ii_k(d+s,d),\;  \text{for all $k\in[1\ldots d-1]$}.\] 
Moreover, if $f_k(P)=\ii_k(d+s,d)$ for some $k\in[1\ldots d-2]$, then $P$ is the $(s,d-s)$-triplex.
\end{theorem}
Recall that the $(s,d-s)$-triplex is the $(d-s)$-fold pyramid over $T(1)\times T(s-1)$.

We proceed with a lower bound theorem for $d$-polytopes with $2d+1$ vertices, independently proved  by Pineda-Villavicencio and Yost~\cite{PinYos22} and by Xue~\cite{Xue24}, as well as a lower bound theorem by Pineda-Villavicencio et al.~\cite{PinTriYos24} for $d$-polytopes with $2d+2$ vertices.

Recall that $J(2,d)$ and $A(d)$ are, respectively, the polytopes obtained by truncating  a simple vertex and a nonsimple vertex in a $(2,d-2)$-triplex. The polytope $\Sigma(d)$ is the convex hull of \[\set*{0,e_1,e_1+e_k,e_2,e_2+e_k,e_1+e_2,e_1+e_2+2e_k: 3\le k\le d},\] and $(T_{2}^{d,d-(\floor{d/2}+2)})^*$ is the $(d-(\floor{d/2}+2))$-fold pyramid over $T(2)\times T(\floor{d/2})$.

\begin{theorem}[Pineda-Villavicencio and Yost 2022; Xue 2024]
\label{thm:2dplus1}
Let $d\ge 3$, let $P$ be a $d$-polytope with \textbf{at least} $2d+1$ vertices, and let $k\in [1\ldots d-2]$. 
\begin{enumerate}[{\rm (i)}]
\item For $d=3$, if $P$ is $\Sigma(3)$ or $J(2,3)$, then $f(P)=(7,11,6)$. Otherwise, $f_1(P)>11$ and $f_2(P)>6$.
\item  If $d\ge 4$ and $P$ has at least $d+3$ facets, then $f_{k}(P)\ge f_{k}(J(2,d))$, with equality for some $k\in [1\ldots d-2]$ only if  $P=J(2,d)$.
\item  If $d\ge 4$ and $P$ has  $d+2$ facets, then $f_{k}(P)\ge f_{k}((T_{2}^{d,d-(\floor{d/2}+2)})^{*})$. Moreover, if $d$ is even  then $(T_{2}^{d,d-(\floor{d/2}+2)})^{*}$ has exactly $2d+1$ vertices.
\end{enumerate}
\end{theorem}

As observed in \cite{PinTriYos24}, a variant of \cref{thm:at-most-2d} extends to $d$-polytopes with at least $d+s$ vertices; it  follows as a corollary of \cref{thm:at-most-2d,thm:2dplus1}. %\blue{This result will be used for the equality analysis.}
%See also \cref{thm:2dplus1}.

\begin{corollary}[{\cite[Cor.~11]{PinTriYos24}}] For each $d\ge2$, each $k\in [1\ldots d-1]$, and each $s\in[1\ldots d]$, if $P$ is a $d$-polytope with \textbf{at least} $d+s$ vertices, then $f_{k}(P)\ge \theta_{k}(d+s,d)$.

Moreover, if $f_k(P)=\ii_k(d+s,d)$ for some $k\in[1\ldots d-2]$, then $P$ has $d+2$ facets.
\label{cor:more-dpluss}
\end{corollary}

We now state a corollary of \cref{prop:number-faces-outside-facet} and \cref{cor:more-dpluss}.

\begin{corollary}\label{cor:number-faces-outside-facet-practical}
Let $d\ge 2$ and let $P$ be a $d$-polytope. Suppose   $r\le d+1$  and  let $S:=(v_1,v_2,\ldots,v_r)$ be a sequence of distinct vertices in $P$. Then
\begin{enumerate}[{\rm (i)}]
    \item  There is a sequence $F_1, F_2,\ldots, F_r$ of faces of $P$ such that each $F_i$ has dimension $d-i+1$ and contains $v_i$, but does not contain any $v_j$ with $j<i$.
    \item For each $1\le i\le r$, let $1\le s_{i} \le d-i+1$, and suppose that $\deg_{F_{i}}(v_i)\ge d-i+1+s_{i}$. 
   Then, for each $k\ge1$, the number of $k$-faces of $P$ containing at least one vertex in $S$ is bounded below by \[\sum_{i=1}^{r}f_{k-1}(F_i/v_i)\ge \sum_{i=1}^{r}\theta_{k-1}(d-i+1+s_{i},d-i).\]
\end{enumerate}
\end{corollary}

% \begin{corollary}\label{cor:theta}
%     \blue{Let $P$ be a $d$-polytope with at least $d+s$ vertices, $s\ge 1$. Then $f_k(P)=\theta_k(d+s,d)$ if and only if $P$ is the $(s,d-s)$-triplex. }
%     \begin{proof}
%       \blue{By\cref{thm:2dplus1}, the number of $k$-faces of $d$-polytopes with at least $2d+1$ vertices is at least $f_k(J(2,d))$ or $f_k((T_{2}^{d,d-(\floor{d/2}+2)})^{*})$, which are both strictly larger than $\theta_k(2d,d)$. By \cref{thm:at-most-2d}, the statement follows. }
%     \end{proof}
% \end{corollary}

\begin{theorem}[Pineda-Villavicencio et al. 2024] Let $d\ge 3$, let $P$ be a $d$-polytope with $2d+2$ vertices, and let $k\in [1\ldots d-2]$. 
\begin{enumerate}[{\rm (i)}]
\item If $P$ has at least $d+3$ facets, then $f_{k}(P)\ge f_{k}(J(3,d))=f_{k}(A(d))=\eta_{k}(2d+2,d)$.

\item If $P$ has $d+2$ facets, then  $f_{k}(P)\ge \hh_{k}(2d+2,d)$. Furthermore, equality is tight for every $k\in [1\ldots d-2]$ when $d$ is odd, $a=\floor{(d+3)/2}+1$, and $P$ is a $(d-a)$-fold pyramid over $T(2)\times T(a-2)$. 
\end{enumerate}
\label{thm:2dplus2}
\end{theorem}

% We remark that, when $d\ge 9$, another dichotomy manifests:   

% \begin{enumerate}[{\rm (i)}] 
% \item If $1\le k\le \ceil{d/3}-2$,  
% $\eta_{k}(2d+2,d)<\tau_{k}(2d+2,d)$, so $\eta_{k}(2d+2,d)$ provides  the lower bound for the number of $k$-faces in $d$-polytopes with $2d+2$ vertices. 

% \item If $\floor{0.4d}\le k\le d-1$, then $\eta_{k}(2d+2,d)>\tau_{k}(2d+2,d)$, making $\tau_{k}(2d+2,d)$ the lower bound for $k$-faces in $d$-polytopes with $2d+2$ vertices.
% \end{enumerate}

%It is worth noting that a $d$-polytope with $2d+2$ vertices and $d+2$ facets exists (if and) only if $d+1$ is composite and $d\ne3$. This is an easy consequence of the last sentence of \cref{lem:dplus2facets}.

Each triplex has $d+2$ facets.  The structure of  $d$-polytopes with $d+2$ facets is well understood; see, for instance, \cref{lem:dplus2facets}. In lower bound theorems for  $d$-polytopes with few vertices (see, for instance, \cref{thm:2dplus1,thm:2dplus2}),  a dichotomy among the minimisers emerges: they have either $d+2$ or $d+3$ facets. Pineda-Villavicencio et al.~\cite{PinWanYos24a}  extended this dichotomy to $d$-polytopes with at most $2d$ vertices by refining \Cref{thm:at-most-2d}. For convenience, define
\begin{equation}
\label{eq:atmost2d-refined}
\zeta_k(d+s,d):=\theta_k(d+s,d)+\binom{d-1}{k}-\binom{d+1-s}{k}.
\end{equation}

  \begin{theorem}[Refined theorem for $d$-polytopes with at most $2d$ vertices]
\label{thm:at-most-2d-refined-short}  
Given parameters $d\ge 3$ and $2\le s \le d$, and a  $d$-polytope $P$ with $d+s$ vertices, the following statements hold: 
\begin{enumerate} [{\rm (i)}]
    \item If  $P$ has $d+2$ facets, then $f_k(P)\ge \theta_k(d+s,d)$ for all $k\in[1\ldots d-2]$. Furthermore, for some $k\in[1\ldots d-2]$, there is a unique polytope whose number of $k$-faces equals $\theta_k(d+s,d)$.
    \item If $s=2$ and $P$ has at least $d+\ell$ facets where $\ell\in [3\ldots d]$, then $f_k(P)\ge \theta_k(d+2,d)+\binom{d-1}{k-1}-\binom{d-\ell+1}{k-\ell+1}$ for all $k\in[1\ldots d-2]$. Moreover, for each  $k\ge \ell-1$,  there is a unique polytope  whose number of $k$-faces equals this lower bound.
    \item If $s\in [3\ldots d]$ and $P$ has at least $d+3$ facets, then $f_k(P)\ge\zeta_k(d+s,d)$. Additionally, for any fixed $k\in[1\ldots d-2]$, there are two polytopes for which the number of $k$-faces equals $\zeta_k(d+s,d)$. If  $s<d$ and $k\ge d-s+2$, a third minimiser arises, and when $s=4$, two further examples appear.  
\end{enumerate}
\end{theorem}

We continue with a basic bound for the number of $k$-faces in a polytope.

\begin{lemma} If $F$ is a facet of a $d$-polytope $P$, then, for each $k\in [0\ldots d-1]$,  $f_{k}(P)\ge f_{k}(F)+f_{k-1}(F)$. 
\label[lemma]{lem:facet-k-faces}
\end{lemma}
  
We next present with a characterisation of simplices and state the excess theorem.

\begin{proposition} A $d$-polytope is a pyramid over $d-1$ or more of its facets if and only if it is a simplex.
  \label[proposition]{prop:polytope-pyramid-all-facets}
\end{proposition}

The \emph{excess degree} of a vertex $u$ in a $d$-polytope is defined as $\deg u-d$. The \emph{excess degree} of  the polytope is  defined as  the sum of the excess degrees of its vertices \cite{PinUgoYos16a}. The excess theorem is stated below.

\begin{theorem}[Excess theorem, {\cite[Thm.~3.3]{PinUgoYos16a}}] The excess degree of any nonsimple $d$-polytope is at least $d-2$.
    \label{thm:excess-degree}
\end{theorem}

Barnette and Reay~\cite{BarRea73} characterised the pairs $(f_{0},f_{2})$ for which there is a 4-polytope with $f_{0}$ vertices and $f_{2}$ 2-faces.   Their result is as follows.   

\begin{theorem}[{\cite[Thm.~10]{BarRea73}}] There is a  $4$-polytope $P$ such that $(f_{0},f_{2}):=(f_{0}(P),f_{2}(P))$ if and only if    
\begin{equation*}
10\le \frac{1}{2}\left(2f_{0}+3+\sqrt{8f_{0}+9}\right)\le f_{2}\le f_{0}^{2}-3f_{0},     
\end{equation*}    
\label{thm:projections-d4-f0f2}
$ f_{2}\ne f_{0}^{2}-3f_{0}-1$, and $(f_{0}, f_{2})$ is not one of the pairs $(6, 12), (6, 14), (7, 13)$,  $(7, 15), (8, 15),(8, 16)$, $ (9, 16), (10, 17),  (11, 20), (13, 21)$.     
\end{theorem}

\section{Proof of the main theorem}
\label{sec:main-proof}

\begin{theorem}Let $d\ge 3$,  $\ell\ge 1$, and let $P$ be a $d$-polytope with $2d+\ell$ vertices. If $P$ has at least $d+3$ facets, then $f_{k}(P)\ge \eta_{k}(2d+\ell,d)$, for each $k\in [1\ldots d-2]$. Moreover, equality holds for some $k$ only if $P$ has exactly $d+3$ facets.
 \label{thm:3minus1}
\end{theorem}

%\begin{proof} 
We prove the inequality by induction on $d\ge 3$ for all $\ell\ge 1$  and address the equality afterwards.  The case $\ell=1$ was settled independently in \cite{PinYos22} and \cite{Xue24}, while the case $\ell=2$  was proved in \cite{PinTriYos24}. Thus we assume $\ell\ge 3$. 

For \(d = 3\), we have  $f_{2}\ge 6$. By Euler's formula and \eqref{eq:3minus1-function-extension},   
\begin{equation*}
f_{1}=f_{0}+f_{2}-2\ge 6+ \ell + 6-2=\ell +10>12=\eta_{1}(3d-1,3)=\eta_{1}(6+\ell,3).  
\end{equation*}

For  $d=4$, since $\ell\ge 3$,  
\begin{equation*}
f_1\ge 2f_0=2(2d+\ell)=16+2\ell\ge 22=\eta_1(3d-1,4)=\eta_1(8+\ell,4),
\end{equation*}
where equality holds if and only if $P$ is a simple 4-polytope with 11 vertices.
%f_1\ge \eta_1(8+\ell,4)=\eta_1(11,4)=22
%the number of edges of a simple 4-polytope with 11 vertices. % (with equality holds if and only if it is a simple 4-polytope with 11 vertices). 

According to \cref{thm:projections-d4-f0f2},  $f_2\ge \frac{1}{2}\left(2f_{0}+3+\sqrt{8f_{0}+9}\right)$. Since $f_0=8+\ell\ge 11$, we have 
 \begin{equation*}
     f_2\ge\frac{1}{2}\left(2f_{0}+3+\sqrt{8f_{0}+9}\right)\ge  18= \eta_2(11,4)=\eta_2(8+\ell,4),
 \end{equation*} with equality if and only if  $P$ is a 4-polytope with 11 vertices. In that case, since $f_1(P)\ge 22$, $f_3(P)\ge 7$, and $f_0(P)=11$, Euler's relation yields 
\begin{equation*}
   f_2(P)=f_1(P)+f_3(P)-f_0(P)\ge 29-11=18, 
\end{equation*}
so equality occurs  if and only if $P$ is a simple 4-polytope with 11 vertices; specifically $J(4,4)$, the polytope obtained by truncating a simple vertex in a simplicial $4$-prism.

Consider the pair $(d,\ell)$ with $d\ge 5$ and $\ell\ge 3$ and assume the inequality holds for all pairs $(d',\ell')$ with  $3\le d'<d$ and  $\ell\ge 1$. The proof is divided into four claims, organised according to the number of vertices of a selected facet.

\subsection{Claim \ref{cl:3minus1-pyramid}: The pyramid case}

\begin{claim} If $P$ is a pyramid, then $f_{k}(P)\ge \eta_{k}(2d+\ell,d)$ for each $k\in [1\ldots d-2]$.
\label{cl:3minus1-pyramid}
\end{claim}
\begin{claimproof} The base $F$ of $P$ has   $2(d-1)+\ell+1$ vertices and at least $d+2$ facets. By the induction hypothesis, $f_{k}(F)\ge \eta_{k}(2(d-1)+\ell+1,d-1)$ for $\ell\ge 1$ and $k\in [1\ldots d-3]$. We remark that $f_{d-2}(F)\ge \eta_{d-2}(2(d-1)+\ell+1,d-1)=d+2$.  Thus, for each $k\in [2\ldots d-2]$ and $1\le \ell\le d-2$, we find 
\begin{equation}\label{eq:claim1_smallell_fk}
  \begin{aligned}
    f_{k}(P)&=f_{k}(F)+f_{k-1}(F)\\
            &\ge \eta_{k}(2(d-1)+\ell+1,d-1)+\eta_{k-1}(2(d-1)+\ell+1,d-1)=\eta_{k}(2d+\ell+1,d)\\
            &\ge \eta_{k}(2d+\ell,d). 
  \end{aligned}
\end{equation}

For $k=1$, we have 
\begin{equation}\label{eq:claim1_smallell_f1}
  \begin{aligned}
    f_{1}(P)&=f_{1}(F)+f_{0}(F)\\
      &\ge \eta_{1}(2(d-1)+\ell+1,d-1)+2(d-1)+\ell+1 =\eta_{1}(2d+\ell+1,d)-1 \; \ge \eta_{1}(2d+\ell,d).
  \end{aligned}
\end{equation}
Equality  holds for $k=1$ if and only if $\ell=d-2$, $f_0(F)=3d-3$, and $f_1(F)=\eta_1(3d-4, d-1)$ by \eqref{eq:3minus1-function-extension}. However, $2f_1(F)<(d-1)f_0(F)$:
\begin{align*}
 2f_1(F)=2\eta_1(3d-4, d-1)=3d^2-7d+4<3d^2-6d+3=(d-1)f_0(F),
\end{align*}
which is a contradiction.

Suppose $\ell\ge d-1$. In this case, for $k\in [1\ldots d-3]$ the induction hypothesis yields $f_{k}(F)\ge \eta_{k}(2(d-1)+d-2,d-1)$,  as per \eqref{eq:3minus1-function-extension}. Then, for $k\in [2\ldots d-2]$, we obtain
\begin{equation}\label{eq:claim1_largeell_fk}
    \begin{aligned}
     f_{k}(P)&=f_{k}(F)+f_{k-1}(F)\\
     &\ge \eta_{k}(2(d-1)+d-2,d-1)+\eta_{k-1}(2(d-1)+d-2,d-1) = \eta_{k}(3d-2,d)\\
     &=\eta_{k}(3d-1,d).
  \end{aligned} 
\end{equation}

Additionally, for $k=1$, 
\begin{equation}\label{eq:claim1_largeell_f1}
    \begin{aligned}
     f_{1}(P)&=f_{1}(F)+f_{0}(F)\ge \eta_1(2(d-1)+d-2, d-1)+3d-2 =\eta_{1}(3d-1,d).
    \end{aligned}
\end{equation}
Equality  holds for $k=1$ if and only if $\ell=d-1$, $f_0(F)=3d-2$, and $f_1(F)=\eta_1(3d-4, d-1)$. However, $2f_1(F)<(d-1)f_0(F)$:
\begin{align*}
 2f_1(F)=2\eta_1(3d-4, d-1)=3d^2-7d+4<3d^2-5d+2=(d-1)f_0(F),
\end{align*}
which is a contradiction. This concludes the proof of the claim.
\end{claimproof}

%Thus, for each $k\in [2\ldots d-2]$ and $1\le \ell\le d-3$, we find 
%\begin{equation}\label{eq:claim1_smallell_fk}
%  \begin{aligned}
%    f_{k}(P)&=f_{k}(F)+f_{k-1}(F)\\
%            &\ge \eta_{k}(2(d-1)+\ell+1,d-1)+\eta_{k-1}(2(d-1)+\ell+1,d-1)=\eta_{k}(2d+\ell+1,d)\\
%            &\ge \eta_{k}(2d+\ell,d). 
%  \end{aligned}
%\end{equation}
%For $k=1$, we have 
%\begin{equation}\label{eq:claim1_smallell_f1}
%  \begin{aligned}
%    f_{1}(P)&=f_{1}(F)+f_{0}(F)\\
%      &\ge \eta_{1}(2(d-1)+\ell+1,d-1)+2(d-1)+\ell+1=\eta_{1}(2d+\ell+1,d)-1 > \eta_{1}(2d+\ell,d).
%  \end{aligned}
%\end{equation}

%Suppose \red{ $\ell\ge d-2$.................} . In this case the induction hypothesis yields  $f_{k}(F)\ge \eta_{k}(2(d-1)+d-2,d-1)$, as per \eqref{eq:3minus1-function-extension}. Then, for $k\in [2\ldots d-2]$, we obtain
%\begin{equation}\label{eq:claim1_largeell_fk}
%    \begin{aligned}
%     f_{k}(P)&=f_{k}(F)+f_{k-1}(F)\\
%     &\ge \eta_{k}(2(d-1)+d-2,d-1)+\eta_{k-1}(2(d-1)+d-2,d-1) = \eta_{k}(3d-2,d)\\ &=\eta_{k}(3d-1,d).
%  \end{aligned} 
%\end{equation}
%Additionally, for $k=1$, 
%\begin{equation}\label{eq:claim1_largeell_f1}
 %   \begin{aligned}
  %   f_{1}(P)&=f_{1}(F)+f_{0}(F)\\
   %  &\ge \eta_{1}(3(d-1)-1,d-1)+3(d-1)+1=\eta_{1}(3d-2,d)+1=\eta_{1}(3d-1,d).
%    \end{aligned}
%\end{equation}

We now record a useful observation, which is a special case of \cref{rmk:dplus2facets-facets}.
\begin{remark}
  \label{rmk:simple-dplus1-facet-faces} Let  $2\le m\le \floor{(d-1)/2}$, and let $F$ be a simple $(d-1)$-polytope of the form $T(m)\times T(d-1-m)$.  By  \cref{rmk:dplus2facets-facets},  the following statements hold for $d\ge 5$.
  \begin{enumerate}[{\rm (i)}]
  \item A $(d-2)$-face $R$ of $F$ is of the form either $T(m-1)\times T(d-1-m)$ with $m(d-m)$ vertices or $T(m)\times T(d-2-m)$ with $(m+1)(d-1-m)$ vertices. Since $2\le m\le \floor{(d-1)/2}$,  $f_0(R)\ge 2d-4$ in both cases. 
  
 \item A $(d-3)$-face $K$ of $F$ is of the form $T(m-2)\times T(d-1-m)$  with $(m-1)(d-m)$ vertices, $T(m-1)\times T(d-2-m)$ with $m(d-1-m)$ vertices, or $T(m)\times T(d-3-m)$ with $(m+1)(d-2-m)$ vertices. For $m\ge3$, we have $f_0(K)\ge 2d-6$ in each case. When $m=2$, we still have $f_0(K)\ge d-2$. 
\end{enumerate}   
\end{remark}

\subsection{Claim  \ref{cl:3minus1-minus-2}: A facet with $f_{0}(P)-2$ vertices}
We begin with an overview of the proof. Let $F$ be a facet with $f_{0}(P)-2$ vertices, and let $X:=\set{v_1,v_2}$ be the vertices outside $F$. We distinguish two cases according to the number of facets of $F$.
\begin{enumerate}   
\item \textbf{\cref{case:3minus1-minus-2-dplus2}: $F$ has at least $d+2$ facets}. We argue as in \cref{cl:3minus1-pyramid}, applying the induction hypothesis together with \cref{lem:facet-k-faces} to get the result. 
\item  \textbf{\cref{case:3minus1-minus-2-dplus1} :  $F$ has exactly $d+1$ facets}. We further split into two subcases according to \cref{lem:dplus2facets}. 

 We use \cref{cor:number-faces-outside-facet-practical} to find  faces $F_1,F_2$ in $P$ such that each $F_i$ has dimension $d-i+1$, contains $v_i\in X$, and excludes any $v_j\in X$ with $j<i$. Note that $F_1=P$. Hence the number of $k$-faces of $P$ that contain at least one vertex in $X$ is bounded below by 
\begin{equation}
\label{eq:claim2-faces-in-X}
    \sum_{i=1}^2 f_{k-1}(F_{i}/v_{i}).
\end{equation}
Moreover,  each vertex in $X$ is the apex of a pyramid over a $(d-2)$-face of $F$.
\begin{enumerate}
\item \emph{Subcase~2.\ref{subcase:3minus1-minus-2-dplus1-nonsimple}: $F$ is a nonsimple polytope}. By \cref{lem:lower-bound-dplus2-facets}, $f_k(F)\ge \tau_{k}(2(d-1)+\ell,d-1)$, and if $\ell\ge d-4$ then $f_{k}(F)\ge \tau_{k}(3(d-1)-3,d-1)$ by \cref{rmk:functions-dplus2-facets}. Since $F$ is a pyramid over a $(d-2)$-face $R$ (\cref{lem:dplus2facets}), the other facet  containing $R$ must contain either $v_{1}$ or $v_{2}$, say $v_{1}$. Thus the apex $v_{1}$ has degree at least $2(d-1)+\ell$ in $P$. 
    
    The next step is to bound $f_{k-1}(P/v_1)$.   If $P/v_1$ has at least $d+2$ facets, we apply the induction hypothesis; if it has exactly $d+1$ facets,  we use \cref{lem:lower-bound-dplus2-facets} or \cref{rmk:functions-dplus2-facets}. Together with \eqref{eq:claim2-faces-in-X}, this settles the subcase. 

\item \emph{Subcase~2.\ref{subcase:3minus1-minus-2-dplus1-simple}: $F$ is a simple polytope}. Here $f_{k}(F)\ge \rho_{k}(d-1,m,d-1)$ by \cref{lem:dplus2facets}.  By analysing the $(d-2)$-faces of $F$ (see   \cref{rmk:simple-dplus1-facet-faces}(i)), we conclude that both vertices in $X$ has degree  at least $2d-4$ in some facet of $P$ and at least $2d-3$ in $P$. Thus,  by \eqref{eq:claim2-faces-in-X} and \cref{cor:more-dpluss}, the number of $k$-faces containing vertices of $X$ satisfies
\begin{equation*}
    \sum_{i=1}^2 f_{k-1}(F_{i}/v_{i})\ge \theta_{k-1}(2(d-1)-1,d-1)+\theta_{k-1}(2(d-2),d-2).
\end{equation*}
These elements settle the subcase. 
\end{enumerate}
\end{enumerate}  
 We now proceed with the details.

%If $F$ being a simple polytope, we can't conclude that a vertex in $X$ has a large degree, which makes the case often more involved.  

\begin{claim} If $P$ has a facet $F$ with $f_{0}(P)-2$ vertices, then $f_{k}(P)\ge \eta_{k}(2d+\ell,d)$ for each $k\in [1\ldots d-2]$.
\label{cl:3minus1-minus-2}
\end{claim}
\begin{claimproof} By \cref{cl:3minus1-pyramid}, we may assume that $P$ is not a pyramid. We distinguish two cases according to the number of $(d-2)$-faces in $F$.

\case The facet $F$ has at least $d+2=(d-1)+3$ facets.
\label{case:3minus1-minus-2-dplus2}

For $k\in[2\ldots d-2]$ and $\ell\ge 3$, the induction hypothesis on $d-1$ and  \cref{lem:facet-k-faces} give 
\begin{equation}\label{eq:claim2_case1_fk}
    \begin{aligned}
     f_{k}(P)&\ge f_{k}(F)+f_{k-1}(F)\\
     &\ge \eta_k(2(d-1)+\ell,d-1)+\eta_{k-1}(2(d-1)+\ell,d-1) \\
     &\ge\brac*{\binom{d}{k+1}+2\binom{d-1}{k+1}-\binom{d-1-\ell}{k+1}} +\brac*{\binom{d}{k}+2\binom{d-1}{k}-\binom{d-1-\ell}{k}}\\
     &=\binom{d+1}{k+1}+2\binom{d}{k+1}-\binom{d-\ell}{k+1}= \eta_k(2d+\ell,d).
    \end{aligned}
\end{equation}
Additionally, since the two vertices outside $F$ are adjacent, 
\begin{equation}\label{eq:claim2_case1_f1}
    \begin{aligned}
    f_{1}(P)&\ge f_{1}(F)+f_{0}(F)+1\\
    &\ge \eta_{1}(2(d-1)+\ell,d-1)+2(d-1)+\ell+1=\eta_{1}(2d+\ell,d).
    \end{aligned}
\end{equation}
 %\green{Equality holds for all $k\in [1\ldots d-2]$.}
 \case The facet $F$ has $d+1$ facets.
 \label{case:3minus1-minus-2-dplus1} 
 \begin{subcases}

 Then $F$ is  a $(d-1-a)$-fold pyramid over $T(m)\times T(a-m)$   for some $2\le a\le d-1$ and $2\le m\le \floor{a/2}$ (\cref{lem:dplus2facets}). Let $v_{1}$ and $v_{2}$ be the two vertices outside $F$. 
 By \cref{cor:number-faces-outside-facet-practical}, there exist faces $F_1, F_2$ in  $P$ of dimension \(d\) and \(d - 1\) respectively, such that $F_i$ contains $v_i$ but does not contain any $v_j$ with $j<i$.  The number of $k$-faces of $P$ that contain at least one of $\set{v_{1},v_{2}}$ is bounded below by $f_{k-1}(F_{1}/v_{1})+f_{k-1}(F_{2}/v_{2})$. 
 
 \subcase \emph{The facet $F$ is a nonsimple polytope.} 
 \label{subcase:3minus1-minus-2-dplus1-nonsimple} 

 By \cref{lem:lower-bound-dplus2-facets}, $f_k(F)\ge \tau_{k}(2(d-1)+\ell,d-1)$ for $k\in [1\ldots d-2]$.  If $\ell\ge d-4$, then, by \cref{rmk:functions-dplus2-facets}, we use $f_{k}(F)\ge \tau_{k}(3(d-1)-3,d-1)$. Since $F$ is a pyramid over a $(d-2)$-face $R$, the other facet $F'$ containing $R$ must contain either $v_{1}$ or $v_{2}$, say $v_{1}$, or else $P$ would be a pyramid over $F'$. Since $f_{0}(F')= 2(d-1)+\ell$, its apex $v_{1}$ has degree at least $2(d-1)+\ell$ in $P$.

Any facet of $P$ containing $v_2$ but not $v_1$ (for instance $F_2$) must be a pyramid over a $(d-2)$-face of $F$, and by \cref{rmk:dplus2facets-facets}, since $m\ge 2$, this $(d-2)$-face  has at least $d+1$ vertices. Thus, \(\deg_{F_2}(v_2) \ge d + 1\), and by \cref{cor:more-dpluss}, we obtain 
\begin{equation}
\label{eq:claim2_subcase2.1-F2}
    f_{k-1}(F_{2}/v_{2})\ge \theta_{k-1}(d+1,d-2).
\end{equation} 

For $k=1$, we find  
\begin{equation}
\label{eq:claim2_subcase2.1_f1}
    \begin{aligned}
        f_{1}(P)&\ge f_{1}(F)+\sum_{i=1}^{2}f_{0}(F_{i}/v_{i})\\
&\ge \tau_{1}(2(d-1)+\ell,d-1)+2(d-1)+\ell+d+1,
    \end{aligned}
\end{equation}
and \cref{lem:combinatorial-ineq}(iii) ensures  $f_{1}(P)> \eta_{1}(2d+\ell,d)$.
% \begin{equation*}
% .
% \end{equation*} 

For $k\in [2\ldots d-2]$, two possibilities arise depending on the number of $(d-2)$-faces in $F_{1}/v_{1}$.

\textbf{Suppose that the number of $(d-2)$-faces in $F_{1}/v_{1}$ is at least $(d-1)+3$}. For $k\in [2\ldots d-2]$ and $\ell\ge 1$,  the induction hypothesis for $d-1$ and \cref{lem:lower-bound-dplus2-facets}  yield 
% that
% \begin{equation*}
%     \begin{aligned}
%         \sum_{i=1}^{2}f_{k-1}(F_{i}/v_{i})&\ge \eta_{k-1}(2(d-1)+\ell,d-1)+\theta_{k-1}(d+1,d-2).
%     \end{aligned}
% \end{equation*}
%&=\binom{d}{k}+2\binom{d-1}{k}-\binom{d-3}{k}+\binom{d-1}{k}.
%In this case,  gives that 
\begin{equation} 
\label{eq:claim2_subcase2.1-dplus2}
\begin{aligned}
f_{k}(P)&\ge f_{k}(F)+\sum_{i=1}^{2}f_{k-1}(F_{i}/v_{i})\\
&\ge \tau_{k}(2(d-1)+\ell,d-1)+\eta_{k-1}(2(d-1)+\ell,d-1)+\theta_{k-1}(d+1,d-2)\\
&= \tau_{k}(2(d-1)+\ell,d-1)+\eta_{k-1}(2(d-1)+\ell,d-1)+\binom{d-1}{k}+\binom{d-2}{k}-\binom{d-4}{k}\\
&> \eta_{k}(2d+\ell,d).
\end{aligned}
\end{equation} 
 If $\ell\ge d-2$  then $f_{k-1}(F_1/v_1)\ge \eta_{k-1}(2(d-1)+d-2,d-1)$ (see \eqref{eq:3minus1-function-extension}). In any case, \cref{lem:combinatorial-ineq}(ii) yields $f_{k}(P)> \eta_{k}(2d+\ell,d)$.

\textbf{Suppose the number of $(d-2)$-faces in $F_{1}/v_{1}$ is $(d-1)+2$}. Then \cref{lem:lower-bound-dplus2-facets} yields 
%that
% \begin{align*}
% \sum_{i=1}^{2}f_{k-1}(F_{i}/v_{i})&\ge \tau_{k-1}(2(d-1)+\ell,d-1)+\theta_{k-1}(d+1,d-2).%\\
% %&=\binom{d}{k}+2\binom{d-1}{k}-\binom{d-3}{k}+\binom{d-1}{k}.
% \end{align*}
% In this case, \cref{lem:lower-bound-dplus2-facets}  gives that 
\begin{equation} 
\label{eq:claim2_subcase2.1-dplus1}
\begin{aligned}
f_{k}(P)&\ge f_{k}(F)+\sum_{i=1}^{2}f_{k-1}(F_{i}/v_{i})\\
&\ge \tau_{k}(2(d-1)+\ell,d-1)+\tau_{k-1}(2(d-1)+\ell,d-1)+\theta_{k-1}(d+1,d-2).
\\
&> \eta_{k}(2d+\ell,d).
\end{aligned}
\end{equation} 
If $\ell\ge d-4$ then $f_{k}(F_1/v_1)\ge \tau_{k}(3(d-1)-3,d-1)$ (\cref{rmk:functions-dplus2-facets}). \cref{lem:combinatorial-ineq}(iv) gives that $f_{k}(P)> \eta_{k}(2d+\ell,d)$ for $k\in [2\ldots d-2]$.
% \begin{equation*}

% \end{equation*}

\subcase \emph{The facet $F$ is a simple polytope.} 
 \label{subcase:3minus1-minus-2-dplus1-simple}
 
Here, $F=T(m)\times T(d-1-m)$ where  $2\le m\le \floor{(d-1)/2}$ (\cref{lem:dplus2facets}), and $f_0(F)=d+m(d-1-m)\ge 3d-6$. So $f_0(P)=d+m(d-1-m)+2\ge 3d-4$, implying $\ell\ge d-4$. The number of $k$-faces in $F$ is  $\rho_{k}(d-1,m,d-1)$ by \cref{lem:dplus2facets}. 

 A facet $F_{2}$ containing $v_{2}$ but not $v_{1}$ must be a pyramid over a $(d-2)$-face $R$ of $F$. Then $F_2$ has $d+1$ $(d-2)$-faces. By   \cref{rmk:simple-dplus1-facet-faces}(i), each such face $R$ has at least $2d-4$ vertices for $d\ge 4$, implying $\deg_{F_{2}}(v_{2})\ge 2d-4$ and $\deg_{P}(v_{2})\ge 2d-3$. The same reasoning yields   $\deg_{P}(v_{1})\ge 2d-3$. Consequently, the $k$-faces of $P$ include  at least $\theta_{k-1}(2(d-2),d-2)$ $k$-faces containing $v_{2}$ inside $F_{2}$  by \cref{cor:more-dpluss}, plus at least $\theta_{k-1}(2(d-1)-1,d-1)$ $k$-faces  containing $v_{1}$ in $P$   by \cref{cor:more-dpluss}, and $\rho_{k}(d-1,m,d-1)$ $k$-faces in $F$. 

Assume $m=2$. Then $F=T(2)\times T(d-3)$, implying $f_0(F)=3(d-1)-3$ and $f_0(P)=3d-4$. %According to \cref{rmk:dplus2facets-facets}, $F$ contains a $(d-2)$-face $R$ with $3(d-2)-3$ vertices. If the other facet $J$ containig $R$ contained the two vertices $v_1$ and $v_2$, then $P$ would be a pyramid over $J$.
\begin{equation}
\label{eq:claim2_subcase2.2-m2}
    \begin{aligned}
f_{k}(P)&\ge f_{k}(F)+\theta_{k-1}(2(d-1)-1,d-1)+\theta_{k-1}(2(d-2),d-2)\\
  &\ge \brac*{\binom{d}{k+1}+\binom{d-1}{k+1}+\binom{d-2}{k+1}-\binom{3}{k+2}}+\brac*{\binom{d}{k}+\binom{d-1}{k}-\binom{2}{k}}\\
  &\quad +\brac*{\binom{d-1}{k}+\binom{d-2}{k}-\binom{1}{k}}\\
  &=\binom{d+1}{k+1}+2\binom{d}{k+1}-\binom{3}{k+2}-\binom{2}{k}-\binom{1}{k} \\
  &=\eta_{k}(2d+d-4,d)+\binom{4}{k+1}-\binom{3}{k+2}-\binom{2}{k}-\binom{1}{k}\ge\eta_{k}(3d-4,d).
    \end{aligned}
\end{equation}
Equality holds for $k\in [1\ldots d-2]$ if and only if $k\ge 4$.

Assume $m\ge 3$. Then  $d\ge 7$, and by \cref{lem:dplus2-vertices-inequalities}, we have $\rho_{k}(d-1,m,d-1)\ge \rho_{k}(d-1,3,d-1)$. Since either $f_0(R)=m(d-m)\ge 2d-2$ or $f_0(R)=(m+1)(d-1-m)\ge 2d-2$, it follows that $\deg_{F_{2}}(v_{2})\ge 2d-2$ and $\deg_{P}(v_{1})\ge 2d-1$. By \cref{thm:2dplus1}, since $F_2$ has $d+1$ $(d-2)$-faces, we obtain 
\begin{equation*}
    f_{k-1}(F_2/v_2)\ge f_{k-1}((T_{2}^{d-2,d-2-(\floor{(d-2)/2}+2)})^{*})>\theta_{k-1}(2(d-2),d-2).
\end{equation*}
Furthermore,  if $P/v_1$ has $d+1$ or at least $d+2$ facets, then by \cref{thm:2dplus1}
\begin{equation*}
    f_{k-1}(P/v_1)\ge f_{k-1}((T_{2}^{d-1,d-1-(\floor{(d-1)/2}+2)})^{*})\; \text{or}\; f_{k-1}(P/v_1)\ge f_{k-1}(J(2,d-1)),
\end{equation*} respectively. Recall that the $(s,d-s)$-triplex is the $(d-s)$-fold pyramid over $T(1)\times T(s-1)$,  $J(2,d-1)$ is the polytope obtained by truncating  a simple vertex  in a $(2,d-3)$-triplex, and $(T_{2}^{d-1,d-1-(\floor{(d-1)/2}+2)})^*$ is the $(d-3-\floor{(d-1)/2})$-fold pyramid over $T(2)\times T(\floor{(d-1)/2})$. In both cases,  $f_{k-1}(P/v_1)>\theta_{k-1}(2(d-1),d-1)$. Thus, for $\ell\ge d-4$, 
\begin{equation}\label{eq:claim2_subcase2.2_mge3}
  \begin{aligned}
      f_{k}(P) &\ge  \rho_{k}(d-1,3,d-1)+f_{k-1}(P/v_1)+f_{k-1}(F_2/v_2)\\
      &> \rho_{k}(d-1,3,d-1)+\theta_{k-1}(2(d-1),d-1)+\theta_{k-1}(2(d-2),d-2)\\
      & = \binom{d+1}{k+1}+2\binom{d}{k+1}+\binom{d-3}{k+1}-\binom{4}{k+2}-2\binom{1}{k} \\
      & = \eta_k(2d+\ell,d)+\binom{d-\ell}{k+1}+\binom{d-3}{k+1}-\binom{4}{k+2}-2\binom{1}{k} \ge \eta_k(2d+\ell,d).
  \end{aligned}    
\end{equation}
Hence $f_k(P)>\eta_k(2d+\ell,d)$.
%where in the last inequality, since $d\ge 7$, we have $\binom{d-\ell}{k+1}+\binom{d-3}{k+1}-\binom{4}{k+2}-2\binom{1}{k} \ge \binom{4}{k+1}-\binom{4}{k+2}-2\binom{1}{k} \ge 0$ for $k\ge 1$.
%&\ge \rho_{k}(d-1,3,d-1)+\eta_{k-1}(2(d-1)+1,d-1)+\eta_{k-1}(2(d-2)+2,d-2)\\
%&\ge \brac*{\binom{d}{k+1}+\binom{d-1}{k+1}+\binom{d-2}{k+1}+\binom{d-3}{k+1}-\binom{4}{k+2}}\\
%&\quad  +\brac*{\binom{d}{k}+\binom{d-1}{k}+\binom{d-2}{k-1}}\\
%& \quad  +\brac*{\binom{d-1}{k}+\binom{d-2}{k}+\binom{d-3}{k-1}+\binom{d-4}{k-1}}\\
%&=\binom{d+1}{k+1}+2\binom{d}{k+1}+\binom{d-1}{k+1}+\binom{d-4}{k-1}-\binom{4}{k+2}\\
%&=\eta_{k}(2d+\ell,d)+\binom{d-\ell}{k+1}+\binom{d-1}{k+1}+\binom{d-4}{k-1}-\binom{4}{k+2}\\
%&\ge \eta_{k}(2d+\ell,d).
This completes the proof of the subcase, and  hence the proof of the claim. 
\end{subcases}
\end{claimproof} 

\subsection{Claim  \ref{cl:3minus1-more-2dplus1}: A facet with between $2d-1$ and $f_{0}(P)-3$ vertices}
We begin with an overview of the proof. By \cref{lem:2d+s-dplus3-simple-pyramid}, we may assume that $P$ contains a vertex that is both nonpyramidal and nonsimple; let  $v_1$ be one such vertex of maximum degree. Let $F$ be a facet of $P$ not containing $v_1$. By  \cref{cl:3minus1-pyramid,cl:3minus1-minus-2}, we may assume that $F$ has at most $f_{0}(P)-3$ vertices.   \cref{cl:3minus1-more-2dplus1} addresses the case $f_0(F)\ge 2d-1$, while  \cref{cl:3minus1-less-2dplus1} covers  $f_0(F)\le 2d-2$. If $f_0(F)\ge 2d-1$, then $f_k(F)$ is bounded by the induction hypothesis,  by \cref{lem:lower-bound-dplus2-facets} in combination with \cref{rmk:functions-dplus2-facets}, or by \cref{lem:dplus2facets}. If instead $f_0(F)\le 2d-2$, then $f_k(F)$ is bounded by \cref{thm:at-most-2d} or \cref{thm:at-most-2d-refined-short}. These different estimates for $f_k(F)$ lead to \cref{cl:3minus1-more-2dplus1,cl:3minus1-less-2dplus1}.  

Write $f_0(F)=2d+\ell-r$, where $3\le r\le \ell+1$, and let $X:=\set{v_1,\ldots,v_r}$ be the set of vertices outside $F$. We use \cref{cor:number-faces-outside-facet-practical} to find  faces $F_1,\ldots,F_r$ in $P$ such that each $F_i$ has dimension $d-i+1$, contains $v_i\in X$, and excludes any $v_j\in X$ with $j<i$. Observe that $F_1=P$. Thus, the number of $k$-faces of $P$ that contain at least one vertex in $X$ is bounded below by 
\begin{equation}
\label{eq:faces-in-X}
    \sum_{i=1}^r f_{k-1}(F_{i}/v_{i}).
\end{equation}
 Two cases arise, according to the number of facets of $F$.

\begin{enumerate}   
\item \textbf{\cref{case:3minus1-minus-r-dplus2}: $F$ has at least $d+2$ facets}.  We apply the induction hypothesis to bound $f_k(F)$. Since $v_1$ is nonsimple in $P$, we use $\theta_{k-1}(d+1,d-1)$ (see \cref{cor:more-dpluss}) to bound $f_{k-1}(P/v_1)$. Together with \eqref{eq:faces-in-X}, these bounds yield the result.
\item \textbf{\cref{case:3minus1-minus-r-dplus1}: $F$ has exactly $d+1$ facets}. We split further  according to \cref{lem:dplus2facets}.
\begin{enumerate}
    \item \emph{Subcase~2.\ref{subcase:3minus1-minus-r-dplus1-nonsimple}: $F$ is a nonsimple polytope}. In this subcase $F$ is a pyramid, whose apex has degree at least $2d+\ell-r$ in $P$ (\cref{lem:dplus2facets}). Moreover, $f_k(F)\ge \tau_{k}(2(d-1)+\ell-r+2,d-1)$ by \cref{lem:lower-bound-dplus2-facets}, and if $\ell-r+2\ge d-4$ then $f_{k}(F)\ge \tau_{k}(3(d-1)-3,d-1)$ by \cref{rmk:functions-dplus2-facets}. 
    
    To bound $f_{k-1}(P/v_1)$, note that by the maximality of $v_1$,  $f_0(P/v_1)=\deg_P(v_1) \ge 2(d-1) + \ell - r+2$.   If $P/v_1$ has at least $d+2$ facets, we apply the induction hypothesis; if it has exactly $d+1$ facets,  we use \cref{lem:lower-bound-dplus2-facets} or \cref{rmk:functions-dplus2-facets}. Together with \eqref{eq:faces-in-X}, these elements settle the subcase. 

\item     \emph{Subcase~2.\ref{subcase:3minus1-minus-r-dplus1-simple}: $F$ is a simple polytope}. Here $f_{k}(F) =\rho_{k}(d-1,m,d-1)$ by \cref{lem:dplus2facets}. The difficulty lies in bounding $f_{k-1}(P/v_1)$, for which we only have $\theta_{k-1}(d+1,d-1)$. The bounds for $f_k(F)$ and $f_{k-1}(P/v_1)$ together with  \eqref{eq:faces-in-X}  are insufficient to get the desired inequality. Thus, we refine the argument by applying \cref{cor:number-faces-outside-facet-practical} to vertices inside a facet $J_1\ne F$ and to vertices outside $F\cup J_1$, thereby replacing \eqref{eq:faces-in-X} and  $\theta_{k-1}(d+1,d-1)$ in the computation. For this approach to yield a larger lower bound  than \eqref{eq:faces-in-X}, we need at least two vertices outside $F\cup J_1$.
\begin{enumerate}
    \item  \textbf{Suppose some facet $J_1$ has at least two vertices outside $F\cup J_1$.}  Let $J_{1}\cap X=\set{v'_{1},\ldots,v'_{t}}$, where $1\le t\le \card X-2=r-2$.  By \cref{cor:number-faces-outside-facet-practical}, there exist faces $J_1,J_2\ldots, J_{t}$ in $J_{1}$ such that each $J_i$ has dimension $d-1-i+1$,  contains $v'_i$, and excludes any $v'_j$ with $j<i$. Thus the number of $k$-faces of $J_{1}$ containing at least one vertex in $J_{1}\cap X$ is bounded below by $  \sum_{i=1}^{t}f_{k-1}(J_{i}/v'_{i})$.
    
  Applying \cref{cor:number-faces-outside-facet-practical} again to  $X\setminus J_1=\set{v_1'',\dots, v_{r-t}''}$, we obtain faces $F_1',\ldots, F'_{r-t}$ in $P$ such that each $F_i'$ has dimension $d-i+1$, contains $v''_i$, and excludes any $v''_j$ with $j<i$.  Hence the number of $k$-faces of $P$ containing at least one vertex in $X\setminus J_1$ satisfies $  \sum_{i=1}^{r-t}f_{k-1}(F_{i}'/v''_{i})\ge \sum_{j=1}^{r-t}\binom{d+1-j}{k}$. As a result, \eqref{eq:faces-in-X} and $\theta_{k-1}(d+1,d-1)$ are replaced by 
  \begin{equation}
      \label{eq:faces-in-X-J1}
\sum_{i=1}^{t} f_{k-1}(J_{i}/v'_{i})+\sum_{j=1}^{r-t}\binom{d+1-j}{k}.
\end{equation}
Together with the bound for $f_k(F)$, the inequality $f_k(P)\ge \eta_{k}(2d+\ell,d)$ follows by running through the possible pairs $(r,t)$ in \eqref{eq:faces-in-X-J1}.

    \item \textbf{Suppose no facet $J_1$ has  two or more vertices outside $F\cup J_1$.} Since $P$ has at least $d+3$ facets while $F$ has only $d+1$, there must exist a facet $U_1$ intersecting $F$ in a face of dimension at most $d-3$. We first show that $U_1$ is a two-fold pyramid over $F\cap U_1$. 
    
    If $r\ge 4$, then we use $U_1$ to show that $\deg_P(v_1)\ge d+2$. Combined with the bounds for $f_k(F)$, \eqref{eq:faces-in-X}, and $f_{k-1}(P/v_1)\ge \theta_{k-1}(d+2,d-1)$ from \cref{cor:more-dpluss}, this gives the desired inequality.

    If $r=3$, the inequality is obtained by combining the bound for $f_k(F)$, \eqref{eq:faces-in-X}, and a detailed analysis of the structure of $F=T(m)\times T(d-1-m)$ and its possible $(d-2)$-faces.
\end{enumerate}
\end{enumerate}
\end{enumerate}   

Now we proceed with the details.

\begin{claim} Let $P$ be a $d$-polytope that contains a vertex that is neither pyramidal nor simple. Let $v_1$ be  such a vertex  in $P$ with the maximum degree, and let $F$ be a facet  that does not contain $v_1$ and has $2d+\ell-r$ vertices, where $3\le r\le \ell+1$. Then, for each $k\in [1\ldots d-2]$, it follows that $f_k(P)\ge \eta_k(2d+\ell,d)$. 
  \label{cl:3minus1-more-2dplus1}
\end{claim}
\begin{claimproof}
By  \cref{cl:3minus1-pyramid,cl:3minus1-minus-2}, we assume that $P$ is not a pyramid and has no facet with $f_{0}(P)-2$ vertices.  We consider two cases, based on the number of $(d-2)$-faces in $F$. Let $X:=\set{v_1,\ldots,v_r}$ be the set of vertices outside $F$. By \cref{cor:number-faces-outside-facet-practical}, there are faces $F_1,\ldots,F_r$ in $P$ such that each $F_i$ has dimension $d-i+1$, contains $v_i$, and excludes any $v_j$ with $j<i$. The number of $k$-faces of $P$ that contain at least one vertex in $X$ is bounded below by $\sum_{i=1}^r f_{k-1}(F_{i}/v_{i})$.
\case The facet $F$ has at least $d+2=(d-1)+3$ facets.
\label{case:3minus1-minus-r-dplus2}

Since $\deg_P(v_1)\ge d+1$,  we obtain $f_{k-1}(P/v_1)\ge \theta_{k-1}(d+1,d-1)$ and   
\begin{equation}
  \label{eq:3minus1-minus-r-dplus2-2}
  \begin{aligned}  
f_k(P)&\ge f_k(F)+\sum_{i=1}^{r}f_{k-1}(F_i/v_i)\\
&\ge \eta_k(2(d-1)+\ell-r+2,d-1)+\theta_{k-1}(d+1,d-1)+\sum_{i=1}^{r-1}\binom{d-i}{k}\\  
&= \brac*{\binom{d}{k+1}+2\binom{d-1}{k+1}-\binom{d-1-(\ell-r+2)}{k+1}}+\binom{d-2}{k-1}+\sum_{i=1}^{r}\binom{d+1-i}{k}\\
&=\binom{d+1}{k+1}+2\binom{d}{k+1}-\binom{d-1-(\ell-r+2)}{k+1}+\sum_{i=1}^{r-3}\binom{d-2-i}{k}\\
&=\eta_k(2d+\ell,d) +\brac*{ \binom{d-\ell}{k+1} -\binom{d-1-(\ell-r+2)}{k+1}}+\sum_{i=1}^{r-3}\binom{d-2-i}{k}\\
&=\eta_k(2d+\ell,d) - \sum_{i=1}^{r-3} \binom{d-\ell+r-3-i}{k}+\sum_{i=1}^{r-3}\binom{d-2-i}{k}.
\end{aligned}
\end{equation}

If $\ell-r+2\ge d-2$, we use $\eta_k(2(d-1)+d-2,d-1)$ in $\eta_k(2(d-1)+\ell-r+2,d-1)$ (see \eqref{eq:3minus1-function-extension}).  Since $r\le \ell+1$, we have $d-2\ge d-\ell+r-3$, and thus the last line in \eqref{eq:3minus1-minus-r-dplus2-2} is at least $\eta_k(2d+\ell,d)$. More precisely, equality holds for every $k$ if and only if $r=3$ or $r=\ell+1$; otherwise (that is, $4\le r\le \ell$) equality holds exactly when $k=d-2$.

%If either $r=3$ or $r\ge 4$ and $k=d-2$ then the last line in \eqref{eq:3minus1-minus-r-dplus2-2} is exactly $\eta_k(2d+\ell,d)$. If $r\ge 4$, then the last line in \eqref{eq:3minus1-minus-r-dplus2-2} is at least $\eta_k(2d+\ell,d)$.
\case The facet $F$ has precisely  $d+1=(d-1)+2$ facets.
\label{case:3minus1-minus-r-dplus1}

Then $F$ is a $(d-1-a)$-fold pyramid over $T(m)\times T(a-m)$  for some $2\le a\le d-1$ and $2\le m\le \floor{a/2}$ (\cref{lem:dplus2facets}).  We consider two subcases according to the number of $(d-2)$-faces in $F$. 
 
\begin{subcases}
 \subcase \emph{The facet $F$ is a nonsimple polytope.} 
 \label{subcase:3minus1-minus-r-dplus1-nonsimple} 

\cref{lem:lower-bound-dplus2-facets} gives $f_k(F)\ge \tau_{k}(2(d-1)+\ell-r+2,d-1)$ for $k\in [1\ldots d-2]$.  If $\ell-r+2\ge d-4$, then, by \cref{rmk:functions-dplus2-facets}, $f_{k}(F)\ge \tau_{k}(3(d-1)-3,d-1)$.
Since $F$ is a pyramid, its  apex has degree $2d+\ell-r$ in $P$ and is nonpyramidal in $P$ (as $P$ is not a pyramid). Hence, by the maximality of $v_1$, $\deg_P(v_1) \ge 2d + \ell - r$. 

For $k=1$, we have
\begin{align*}
f_{1}(P)&\ge f_{1}(F)+\sum_{i=1}^{r}f_{0}(F_{i}/v_{i})\\
&\ge \tau_{1}(2(d-1)+\ell-r+2,d-1)+2(d-1)+\ell-r+2+\sum_{i=1}^{r-1}d-i.
\end{align*}
Then \cref{lem:combinatorial-ineq}(v) gives that 
\begin{equation} 
\label{eq:3minus1-minus-r-dplus1-subcase2.1-k1}
f_{1}(P) > \eta_{1}(2d+\ell,d).
\end{equation}

Assume $k\in [2\ldots d-2]$. We distinguish  two scenarios based on the number of $(d-2)$-faces in $F_1/v_1=P/v_1$.
 
 \textbf{Suppose the number of $(d-2)$-faces in $F_1/v_1$ is at least $d+2$.} The  induction hypothesis for $F_1/v_1$ and all $\ell\ge 3$ gives that
 \begin{align*}
  f_k(P)&\ge f_k(F)+\sum_{i=1}^{r}f_{k-1}(F_i/v_i)\\
  &\ge f_k(F)+\eta_{k-1}(2(d-1)+\ell-r+2,d-1)+\sum_{i=2}^{r}f_{k-1}(F_i/v_i)\\
  &\ge \tau_{k}(2(d-1)+\ell-r+2,d-1)+\eta_{k-1}(2(d-1)+\ell-r+2,d-1)+\sum_{i=1}^{r-1}\binom{d-i}{k}.
 \end{align*}
If $\ell-r+2\ge d-2$, we use $f_{k-1}(F_1/v_1)\ge \eta_{k-1}(2(d-1)+d-2,d-1)$. In either case, \cref{lem:combinatorial-ineq}(vi) gives 
\begin{equation} 
\label{eq:3minus1-minus-r-dplus1-subcase2.1-dplus2}
f_k(P)> \eta_k(2d+\ell,d).
\end{equation}  

\textbf{Suppose  the number of $(d-2)$-faces in $F_1/v_1$ is exactly $d+1$.} \cref{lem:lower-bound-dplus2-facets} gives that 

\begin{align*}
  f_k(P)&\ge f_k(F)+\sum_{i=1}^{r}f_{k-1}(F_i/v_i)\\
  &\ge f_k(F)+\tau_{k-1}(2(d-1)+\ell-r+2,d-1)+\sum_{i=2}^{r}f_{k-1}(F_i/v_i)\\
  &\ge \tau_{k}(2(d-1)+\ell-r+2,d-1)+\tau_{k-1}(2(d-1)+\ell-r+2,d-1)+\sum_{i=1}^{r-1}\binom{d-i}{k}.
 \end{align*}
 If $\ell-r+2\ge d-4$, then, thanks to \cref{rmk:functions-dplus2-facets}, we use $f_{k}(F)\ge \tau_{k}(3(d-1)-3,d-1)$ and $f_{k}(F_1/v_1)\ge \tau_{k-1}(3(d-1)-3,d-1)$.  \cref{lem:combinatorial-ineq}(vii) ensures that
 \begin{equation} 
   \label{eq:3minus1-minus-r-dplus1-subcase2.1-dplus1}
   f_k(P) > \eta_k(2d+\ell,d).
 \end{equation}  
 \subcase \emph{The facet $F$ is a simple polytope.} 
 \label{subcase:3minus1-minus-r-dplus1-simple} 

Then $F=T(m)\times T(d-1-m)$  for   $2\le m\le \floor{(d-1)/2}$ (\cref{lem:dplus2facets}). Hence, $f_0(F)=d+m(d-1-m)\ge 3d-6$, so $f_0(P)=d+m(d-1-m)+r\ge 3d-3$, implying $\ell\ge d-3$. The number of $k$-faces in $F$ is exactly $\rho_{k}(d-1,m,d-1)$ by \cref{lem:dplus2facets}.
% Denote by $X$ the set of vertices outside $F$.

 \textbf{Suppose that there is a facet $J_1$ for which at least two vertices lie outside $F\cup J_1$.} Let $t$ be the number of vertices in $J_1\cap X$. Then $1\le t\le \card X-2=r-2$.  Suppose that $J_{1}\cap X=\set{v'_{1},\ldots,v'_{t}}$. By \cref{cor:number-faces-outside-facet-practical}, there are faces $J_1,J_2\ldots, J_{t}$ in $J_{1}$ such that each $J_i$ has dimension $d-1-i+1$ and contains $v'_i$ but does not contain any $v'_j$ with $j<i$.  The number of $k$-faces of $J_{1}$ containing at least one vertex in $\set{v'_{1},\ldots,v'_{t}}$ is bounded below by $\sum_{i=1}^{t}f_{k-1}(J_{i}/v'_{i})$. 
There are $\card X-t=r-t\ge 2$ vertices outside $F\cup J_{1}$, and so
\begin{equation}
\label{eq:2d+2-dplus3-2}
\begin{aligned}
f_{k}(P)&\ge f_{k}(F)+\sum_{i=1}^{t}f_{k-1}(J_{i}/v'_{i})+\sum_{j=1}^{r-t}\binom{d+1-j}{k}\\
&\ge \rho_{k}(d-1,m,d-1)+\sum_{i=1}^{t}f_{k-1}(J_{i}/v'_{i})+\sum_{j=1}^{r-t}\binom{d+1-j}{k}.
\end{aligned}
\end{equation}

\subsubsection*{Consider the case $2\le t\le r-2$} When $m\ge 3$, we have $d\ge 7$ and  $\rho_{k}(d-1,m,d-1)\ge  \rho_{k}(d-1,3,d-1)$ by \cref{lem:dplus2-vertices-inequalities}. Thus, for $\ell\ge d-3$ and $2\le t\le r-2$, \eqref{eq:2d+2-dplus3-2} becomes
\begin{equation}
\label{eq:claim3_subcase2.2-large-t-m3}
\begin{aligned}
  f_{k}(P)&\ge \rho_{k}(d-1,m,d-1)+\sum_{i=1}^{t}\binom{d-i}{k}+\sum_{j=1}^{r-t}\binom{d+1-j}{k}\\
  &\ge \brac*{\binom{d}{k+1}+\binom{d-1}{k+1}+\binom{d-2}{k+1}+\binom{d-3}{k+1}-\binom{4}{k+2}}+\sum_{i=1}^{t}\binom{d-i}{k}\\
  &\quad +\sum_{j=1}^{r-t}\binom{d+1-j}{k}\\
  &\ge \eta_{k}(2d+\ell,d)+\binom{d-\ell}{k+1}+\binom{d-3}{k+1}+\sum_{i=1}^{t-2}\binom{d-2-i}{k}+\sum_{j=1}^{r-t-2}\binom{d-1-j}{k}\\
  &\quad-\binom{4}{k+2}\\
 % &= \eta_{k}(2d+\ell,d)+\underbrace{\brac*{\binom{d-\ell}{k+1}+\binom{d-3}{k+1}-\binom{4}{k+2}}}_{\text{$\ge  0$}}+\underbrace{\brac*{\binom{d-2}{k+1}-\binom{d-t}{k+1}}}_{\text{$\ge 0$ as $t\ge 2$}}\\
 % &\quad+\underbrace{\brac*{\binom{d-1}{k+1}-\binom{d-1+t-r+2}{k+1}}}_{\text{$\ge 0$ as $t\le r-2$ \blue{(\text{it seems to me that the $\sum$ expression for this is more clear})}}} \\
  &\ge \eta_{k}(2d+\ell,d).
\end{aligned}
\end{equation}
Equality can occur for $k\in [1\ldots d-2]$ only when $k=d-3,d-2$. For $k=d-2$, equality arises only when $r-t=2$;  for $k=d-3$  only when $t=2$ and $r=4$.

When $m=2$, then $d\ge 5$ and $F=T(2)\times T(d-3)$,  so $f_0(F)=3(d-1)-3$ and $f_0(P)=3d-6+r$. Then, for $\ell= d-6+r\ge d-3$ and $2\le t\le r-2$, 
\begin{equation}
    \label{eq:claim3_subcase2.2-large-t-m2}
\begin{aligned}
  f_{k}(P)&\ge \rho_{k}(d-1,2,d-1)+\sum_{i=1}^{t}\binom{d-i}{k}+\sum_{j=1}^{r-t}\binom{d+1-j}{k}\\
  &= \brac*{\binom{d}{k+1}+\binom{d-1}{k+1}+\binom{d-2}{k+1}-\binom{3}{k+2}}+\sum_{i=1}^{t}\binom{d-i}{k}+\sum_{j=1}^{r-t}\binom{d+1-j}{k}\\
  &= \eta_{k}(3d-6+r,d)+\binom{6-r}{k+1}+\sum_{i=1}^{t-2}\binom{d-2-i}{k}+\sum_{j=1}^{r-t-2}\binom{d-1-j}{k}-\binom{3}{k+2}\\
%  &= \eta_{k}(3d-6+r,d)+\underbrace{\brac*{\binom{6-r}{k+1}-\binom{3}{k+2}}}_{\text{$\ge -1$ \blue{ with equality only for $k=1$ when $r\ge 5$}}}+\underbrace{\brac*{\binom{d-2}{k+1}-\binom{d-t}{k+1}}}_{\text{$\ge 0$ as $t\ge 2$ \blue{equality can occur for $k=d-2$ or any $k\le d-3$ only if $t=2$}}}\\
%  &\quad+\underbrace{\brac*{\binom{d-1}{k+1}-\binom{d-1+t-r+2}{k+1}}}_{\text{$\ge 0$ as $t\le r-2$ }}\\
  &\ge \eta_{k}(3d-6+r,d).
\end{aligned}
\end{equation}
The term $\binom{6-r}{k+1}-\binom{3}{k+2}\ge -1$, with equality only for $k=1$ and $r\ge 5$. If   $k=1$ and $r\ge 5$, then  the sums $\sum_{i=1}^{t-2}\binom{d-2-i}{k}$ and $\sum_{j=1}^{r-t-2}\binom{d-1-j}{k}$ cannot both be equal to zero. Thus  $f_k(P)\ge  \eta_{k}(3d-6+r,d)$.

Equality can occur for $k\in [1\ldots d-3]$ if and only if  $t=2$ and $r=4$. For $k=d-2$, equality arises only when $r-t=2$.

\subsubsection*{Now consider the case $t=1$} Then $J_1$ is a pyramid over a $(d-2)$-face $R$.  By    \cref{rmk:simple-dplus1-facet-faces}(i),  any $(d-2)$-face of $F$ has at least $2d-4$ for $d\ge 5$. Thus, the apex of $J_1$ has degree at least $2d-4$ in $J_1$, so $f_{k-1}(J_1/v_1')\ge \theta_{k-1}(2(d-2),d-2)$. Therefore, \eqref{eq:2d+2-dplus3-2} becomes 
\begin{equation}\label{eq:claim3_subcase2.2_2}
  f_{k}(P)\ge f_{k}(F)+\theta_{k-1}(2(d-2),d-2)+\sum_{j=1}^{r-1}\binom{d+1-j}{k}.
  \end{equation}
  
When $m\ge 3$, we again have $d\ge 7$ and $\rho_{k}(d-1,m,d-1)\ge  \rho_{k}(d-1,3,d-1)$. Thus, for $\ell\ge d-3$,  
  \begin{align*}
    f_{k}(P)&\ge \rho_{k}(d-1,3,d-1)+\theta_{k-1}(2(d-2),d-2)+\sum_{j=1}^{r-1}\binom{d+1-j}{k}\\
    &\ge \brac*{\binom{d}{k+1}+\binom{d-1}{k+1}+\binom{d-2}{k+1}+\binom{d-3}{k+1}-\binom{4}{k+2}}\\
    &\quad + \brac*{\binom{d-1}{k}+\binom{d-2}{k}-\binom{1}{k}}+\sum_{j=1}^{r-1}\binom{d+1-j}{k}\\
    &\ge \eta_{k}(2d+\ell,d)+\binom{d-\ell}{k+1}+\binom{d-3}{k+1}+\sum_{j=1}^{r-3}\binom{d-1-j}{k}-\binom{4}{k+2}-\binom{1}{k}\\
    &=\eta_{k}(2d+\ell,d)+\underbrace{\brac*{\binom{d-\ell}{k+1}+\binom{d-3}{k+1}-\binom{4}{k+2}-\binom{1}{k}}}_{\text{$\ge  0$}}+\sum_{j=1}^{r-3}\binom{d-1-j}{k}\\
    &\ge \eta_{k}(2d+\ell,d).
  \end{align*}

  Equality holds for $k\in [1\ldots d-2]$ if and only if  $r=3$ and $k=d-3,d-2$. %\blue{(Equality can occur when $k=d-2$ only when $r=3$; equality can occur when $k=d-3$ only when $r=3$; it is strict inequality for $k\in [1\ldots d-4]$. )}
  
Suppose $m=2$. Then $d\ge 5$ and $F=T(2)\times T(d-3)$, so $f_0(F)=3(d-1)-3$ and $f_0(P)=3d-6+r$. Thus,  for $\ell=d-6+r\ge d-3$,  
  \begin{align*}
    f_{k}(P)&\ge \rho_{k}(d-1,2,d-1)+\theta_{k-1}(2(d-2),d-2)+\sum_{j=1}^{r-1}\binom{d+1-j}{k}\\
    &= \brac*{\binom{d}{k+1}+\binom{d-1}{k+1}+\binom{d-2}{k+1}-\binom{3}{k+2}}+\brac*{\binom{d-1}{k}+\binom{d-2}{k}-\binom{1}{k}}\\
    &\quad +\sum_{j=1}^{r-1}\binom{d+1-j}{k}\\
    &= \eta_{k}(3d-6+r,d)+\binom{6-r}{k+1}+\sum_{j=1}^{r-3}\binom{d-1-j}{k}-\binom{3}{k+2}-\binom{1}{k}\\
    % &=\eta_{k}(3d-6+r,d)+\underbrace{\brac*{\binom{6-r}{k+1}-\binom{3}{k+2}-\binom{1}{k}}}_{\text{$\ge  -2$ \blue{(it's a bit confusing here to use this to see the ineq )}}}+\underbrace{\brac*{\binom{d-1}{k+1}-\binom{d-r+2}{k+1}}}_{\text{$\ge 0$ as $r\ge 3$}}\\
    &\ge \eta_{k}(3d-6+r,d).
  \end{align*} 
When $r=3$, the term $\binom{6-r}{k+1}-\binom{3}{k+2}-\binom{1}{k}\ge 0$ for all $k\ge 1$. When $r\ge 4$, we have $\binom{6-r}{k+1}-\binom{3}{k+2}-\binom{1}{k}+\sum_{j=1}^{r-3}\binom{d-1-j}{k}\ge \binom{d-2}{k}-\binom{3}{k+2}-\binom{1}{k}>0$ for $k\ge 1$.

  Equality holds for $k\in [1\ldots d-2]$ if and only if  $r=3$ and $k\ge 3$. 
  %\blue{In more detail, in the second last equation, when $r=3$, $\binom{6-r}{k+1}-\binom{3}{k+2}-\binom{1}{k}+\binom{d-1}{k+1}-\binom{d-r+2}{k+1}$ is 1 for $k=1, 2$ and is 0 for $k\ge 3$; when $r\ge 4$, }

  \textbf{Suppose that there is no facet $J_1$ for which at least two vertices lie outside $F\cup J_1$.} Since $P$ has at least $d+3$ facets and $F$ has $d+1$, there must exist a facet $U_1$ of $P$ intersecting $F$ in a face of dimension at most $d-3$. Then $U_1$ contains either every vertex in $X$ or all except one. Thus, $f_0(U_1)\ge r-1+f_0(U_1\cap F)$. 

Suppose for contradiction that $U_1$ is disjoint from $F$.  Since every facet other than $F$ misses at most one vertex from $X$, the facet $U_1$ must be a pyramid over each of its $(d-2)$-faces; otherwise, another facet of $P$ containing a $(d-2)$-face of $U_1$ would miss at least two vertices from $X$.  By \cref{prop:polytope-pyramid-all-facets}, $U_1$ is a simplex in this case. However, one of the facets in $P$ containing a $(d-2)$-face of $F$, say $I$, must intersect $U_1$ in a  face of dimension at most $d-3$, since $F$ has $d+1$ $(d-2)$-faces while $U_1$ has only $d$. Then $I$ avoids at least two vertices from $X$, which is a contradiction. Hence, every  such facet $U_1$ of $P$ has a nonempty intersection with $F$.

It follows that $U_1$ has two $(d-2)$-faces  containing $F\cap U_1$. Each such $(d-2)$-face lies in a facet $J$ of $P$ other than $U_1$ and $F$, and such a facet cannot miss two vertices of $U_1 \setminus F$; hence $U_1 \setminus J$ has only one vertex. Thus $U_1$ is a two-fold pyramid. More generally, we have the following.

\begin{fact} Every facet of $P$ intersecting $F$ in a face of dimension at most $d-3$ is a  two-fold pyramid. 
 \label{fact:two-fold-facet}   
\end{fact}
% \red{Every vertex in $F_1\setminus F$ must be pyramidal; otherwise a facet intersecting $F_1$ and missing a nonpyramidal vertex from $F_1\setminus F$ would avoid at least two  vertices from $X$, contrary to our assumption. Thus, $F_1$  is a $\card(F_1\setminus F)$-fold pyramid over $F_1\cap F$. } 
%The same applies to every facet intersecting $F$ at a face whose dimension is less than $d-2$.

\textbf{Suppose that $r\ge 4$}.  We claim that  $P$ contains a vertex of degree at least $d+2$.  If $F\cap U_1$ has dimension less than $d-3$, then  each vertex in $F\cap U_1$ has degree at least $d+2$ in the polytope. If instead $F\cap U_1$ is a $(d-3)$-face, then $U_1$ is a two-fold pyramid over $F\cap U_1$; but this is a contradiction, since $2=\card (U_1\setminus F)\ge r-1\ge 3$. Therefore,  $\dim (F\cap U_1)\le d-4$, and  every vertex of  $F\cap U_1$ has degree at least $d+2$ in $P$;  in particular, $\deg(v_1)\ge d+2$.

%\green{Since $F_1$ is a $\card(F_1\setminus F)$-fold pyramid, we find  that $\dim (F\cap F_1)=d-1-\card(F_1\setminus F)=d-r,d-1-r\le d-4$. Hence, each vertex of  $F\cap F_1$ would have degree at least $d+2$ in the polytope, and in particular, $\deg(v_1)\ge d+2$.}

% Suppose, to the contrary, that every vertex has degree at most $d+1$. Then the  facet $F_1$ must intersect $F$ at  $(d-3)$-face $K$; otherwise contradicting the assumption.
% However, in this scenario, 

% If $f_0(K)\ge 2d-6$, then the apex of $F_1$ has degree at least $2d-6+r-2$ in $F_1$, and thus the degree of $v_1$ is at least $2d-3\ge d+2$ in $P$, a contradiction. Thus, by \cref{rmk:simple-dplus1-facet-faces}, $f_0(K)=d-2$ and $m=2$. The facet $T(2)\times T(d-3)$ contains precisely three $(d-3)$-simplices, and each such $(d-3)$-simplex is disjoint from a $(d-2)$-face $R$ with $2d-4$ vertices.  It follows that $f_0(F)=3d-6$, $f_0(P)=3d-3$, and $\ell=d-3$. If $d=5$, then $f_0(P)=2d+2$, respectively, a case that has been covered. Thus, we assume that $d\ge 6$.  The other facet $Y$ containing $R$ is either a pyramid with apex 
% $v'_1\in X\setminus \V(F_1)$ or contains a vertex in  $X\cap\V(F_1)$, say $v'_2$.  In the former scenario, $\deg_P(v'_1)\ge \deg_P(v_1)\ge 2d-3\ge d+3$, a contradiction. In the latter scenario, we must have that $\deg_P(v'_2)\ge d-2+r-2+d-1-(r-2)=2d-3$, implying that $\deg_P(v_1)\ge 2d-3\ge d+3$, a final contradiction.

When $m\ge 3$, we have that $d\ge 7$ and $\rho_{k}(d-1,m,d-1)\ge  \rho_{k}(d-1,3,d-1)$ by \cref{lem:dplus2-vertices-inequalities}. In this case, for $\ell\ge d-3$, we find 
\begin{equation}
\label{eq:claim3_subcase2.nofacetJ1-r4}
\begin{aligned}
    f_{k}(P)&\ge \rho_{k}(d-1,3,d-1)+\theta_{k-1}(d+2,d-1)+\sum_{j=1}^{r-1}\binom{d-j}{k}\\
    &\ge \brac*{\binom{d}{k+1}+\binom{d-1}{k+1}+\binom{d-2}{k+1}+\binom{d-3}{k+1}-\binom{4}{k+2}}\\
    &\quad + \brac*{\binom{d}{k}+\binom{d-1}{k}-\binom{d-3}{k}}+\sum_{j=1}^{r-1}\binom{d-j}{k}\\
    &\ge \eta_{k}(2d+\ell,d)+\binom{d-\ell}{k+1}+\binom{d-3}{k+1}+\sum_{j=1}^{r-4}\binom{d-3-j}{k}-\binom{4}{k+2}\\
    % &= \eta_{k}(2d+\ell,d)+\underbrace{\brac*{\binom{d-\ell}{k+1}+\binom{d-3}{k+1}-\binom{4}{k+2}}}_{\text{$\ge  0$}}+\underbrace{\brac*{\binom{d-3}{k+1}-\binom{d+1-r}{k+1}}}_{\text{$\ge 0$ as $r\ge 4$}}\\
    &\ge \eta_{k}(2d+\ell,d).
  \end{aligned}
\end{equation}
  Equality holds for $k\in [1\ldots d-2]$ if and only if  $k=d-3,d-2$.
  
When $m=2$, we have $d\ge 5$ and $F=T(2)\times T(d-3)$, implying $f_0(F)=3(d-1)-3$ and $f_0(P)=3d-6+r$. In this case,  for $\ell=d-6+r\ge d-3$, we obtain 
\begin{equation}\label{eq:claim3_subcase2.2_3}
\begin{aligned}
  f_{k}(P)&\ge \rho_{k}(d-1,2,d-1)+\theta_{k-1}(d+2,d-1)+\sum_{j=1}^{r-1}\binom{d-j}{k}\\
  &= \brac*{\binom{d}{k+1}+\binom{d-1}{k+1}+\binom{d-2}{k+1}-\binom{3}{k+2}}+\brac*{\binom{d}{k}+\binom{d-1}{k}-\binom{d-3}{k}}\\
  &\quad +\sum_{j=1}^{r-1}\binom{d-j}{k}\\
  &= \eta_{k}(3d-6+r,d)+\binom{6-r}{k+1}+\sum_{j=1}^{r-4}\binom{d-3-j}{k}-\binom{3}{k+2}\\
  &=\eta_{k}(3d-6+r,d)+\underbrace{\brac*{\binom{6-r}{k+1}-\binom{3}{k+2}}}_{\text{$\ge  -1$}}+\underbrace{\brac*{\binom{d-3}{k+1}-\binom{d+1-r}{k+1}}}_{\text{$\ge 0$ as $r\ge 4$}}\\
  &\ge \eta_{k}(3d-6+r,d).
  \end{aligned}
\end{equation} 
  Equality holds for $k\in [1\ldots d-2]$ if and only if either $r=4$, or $r\ge 5$ and $k=d-3,d-2$. For $d=5$, there is an additional equality pair $(r,k)=(5,1)$.

%If $r\ge 5$, then $\sum_{j=1}^{r-4}\binom{d-3-j}{k}-\binom{3}{k+2}\ge 0$ since $d\ge 5$. If $r=4$, then $\binom{6-r}{k+1}-\binom{3}{k+2}=0$. Hence, this scenario is settled.  

Finally, \textbf{suppose that $r=3$}. Let  $I_1$ be a facet of $P$, other than $F$, that does not contain $v_1$. Then $I_1$ must contain the other two vertices $v_2$ and $v_3$ from $X$, and it must intersect $F$ in either a ridge of $P$ or a $(d-3)$-face. If $I_1\cap F$ is a ridge, then each of $v_2$ and $v_3$ serves as the apex of a pyramid over a $(d-3)$-face $K\subseteq I_1\cap F$. If instead  $I_1\cap F$ is a $(d-3)$-face $K$, then $I_1$ is a two-fold pyramid over $K$ (see \cref{fact:two-fold-facet}). 

If $m\ge 3$, then $d\ge 7$, and according to \cref{rmk:simple-dplus1-facet-faces}, $f_0(K)\ge 2d-6$. In this case, $v_2$ or $v_3$ has degree at least $2d-5$ in $I_1$, so $\deg_P(v_1)\ge 2d-4$, and the following computation applies:
\begin{equation}\label{eq:claim3_subcase2.2_4}
\begin{aligned}
 f_{k}(P)&\ge \rho_{k}(d-1,3,d-1)+\theta_{k-1}(2d-4,d-1)+\theta_{k-1}(2d-5,d-2)+\binom{d-2}{k}\\
  &\ge \brac*{\binom{d}{k+1}+\binom{d-1}{k+1}+\binom{d-2}{k+1}+\binom{d-3}{k+1}-\binom{4}{k+2}}\\
  &\quad + \brac*{\binom{d}{k}+\binom{d-1}{k}-\binom{3}{k}}+\brac*{\binom{d-1}{k}+\binom{d-2}{k}-\binom{2}{k}} +\binom{d-2}{k}\\
  &\ge \eta_{k}(2d+\ell,d)+\binom{d-\ell}{k+1}+\binom{d-3}{k+1}+\binom{d-2}{k}-\binom{4}{k+2}-\binom{3}{k}-\binom{2}{k}\\
  &>\eta_{k}(2d+\ell,d).    
\end{aligned} 
\end{equation}

Now suppose  $m=2$, and  let $X=\set{v_1',v_2', v_3'}$, where $v_1$ could be any of these vertices. Then $d\ge 5$ and $F=T(2)\times T(d-3)$, giving $f_0(F)=3d-6$, $f_0(P)=3d-3$, and $\ell=d-3$. The $d=5$ corresponds to $f_0(P)=2d+2$,  which  has been addressed elsewhere, so assume $d\ge 6$. For convenience, we list all the $(d-2)$-,  $(d-3)$-, and $(d-4)$-faces of $F$ (see \cref{rmk:simple-dplus1-facet-faces}).
\begin{fact}
\label{fact:faces-F}  
\begin{enumerate}[{\rm (i)}]
      \item A $(d-2)$-face of $F$ is  either $T(1)\times T(d-3)$ with $2(d-2)$ vertices or $T(2)\times T(d-4)$ with $3(d-3)$ vertices. 
  
 \item A $(d-3)$-face of $F$ is  $T(0)\times T(d-3)$  with $d-2$ vertices, $T(1)\times T(d-4)$ with $2(d-3)$ vertices, or $T(2)\times T(d-5)$ with $3(d-4)$ vertices.

 \item A $(d-4)$-face of $F$ is a $(d-4)$-simplex with $d-3$ vertices, $T(1)\times T(d-5)$ with $2(d-4)$ vertices, or $T(2)\times T(d-6)$ with $3(d-5)$ vertices.
\end{enumerate}
\end{fact}

If all the vertices of $F$ are simple in $P$, then the number of edges between the vertices in $F$ and those in $X$ would be exactly $3(d-2)$, implying that $P$ is simple, contrary to our assumption. Hence,  $F$ must contain a nonsimple vertex $w$ in $P$.

Let $Y$ be a facet of $P$ that contains a $(d-2)$-face $T(2)\times T(d-4)$ $R$ of $F$ and precisely two vertices from $X$, say $v_1'$ and $v_2'$. According to \cref{rmk:simple-dplus1-facet-faces}, the $(d-3)$-faces within $R$ are either  $T(1)\times T(d-4)$ with $2(d-3)$ vertices or $T(2)\times T(d-5)$ with $3(d-4)$ vertices, where $3(d-4)\ge 2(d-3)$ for $d\ge 6$. By considering a $(d-3)$-face $K$ within $R$ that contains $v_1'$ and not $v_2'$, we get a $(d-2)$-face $R'$ of $Y$ that is a pyramid with apex $v_1'$ and base $K$. In the other facet $Z$ of $P$ that contains $R'$, we have that $v_2'\notin Z$ and $\deg_Z(v_1')\ge 2(d-3)+1$. The same reasoning gives that  $\deg_P(v_2')\ge 2(d-3)+2$. Consequently,  for $d\ge 6$, 
\begin{equation}
\label{eq:claim3_r3_m2}    
\begin{aligned}
  f_{k}(P)&\ge \rho_{k}(d-1,2,d-1)+f_{k-1}(P/v'_2)+ f_{k-1}(Z/v'_1) +\binom{d-2}{k}\\
  & \ge \rho_{k}(d-1,2,d-1)
  +\theta_{k-1}(2d-4,d-1)+\theta_{k-1}(2d-5,d-2)+\binom{d-2}{k}\\
  &\ge \brac*{\binom{d}{k+1}+\binom{d-1}{k+1}+\binom{d-2}{k+1}-\binom{3}{k+2}}\\
  &\quad + \brac*{\binom{d}{k}+\binom{d-1}{k}-\binom{3}{k}}+\brac*{\binom{d-1}{k}+\binom{d-2}{k}-\binom{2}{k}} +\binom{d-2}{k}\\
  &\ge \eta_{k}(3d-3,d)+\binom{3}{k+1}+\binom{d-2}{k}-\binom{3}{k+2}-\binom{3}{k}-\binom{2}{k}> \eta_{k}(3d-3,d).
\end{aligned}
\end{equation}

\textbf{Now assume that that every facet (other than $F$) containing a $(d-2)$-face $T(2)\times T(d-4)$ of $F$ also contains all the three vertices from $X$.} Let $R_w$ be a $(d-2)$-face of this form that does not contain the nonsimple vertex $w$ of $F$, and let $F_w$ be the other facet of $P$ containing $R_w$. Then $X\subset F_w$, and $f_0(F_w)= 3(d-3)+3> 2d-1$ for $d\ge 6$. The situation of $F_{w}$ having at least $d+2$ facets has already been treated in Case \ref{case:3minus1-minus-r-dplus2} of \cref{cl:3minus1-more-2dplus1}. So assume $F_{w}$ has exactly $d+1$ facets. Since $F_{w}$ is not a pyramid---none of the vertices in $X$ or $R_w$ can be apices---it must be a simple polytope (\cref{lem:dplus2facets}). Since $F_w$ has $3(d-2)$ vertices, it must be $T(2)\times T(d-3)$. Hence, every vertex in $R_w=F_w\cap F$ is simple in $P$.

Now consider a facet $U_1$ intersecting $F$ in a face $K_1$ with $\dim(K_1)\le d-3$. Then  $U_1$ is a two-fold pyramid (see \cref{fact:two-fold-facet}). The facet $U_1$ contains either two or all three vertices from $X$, and $K_1$ is either a $(d-3)$-face  or a $(d-4)$-face.

%By \cref{fact:faces-F}, $K_1$ is either a $(d-3)$-simplex, $T(1)\times T(d-4)$, or $T(2)\times T(d-5)$.
Suppose first that $K_1$ is a $(d-3)$-face.  Then $f_0(K_1)\ge d-2\ge 4$ for $d\ge 6$. Since $U_1$ is a two-fold pyramid over $K_1$, $\card{(F\setminus R_w)}=3$, and $d\ge 6$, we find that $K_1\cap R_w\ne \emptyset$, so the vertices in  $K_1\cap R_w$ must be adjacent to the two apices of $U_1$, which lie outside $F$. But the vertices in $F_w\cap F$ are all simple in $P$, a contradiction. 

Finally, assume that $K_1$ is a $(d-4)$-face. By \cref{fact:faces-F}, $K_1$ is a $(d-4)$-simplex, $T(1)\times T(d-5)$, or $T(2)\times T(d-6)$. If $K_1$ contains any vertex  from $R_w$, then the same adjacency argument as above leads to a contradiction. So we must have   $K_1=T(2)\times T(0)$, which forces  $d=6$, $X\subset \V(U_1)$, and $U_1$ to be a 3-fold pyramid over the triangle $K_1$. Since $K_1$ is a $(d-4)$-face (a 2-simplex), there must exist a facet $F_2\ne F,U_1$ of $P$ that contains $K_1$ and intersects $F$ in a $(d-2)$-face $T(2)\times T(2)$ (\cref{fact:faces-F}). By assumption, $F_2$ must contain all the three vertices in $X$  and hence that $U_1\subset F_2$, a contradiction.
%If $F_2\cap F=K_1$, then $U_1= F_2$, a contradiction. If $F_2\cap F$ is a $(d-2)$-face $T(2)\times T(2)$ (\cref{fact:faces-F}), then $F_2$ must contain all the three vertices in $X$ by assumption, and hence that $U_1\subset F_2$, again a contradiction.  The case $F_2\cap F=T(1)\times T(3)$ (\cref{fact:faces-F}) is not possible,  as there is no such face in $F$ that contains $K_1$ \blue{(I can't see this now....why...$K_1$ is a $T(2)$... )}. 
%The only remaining case is when $F_2\cap F$ is a $(d-3)$-face. However, the case of a facet intersecting $F$ in a $(d-3)$-face has been already ruled out. 
% \blue{(Perhaps this part of the argument can be simplified. Once we have ``...$d=6$, and $U_1$ is a 3-fold pyramid over the triangle $K_1$." We can consider a $4$-face containing $K_1$ which is of the form $T(2)\times T(2)$, by assumption, the other facet containing this $4$-face must contain all the three vertices in $X$, however this would imply that this facet contains $U_1$, which is not possible.)}
\end{subcases}
\end{claimproof} 

\subsection{Claim  \ref{cl:3minus1-less-2dplus1}: A facet with between $d$ and $2d-2$ vertices} We begin with an overview of the proof. By \cref{lem:2d+s-dplus3-simple-pyramid}, we may assume that $P$ contains a vertex that is both nonpyramidal and nonsimple; let  $v_1$ be one such vertex of maximum degree. Let $F$ be a facet of $P$ not containing $v_1$. By  \cref{cl:3minus1-pyramid,cl:3minus1-minus-2,cl:3minus1-more-2dplus1}, we may assume that $F$ has at most $2d-2$ vertices.  

We bound $f_k(F)$ using \cref{thm:at-most-2d} if $F$ has $d$ or $d+1$ facets, or \cref{thm:at-most-2d-refined-short} if it has at least $d+2$ facets. The analysis from \cref{cl:3minus1-more-2dplus1} to bound the faces containing the vertices outside $F$ is insufficient to settle this claim.   We therefore resort to a shelling argument inspired by Blind and Blind~\cite{BliBli99}, coupled with \cref{cor:number-faces-outside-facet-practical}.

Let  $\mathcal A(v_1)$ be the antistar  of $v_1$ in the boundary complex of $P$. There is a (line) shelling $S=F_0\ldots F_m$ of $\mathcal A(v_1)$ that begins with the facet $F$ (that is,  $F_0=F$), followed by the other facets in $\mathcal A(v_1)$.

Since each facet in $\mathcal A(v_1)$ has at most $2d-2$ vertices, for any $j\ge 1$  the set $F_{j}\setminus (\cup_{i<j} F_i)$ contains at most $d-1$ vertices. The main idea is to apply \cref{cor:number-faces-outside-facet-practical} to these `new' vertices at each step.  We distinguish two cases according to the size of $F\cup F_1$. In  \cref{case:3minus1-small-facet-two}, $F\cup F_1$ does not contain all vertices of $P$ except $v_1$; in   \cref{case:3minus1-small-facet-one}, $F \cup F_1$ does. We now proceed with the details.

\begin{claim} Let $P$ be a $d$-polytope with a vertex $v_1$ that is neither pyramidal nor simple and has the maximum degree among such vertices. Let $F$ be a facet with $d-1+p$ vertices and not containing $v_1$, where $1\le p\le d-1$. Then, for each $k\in [1\ldots d-2]$, it holds that $f_k(P)\ge \eta_k(2d+\ell,d)$. 
  \label{cl:3minus1-less-2dplus1}
\end{claim}
\begin{claimproof}
  By  \cref{cl:3minus1-pyramid,cl:3minus1-minus-2,cl:3minus1-more-2dplus1}, we assume that  all the facets of $P$ not containing $v_1$ have at most $2d-2$ vertices.  %We consider two cases according to the number of $(d-2)$-faces in $F$. %Let $q:=d+\ell+1-p$ and let $X:=\set{v_1,\ldots,v_{q}}$ be the set of vertices outside $F$. 

  % The number of $k$-faces in $F$ is at least $\theta_k(d-1+p,d-1)$ (\cref{cor:more-dpluss}). 
  
 % Let $P^*$ be the dual polytope of $P$ and let $\psi$ be the antiisomorphism from the face lattice of $P$ to that of $P^*$. %As the removal of a facet does not disconnect the vertices outside the facet \cite[Thm.~4.1.6]{Pin24}, there is a connected  subgraph of $G(P^*)$ that includes the vertices in $P^*$ outside $\psi(v_1)$. This means that   facet-ridge path $F_1, \dots, F_m$ in $P$ where $F_i \cap F_{i+1}$ is a ridge of $P$, and the union of these facets gives all the facets that do not contain $v_1$, and $m\ge 3$ by our assumption.
  The \emph{antistar} $\mathcal A(v_1)$ of $v_1$ in the boundary complex of $P$ consists of all faces of $P$ disjoint from $v_1$ and forms a strongly connected $(d-1)$-complex \cite[Prop.~3.3.1]{Pin24}. (A polytopal complex is \emph{pure} if every face is contained in some facet, while a pure complex $\mathcal{C}$ is \emph{strongly connected} if every two facets $J$ and $J'$ are connected by a path $J_1\ldots J_n$ of facets in $\mathcal{C}$ such that $J_i\cap J_{i+1 }$ is a ridge of $\mathcal C$ for $i\in [1\ldots n-1]$, $J_1 = J$ , and $J_n = J'$.) 

  There exists a shelling of $P$ in which the facets containing $v_1$ appear first, and the facet $F\in \mathcal A(v_1)$ appears last \cite[Prop.~2.12.13]{Pin24}. Reversing this shelling yields a shelling $S:=F_0\ldots F_m$ of $\mathcal A(v_1)$ that begins with the facet $F=F_0$, followed by the other facets in $\mathcal A(v_1)$. (A \emph{shelling} of  a pure polytopal complex $\mathcal{C}$ is a linear ordering $J_1,\ldots,J_s$ of its facets such that either $\dim \mathcal C = 0$ or it satisfies the following: (i) The boundary complex of $J_1$ has a shelling, and (ii) for each $j\in [2\ldots s ]$, the intersection $J_j\cap \paren*{\cup_{i=1}^{j-1} J_i}$ is nonempty and forms the beginning of a shelling of the boundary complex of $J_j$.)

Each facet in $\mathcal A(v_1)$ has at most $2(d-1)$ vertices, so for any facet $F_{j}$  in  $S$, the set $F_{j}\setminus (\cup_{i=1}^{j-1} F_i)$ contains at most $d-1$ vertices.  See \cite[Sec.~2.12]{Pin24} for details on shellings. Let $v'_{11},\ldots v'_{1a_1}$ be the vertices in $F_{1}\setminus F_{0}$. Then  $1\le a_1 \le d-1$.  By \cref{cor:number-faces-outside-facet-practical}, there are faces $J_{11},\ldots, J_{1a_1}$ in $F_{1}$ such that each $J_{1i}$ has dimension $d-1-i+1$ and contains $v'_{1i}$ but does not contain any $v'_{1j}$ with $j<i$. Therefore, the number of $k$-faces of $F_{1}$ that contain at least one of the vertices in $\set{v'_{11},\ldots,v'_{1a_1}}$ is bounded  below by
\begin{align*}
  \sum_{i=1}^{a_1}f_{k-1}(J_{1i}/v'_{1i})&\ge \sum_{i=1}^{a_1} \binom{d-i}{k}.
  \end{align*}
  We  consider the facet $F_{2}$ in the shelling. Let  $a_2$ be the number of vertices in $F_{2}\setminus (F_{0}\cup F_{1})$, so $0\le a_2 \le d-1$. We repeat the above argument to bound  the number of $k$-faces within $F_2$ containing at least one of  these $a_2$ vertices. We continue  this process until the facet $F_{m}$, where $F_m\setminus (F_0\cup \cdots \cup F_{m-1})$ contains $a_m$ vertices, again with $0\le a_m\le d-1$. In the $i$th step, we count only $k$-faces that include at least one of the  vertices in $F_i\setminus (F_0\cup \cdots \cup F_{i-1})$, so there is no double counting. Also, every vertex in $P$ other than $v_1$ is in $\mathcal A(v_1)$.  It follows that 
  \begin{equation}\label{eq:antistar-number-vertices}
    \sum_{i=1}^m a_i=d+\ell-p, 
  \end{equation}
and
  \begin{equation}\label{eq:3minus1-less-2dplus1-1}
      \begin{aligned}
        f_k(P) & \ge f_k(F)+f_{k-1}(P/v_1)+\sum_{i=1}^{a_1}\binom{d-i}{k}+\dots+ \sum_{i=1}^{a_{m}}\binom{d-i}{k}.
     \end{aligned}
     \end{equation}
Two cases arise, depending on the number of positive numbers in $\set{a_1,\ldots,a_m}$.

\case There are at least two positive numbers in $\set{a_1,\ldots,a_m}$.
\label{case:3minus1-small-facet-two}

We have $f_k(F)\ge \theta_k(d-1+p,d-1)$ by \cref{thm:at-most-2d}. Since $v_1$ is nonsimple,  it lies in at least $\theta_{k-1}(d+1,d-1)$ $k$-faces. The following estimate holds:

% Since $\sum_{i=1}^{a_1}\binom{d-i}{k}+\sum_{i=1}^{a_2}\binom{d-i}{k}+\dots+ \sum_{i=1}^{a_{m}}\binom{d-i}{k}\ge \sum_{i=1}^{b_1}\binom{d-i}{k}+\sum_{i=1}^{b_2}\binom{d-i}{k}$ for $\sum_{i=1}^m a_i=b_1+b_2=d+\ell-p=(d-1)q+c_1+c_2$  and $q\ge 0$, $1\le c_1,c_2\le d-1$,  the scenario in \eqref{eq:3minus1-less-2dplus1-1} where there are exactly two positive numbers in $\set{a_1,\ldots,a_m}$ yields a lower bound:
% \begin{equation}
% \label{eq:3minus1-less-2dplus1-2}
%     \begin{aligned}
%       f_k(P)&  \ge \theta_k(d-1+p,d-1)+\theta_{k-1}(d+1,d-1)+\sum_{i=1}^{b_1}\binom{d-i}{k}+\sum_{i=1}^{b_2}\binom{d-i}{k}.
% \end{aligned}
% \end{equation}
We consider three subcases depending on the value of $\ell - p + 1$.

% \begin{subcases}
%  \subcase \emph{$1\le \ell-p+1< d-1$.} 
%  \label{subcase:3minus1-small-facet-two-1} 
 
 \textbf{Suppose $1\le \ell-p+1< d-1$.} Then $d+\ell-p> d-1$, and using \eqref{eq:antistar-number-vertices} and \eqref{eq:3minus1-less-2dplus1-1} we obtain 
\begin{equation}
\label{eq:3minus1-less-2dplus1-case1-part1}
    \begin{aligned}
     f_k(P) & \ge \theta_k(d-1+p,d-1)+\theta_{k-1}(d+1,d-1)+\sum_{i=1}^{d-1}\binom{d-i}{k}+\sum_{i=1}^{\ell-p+1}\binom{d-i}{k} \\
      & =\binom{d+1}{k+1}+\binom{d}{k+1}-\binom{d-p}{k+1}-\binom{d-2}{k}+\sum_{i=1}^{d-1}\binom{d-i}{k}+\sum_{i=1}^{\ell-p+1}\binom{d-i}{k} \\
      & =\eta_k(2d+\ell,d)-\underbrace{\brac*{\binom{d}{k+1}-\sum_{i=1}^{d-1}\binom{d-i}{k}}}_{=0}-\brac*{\binom{d-p}{k+1}-\binom{d-\ell}{k+1}}-\binom{d-2}{k}\\
      &\quad +\sum_{i=1}^{\ell-p+1}\binom{d-i}{k} \\
      & =\eta_k(2d+\ell,d)-\sum_{i=1}^{\ell-p}\binom{d-p-i}{k}-\binom{d-2}{k}+\sum_{i=1}^{\ell-p+1}\binom{d-i}{k}\\
      & =\eta_k(2d+\ell,d)+\sum_{i=1}^{\ell-p}\brac*{\binom{d-1-i}{k} -\binom{d-p-i}{k}}+ \binom{d-1}{k}-\binom{d-2}{k} \\
      &>  \eta_k(2d+\ell,d). 
      \end{aligned}
  \end{equation}

  \textbf{Suppose $ \ell-p+1\ge  d-1$}. Then $d-\ell\le 1$. Since $\sum_{i=1}^m a_i=d+\ell-p$ (see \eqref{eq:antistar-number-vertices}), we get a simpler computation in \eqref{eq:3minus1-less-2dplus1-1}:
  
\begin{equation}
\label{eq:3minus1-less-2dplus1-case1-part2}
    \begin{aligned}
     f_k(P) & \ge \theta_k(d-1+p,d-1)+\theta_{k-1}(d+1,d-1)+\sum_{i=1}^{d-1}\binom{d-i}{k}+\sum_{i=1}^{d-1}\binom{d-i}{k} \\
      & =\eta_k(2d+\ell,d)-\underbrace{\brac*{\binom{d}{k+1}-\sum_{i=1}^{d-1}\binom{d-i}{k}}}_{=0}-\binom{d-p}{k+1}+\underbrace{\binom{d-\ell}{k+1}}_{=0} -\binom{d-2}{k} \\
      &\quad +\sum_{i=1}^{d-1}\binom{d-i}{k} \\
& =\eta_k(2d+\ell,d)-\binom{d-p}{k+1}-\binom{d-2}{k}+\binom{d}{k+1}\\
      & >  \eta_k(2d+\ell,d). 
      \end{aligned}
  \end{equation}

\textbf{Suppose that $\ell-p+1<1$.} Then $d+\ell-p\le d-1$, and \eqref{eq:antistar-number-vertices} and \eqref{eq:3minus1-less-2dplus1-1} imply
\begin{equation}
\label{eq:3minus1-less-2dplus1-case1-part3}
    \begin{aligned}
     f_k(P) & \ge \theta_k(d-1+p,d-1)+\theta_{k-1}(d+1,d-1)+\sum_{i=1}^{d+\ell-p-1}\binom{d-i}{k}+\binom{d-1}{k} \\
      & =\eta_k(2d+\ell,d)-\binom{d}{k+1}+\brac*{\binom{d-\ell}{k+1}-\binom{d-p}{k+1}}-\binom{d-2}{k}\\
      &\quad +\sum_{i=1}^{d+\ell-p-1}\binom{d-i}{k}+\binom{d-1}{k} \\
      & =\eta_k(2d+\ell,d)-\binom{d-1}{k+1}+\brac*{\binom{d-\ell}{k+1}-\binom{d-p}{k+1}}-\binom{d-2}{k}+\sum_{i=1}^{d+\ell-p-1}\binom{d-i}{k}\\
      & =\eta_k(2d+\ell,d)-\sum_{i=1}^{d-2}\binom{d-1-i}{k}+\sum_{i=1}^{p-\ell}\binom{d-\ell-i}{k}-\binom{d-2}{k}\\
      &\quad +\sum_{i=1}^{d+\ell-p-2}\binom{d-1-i}{k}+\binom{d-1}{k}\\
      & =\eta_k(2d+\ell,d)-\sum_{i=d+\ell-p-1}^{d-2}\binom{d-1-i}{k}+\sum_{i=1}^{p-\ell}\binom{d-\ell-i}{k}+\binom{d-1}{k}-\binom{d-2}{k}\\
      & =\eta_k(2d+\ell,d)+\sum_{i=1}^{p-\ell}\underbrace{\brac*{\binom{d-\ell-i}{k}-\binom{p+1-\ell-i}{k}}}_{\ge 0\; \text{as $d\ge p+1$}} +\binom{d-1}{k}-\binom{d-2}{k}\\
      & >  \eta_k(2d+\ell,d). 
      \end{aligned}
  \end{equation}
\case There is only one positive number in $\{a_1, \dots, a_m\}$.
\label{case:3minus1-small-facet-one}

That is, there is a facet $F_i$ such that $F_i \cup F$ contains all the vertices of $P$ except $v_1$. This forces $i=1$.  We have  $\card(F_1\setminus F)=d+\ell-p \le d-1$; that is, $a_1=d+\ell-p$ and $\ell \le p-1$. %As before, $f_k(F)\ge \theta_k(d-1+p,d-1)$ by \cref{thm:at-most-2d}.

\textbf{Suppose that $F$ has at least $d+2$ facets.} Then $f_k(F)\ge \zeta_k(d-1+p,d-1)$ by \cref{thm:at-most-2d-refined-short} (see also \eqref{eq:atmost2d-refined}), and  \eqref{eq:3minus1-less-2dplus1-1}  becomes
\begin{equation}\label{eq:3minus1-less-2dplus1-3}
      \begin{aligned}
        f_k(P) & \ge \theta_k(d-1+p,d-1)+\binom{d-2}{k}-\binom{d-p}{k}+\theta_{k-1}(d+1,d-1) +\sum_{i=1}^{d+\ell-p}\binom{d-i}{k}\\  
        & =\binom{d+1}{k+1}+\binom{d}{k+1}-\binom{d-p}{k+1}-\binom{d-p}{k}+\sum_{i=1}^{d+\ell-p}\binom{d-i}{k}\\  
        & =\eta_k(2d+\ell,d)-\binom{d}{k+1}+\brac*{\binom{d-\ell}{k+1}-\binom{d-p+1}{k+1}}+\sum_{i=1}^{d+\ell-p}\binom{d-i}{k}\\ 
        & = \eta_k(2d+\ell,d)-\sum_{i=0}^{d-1}\binom{i}{k}+\sum_{i=1}^{p-\ell-1}\binom{d-\ell-i}{k}+\sum_{i=p-\ell}^{d-1}\binom{i}{k} \\
        & = \eta_k(2d+\ell,d)-\brac*{\sum_{i=1}^{p-\ell-1}\binom{i}{k}+\sum_{i=p-\ell}^{d-1}\binom{i}{k}}+\sum_{i=1}^{p-\ell-1}\binom{d-\ell-i}{k}+\sum_{i=p-\ell}^{d-1}\binom{i}{k} \\
        & = \eta_k(2d+\ell,d)+\sum_{i=1}^{p-\ell-1}\underbrace{\brac*{\binom{d-\ell-i}{k}-\binom{p-\ell-i}{k}}}_{\ge 0\; \text{as $d>p$}} \ge \eta_k(2d+\ell,d).
      \end{aligned}
  \end{equation}
Hence we have equality only if $\ell=p-1$ or $k\ge d-\ell$.    
%In the last inequality, there are $d-1$ negative terms, and $d-1$ positive terms, where the ground set of $\sum_{i=d-p+1}^{d-\ell-1}\binom{i}{k}$ starts from $d-p+1 \ge 2$.

\textbf{Suppose that $F$ has exactly $d$ facets.} Then $F$ is a $(d-1)$-simplex. Since $\card(F_1\cup F)\le 2d-1$,  the number of vertices in $P$ would be at most $2d$, contrary to  our assumption.

\textbf{Suppose that $F$ has exactly $d+1$ facets and is a nonsimple polytope.} Then, $F$  is a $(d-1-p)$-fold pyramid over $T(1)\times T(p-1)$ for some $2\le p\le d-1$ (\cref{lem:dplus2facets}), so $f_k(F)=\theta_k(d-1+p,d-1)$ (\cref{thm:at-most-2d}). The apex of $F$ has degree $d-2+p$ in $F$, so $\deg_P(v_1)\ge d-1+p$. The corresponding bound is:
\begin{equation}\label{eq:3minus1-less-2dplus1-4}
    \begin{aligned}
    f_k(P) & \ge \theta_k(d-1+p,d-1)+\theta_{k-1}(d-1+p,d-1) +\sum_{i=1}^{d+\ell-p}\binom{d-i}{k}\\  
    & = \binom{d}{k+1}+\binom{d-1}{k+1}-\binom{d-p}{k+1}+\binom{d}{k}+\binom{d-1}{k}-\binom{d-p}{k} +\sum_{i=1}^{d+\ell-p}\binom{d-i}{k}\\  
    & = \binom{d+1}{k+1}+\binom{d}{k+1}-\binom{d-p}{k+1}-\binom{d-p}{k}+\sum_{i=1}^{d+\ell-p}\binom{d-i}{k} \ge \eta_k(2d+\ell,d), 
   \end{aligned}
\end{equation}
with the final inequality obtained exactly as in  \eqref{eq:3minus1-less-2dplus1-3}. Consequently,  equality holds precisely when $\ell=p-1$ or $k\ge d-\ell$.

\textbf{Suppose that $F$ has exactly $d+1$ facets and is a simple polytope.} Then $F$ is a simplicial $(d-1)$-prism and $f_k(F)=\theta_k(2(d-1),d-1)$. By  the previous analysis, we may assume that $F_1$ is also a simplicial $(d-1)$-prism. Since $F\cap F_1$ must be a $(d-2)$-simplex $R$ (else $f_0(P)\le 2d+1$), we have  $f_0(P)=3d-2$. Label the vertices in $R$ as $w_1,\ldots,w_{d-1}$.%, and label the corresponding vertices in $F$, $F_1$ as $u_i$, $u_i'$, respectively, $i=1\ldots d-1$. 
%It follows that $w_i$ is adjacent to $v_1$.

Assume a vertex in $R$, say $w_1$, is  nonsimple in $P$. Then $w_1$ is  adjacent to $v_1$.  Let $F_{w_1}$ be a facet  that intersects $F$ at the simplicial $(d-2)$-prism that does not contain  $w_1$. Suppose that $F_{w_1}$ is a pyramid. 
%If $v_1\notin F_{w_1}$, then $F_{w_1}$ would be a nonsimple polytope with exactly $d+1$ facets, reducing the analysis to one of the previous cases of \cref{cl:3minus1-less-2dplus1}, with $F_{w_1}$ in place of $F$. 
Since no vertex in $F_1\setminus F$ can be the apex of $F_{w_1}$,  $v_1$ must be the apex, in which case $\deg_{F_{w_1}}(v_1)\ge 2d-4$. Moreover, $v_1$ is adjacent to $w_1$ and to a vertex in $F_1\setminus F$ because removing the facet $F$ cannot disconnect the graph of $P$ \cite[Thm.~4.16]{Pin24}. It follows that $\deg_{P}(v_1)\ge 2d-2$. Since  $f_k(F)=\theta_k(2(d-1),d-1)$, the bound in \eqref{eq:3minus1-less-2dplus1-1} gives  $f_k(P)\ge \eta_k(3d-2,d)$:
\begin{equation}
\label{eq:cl:3minus1-less-2dplus1-pyramid}
    \begin{aligned}
    f_k(P) &\ge  f_k(F)+f_{k-1} (P/v_1)+\sum_{i=1}^{d+(d-2)-(d-1)}\binom{d-i}{k}\\
    &\ge \theta_k(2(d-1),d-1)+\theta_{k-1}(2(d-1),d-1) +\sum_{i=1}^{d-1}\binom{d-i}{k}= \eta_k(3d-2,d).
   \end{aligned}
\end{equation}
Suppose that $F_{w_1}$ is not a pyramid. Then $F_{w_1}\cap F_1$ (and $F_{w_1}\cap F$) contains vertices from the two disjoint $(d-2)$-simplices in $F_1$ (and in $F$, respectively), implying that $F_{w_1}\cap F_1$ is a simplicial $(d-2)$-prism. It follows that $f_0(F_{w_1})\ge 3d-6\ge 2d-1$ for $d\ge 5$.  By \cref{cl:3minus1-pyramid,cl:3minus1-minus-2,cl:3minus1-more-2dplus1}, we may assume $v_1\in F_{w_1}$, in which case $f_0(F_{w_1})= 3d-5$ because every facet in $P$ has at most $f_0(P)-3$ vertices (\cref{cl:3minus1-pyramid,cl:3minus1-minus-2}). The only simple $(d-1)$-polytope with $3(d-1)-2$ vertices is $T(3)\times T(3)$   \cite[Lemma~2.19]{PinUgoYos16a},  but it does not have a simplicial prism facet. Hence, $F_{w_1}$ is not a simple polytope (nor a pyramid), so it must have  at least $d+2$ facets (\cref{lem:dplus2facets}). Moreover,  $w_1\notin F_{w_1}$ has degree at least $d+1$ in $P$, and the computation in Case 1 of \cref{cl:3minus1-more-2dplus1} applies (see \eqref{eq:3minus1-minus-r-dplus2-2}), with $F_{w_1}$ playing the role of $F$ and $w_1$ the role of $v_1$:
\begin{equation}
\label{eq:cl:3minus1-less-2dplus1-nonpyramid}
\begin{aligned}  
f_k(P)&\ge f_k(F_{w_1})+f_{k-1} (P/w_1)+\sum_{i=1}^{2}\binom{d-i}{k}\\
&\ge \eta_k(2(d-1)+d-3,d-1)+\theta_{k-1}(d+1,d-1)+\sum_{i=1}^{2}\binom{d-i}{k}\\  
&= \brac*{\binom{d}{k+1}+2\binom{d-1}{k+1}-\binom{2}{k+1}}+\brac*{\binom{d}{k}+\binom{d-1}{k}-\binom{d-2}{k}}+\sum_{i=1}^{2}\binom{d-i}{k}\\
&=\eta_k(3d-2,d).
\end{aligned}
\end{equation}
Consequently, we may assume that every vertex in $R$ is simple in $P$. We now show  all the vertices of $P$ except $v_1$ are simple in $P$. Suppose otherwise, and without loss of generality, let  $u_1\in (F\setminus R)\cup (F_1\setminus R)$ be a nonsimple vertex in $P$. We may further assume that $u_1$ is adjacent to $w_1$. Since $w_2$ is simple in $P$, $w_2$ and its $d-1$ neighbours in $F\cup F_1$ other than $w_1$ define a facet $F_{u_1}$ of $P$ that does not contain $u_1$ or $w_1$ \cite[Thm.~2.8.6]{Pin24}. 

Because $F_{u_1}\cap F$ meets both disjoint $(d-2)$-simplices in $F$,  it follows that $F_{u_1}\cap F$ is a simplicial $(d-2)$-prism; likewise, $F_{u_1}\cap F_1$ is a simplicial $(d-2)$-prism. Hence $\card(F_{u_1}\cap (F\cup F_1))=3d-6$. If $F_{u_1}$ does not contain $v_1$, then $f_0(F_{u_1})=3d-6$. Since $3d-6\ge  2d-1$ for $d\ge 5$, this scenario was settled in \cref{cl:3minus1-more-2dplus1}. If $F_{u_1}$ contains $v_1$, then $f_0(F_{u_1})= 3d-5$. As before, the only simple $(d-1)$-polytope with $3(d-1)-2$ vertices is $T(3)\times T(3)$   \cite[Lemma~2.19]{PinUgoYos16a}, which has no simplicial prism facet. Thus $F_{u_1}$ is neither a simple polytope  nor a pyramid, so it must have  at least $d+2$ facets (\cref{lem:dplus2facets}). Accordingly, the computation in \eqref{eq:cl:3minus1-less-2dplus1-nonpyramid} applies, with $F_{u_1}$ replacing $F_{w_1}$ and $u_1$ replacing $w_1$, yielding $f_k(P)\ge \eta_k(3d-2,d)$.

Hence every vertex in $P$ except $v_1$ must be simple in $P$, and the graph of $P$ is therefore determined. By the Excess theorem (\cref{thm:excess-degree}),  $\deg_P(v_1)=2d-2$; that is, $v_1$ is adjacent to every vertex in $(F\setminus F_1)\cup (F_1\setminus F)$. Since a polytope with a single nonsimple vertex is determined by its graph \cite{DPUY2019}, $P$ must be $\Sigma(d)$, the convex hull of \[\set*{0,e_1,e_1+e_k,e_2,e_2+e_k,e_1+e_2,e_1+e_2+2e_k: 3\le k\le d}.\]  See also \eqref{eq:Sigma_d}. Therefore $f_k(P)= \eta_k(3d-2,d)$ for all $k\in [1\ldots d-2]$, and the claim follows.
\end{claimproof}

%In the second scenario, $F_{w_i}$ is a facet not containing the vertex with maximum degree and with at least $2(d-1)$ vertices when $d\ge 4$, so this scenario has been covered by the previous claims.
%\textcolor{red}{If instead $F_{w_i}$  is not a pyramid, it must have exactly $3d-5=3(d-1)-2$ vertices implying that $F_{w_i}$ is not simple \cite{PinUgoYos16a}}. 
%Therefore, we may assume that $F_{w_i}$ has at least $d+2$ facets. So we can apply the same computation as of \cref{eq:claim4_2}.

%\blue{It follows that, there is a facet containing vertices in $R\setminus \{w_1\}$ but not $w_i$ with at least $3d-6$ vertices ($> 2d-2$ when $d\ge 5$), this scenario has been settled by the previous claims.}

%\red{The same argument shows that all vertices in $F\setminus F_1$ and $F_1\setminus F$ are simple in $P$ (this is incorrect).}
%By \cref{cor:number-faces-outside-facet-practical}, there is a sequence $F_1,\ldots,F_r$ of faces within $P$ such that each $F_i$ has dimension $d-i+1$ and contains $v_i$ but does not contain $v_j$ with $j<i$. The number of $k$-faces of $P$ that contains at least one of the vertices in $X$ is bounded from below by
%\begin{align}  
%  \label{eq:3minus1-minus-r-dplus2-1} f_k(P)&\ge f_k(F)+\sum_{i=1}^{r}f_{k-1}(F_i/v_i).
%\end{align}

\subsection{Final part for the inequality statement of the theorem.} By \cref{cl:3minus1-pyramid}, we may assume that $P$ is not a pyramid, while by \cref{cl:3minus1-minus-2}, we may further restrict to the case where $P$ does not have a facet with  $f_0(P)-2$ vertices. Also, by \cref{lem:2d+s-dplus3-simple-pyramid}, we may assume that $P$ contains a vertex that is both nonpyramidal and nonsimple; let  $v_1$ be one such vertex of maximum degree. Let $F$ be a facet of $P$ that does not contain $v_1$. If $2d-1=f_0(P)-(\ell+1)\le f_0(F)\le f_0(P)-3$, then \cref{cl:3minus1-more-2dplus1} establishes that $f_k(P)\ge \eta_k(2d+\ell,d)$ for $k\in [1\ldots d-2]$ and $\ell\ge 1$. Otherwise  \cref{cl:3minus1-less-2dplus1} yields
$f_k(P)\ge \eta_k(2d+\ell,d)$ for $k\in [1\ldots d-2]$ and $\ell\ge 1$. 

\subsection{The equality statement of the theorem}

The equality part is settled in \cref{claim5:equality_analysis}. In this claim, we examine all the inequalities in \cref{cl:3minus1-pyramid,cl:3minus1-minus-2,cl:3minus1-more-2dplus1,cl:3minus1-less-2dplus1} that can hold with equality, relying on the analysis carried out in the inequality part.

\begin{claim}\label{claim5:equality_analysis}
   % Let $P$ be a $d$-polytope with at least $d+3$ facets, 
   Let $d\ge 3$ and $\ell \ge 3$. If $P$ has $2d+\ell$ vertices and at least $d+3$ facets, and if  $f_k(P)=\eta_k(2d+\ell, d)$ for  some $k\in [1\ldots d-2]$, then $P$ has exactly $d+3$ facets.
\begin{proof}
    We proceed by induction on $d\ge 3$ for all $\ell\ge 1$. As noted earlier, the case $\ell = 1$ was independently settled in \cite{PinYos22} and in \cite{Xue24}, and the case $\ell = 2$ was verified  in \cite{PinTriYos24}.  Assume $\ell\ge 3$.
    
 For $d=3$, if $\ell \ge 3$ and $f_2\ge 6$, then  by Euler's polyhedral formula,  $f_1=f_0+f_2-2\ge 13>12=\eta_1(6+\ell,3)$. Hence   no 3-polytope with $2d+\ell$ vertices and at least $d+3$ facets satisfies  equality.

    For $d=4$ and $\ell \ge 3$, the proof of the inequality part already established the equality part: $f_1=\eta_1(8+\ell, 4)$ or $f_2=\eta_2(8+\ell, 4)$ holds if and only if $P$ is a simple 4-polytope with 11 vertices---specifically, if  $P=J(4,4)$, which  has 7 facets. Recall that $J(4,4)$ is obtained by truncating a simple vertex from a simplicial $4$-prism. 

    Now let $P$ be a $d$-polytope with $2d+\ell$ vertices, where $d\ge 5$ and $\ell \ge 3$, and at least $d+3$ facets such that  $f_k(P)=\eta_k(2d+\ell, d)$ for some $k\in [1\ldots d-2]$. We examine the possible equality cases arising from the proofs of  \cref{cl:3minus1-pyramid,cl:3minus1-minus-2,cl:3minus1-more-2dplus1,cl:3minus1-less-2dplus1}.
    
  \subsection*{(i) \cref{cl:3minus1-pyramid}:  $P$ is a pyramid over a facet $F$} 
  
  For some $k\in [2\ldots d-2]$, if $1\le \ell\le d-2$, then $f_k(P)=\eta_k(2d+\ell, d)$ in \eqref{eq:claim1_smallell_fk} implies  $f_k(F)=\eta_{k}(2(d-1)+\ell+1,d-1)$; moreover, if  $\ell\ge d-1$, then $f_k(P)=\eta_k(3d-2, d)=\eta_k(3d-1, d)$ in  \eqref{eq:claim1_largeell_fk} implies $f_k(F)=\eta_{k}(3d-4,d-1)$ for some $k\in [2\ldots d-2]$. By induction hypothesis, equality for some $k\in [2\ldots d-2]$ implies that $F$ has exactly $d+2$ $(d-2)$-faces, and so  $P$ has exactly $d+3$ facets. For $k=1$, equality never occurs (see \eqref{eq:claim1_smallell_f1} and \eqref{eq:claim1_largeell_f1}).
  %Moreover, if $f_0(P)=3d-2$, then by \eqref{eq:claim1_largeell_f1}, we have $f_1(P)>\eta_1(3d-2, d)$. If $\ell\ge d-1$, then $f_1(P)=\eta_1(3d-1, d)$ in \eqref{eq:claim1_largeell_f1} gives  $f_0(F)=3d-3$, a contradiction. 

%If $f_0(P)\ge 3d-1$, then $f_1(P)=\eta_1(3d-1, d)$ in \eqref{eq:claim1_largeell_f1} gives $f_1(F)=\eta_1(3d-4, d-1)$ and $f_0(F)=3d-2$, a contradiction.
    
    %\blue{(Side remark of some possible minimizers: for $P$ with $2d+2$ vertices, $P$ can be a pyramid over $J(4, d-1)$ which minimizes $k=d-2$.... it seems that for $P$ with $2d+\ell$ vertices, $P$ can be a pyramid over $J(\ell+2, d-1)$, which minimizes $f_k$ where $k\ge d-\ell$.)}

  \subsection*{(ii) \cref{cl:3minus1-minus-2}:  $P$  has a facet $F$ with $f_{0}(P)-2$ vertices and is not a pyramid} 
  
  Let $v_{1}$ and $v_{2}$ be the vertices outside $F$.  Two cases arise depending on the number of $(d-2)$-faces in $F$.

\case $F$ has at least $d+2$ facets.

%all the $k$-faces of $F$ and the $k$-faces of $P$ that intersect with $F$ at $(k-1)$-faces are all the $k$-faces of the polytope, where $k\ge 2$; for $k=1$, all the edges of the polytope are edges in $F$, the unique edge incident to the vertex in $F$ for each vertex, and the edges between the two vertices outside $F$. 

If there exists a  facet $F'$ of $P$ not intersecting $F$ in a ridge of $P$, then $F'$ must be a two-fold pyramid over a $(d-3)$-face of $F$.

For $k\in [2\ldots d-2]$, $f_k(P)=\eta_k(2d+\ell, d)$ in \eqref{eq:claim2_case1_fk} forces $f_k(F)=\eta_k(2d+\ell-2,d-1)$ and $ f_{k-1}(F)=\eta_{k-1}(2d+\ell-2,d-1)$. By induction, $F$ has exactly $d+2$ $(d-2)$-faces.  Any facet $F'$ as above would contribute $k$-faces that intersect $F$ at $(k-2)$-faces, contradicting equality. Hence $P$ has exactly $d+3$ facets. 

Similarly, for $k=1$, equality $f_1(P)=\eta_1(2d+\ell, d)$ in \eqref{eq:claim2_case1_f1} implies that $f_1(F)=\eta_1(2d+\ell-2,d-1)$ and that  $P$ has a unique edge not incident with a vertex of $F$. By induction, $F$ has exactly $d+2$ $(d-2)$-faces. Additionally,  each vertex of $F$ has a unique neighbour outside $F$, which rules out the existence of a facet $F'$ (as it would be a two-fold pyramid). Again, $P$ has exactly $d+3$ facets.

\case  The facet $F$ has exactly $d+1$ facets.
    \begin{subcases}
    
    Two subcases arise depending on whether $F$ is nonsimple or simple  (see \cref{lem:dplus2facets}).
        \subcase  Strict inequalities hold in all scenarios; see \eqref{eq:claim2_subcase2.1_f1}, \eqref{eq:claim2_subcase2.1-dplus2}, and \eqref{eq:claim2_subcase2.1-dplus1}.
        
        % $F$ is a $(d-a-1)$-fold pyramid over $T(m)\times T(a-m)$  for some \green{$2\le a\le d-2$} and $2\le m\le \floor{a/2}$. Then $f_0(F)=d+m(a-m)$. Additionally, there exists a facet $F_2$ of $P$, which is a pyramid with apex $v_2$ and base a $(d-2)$-face of $F$, such that  that $v_1\notin F_2$ and  $\deg_{F_2}(v_2)=d+1$ (see \eqref{eq:claim2_subcase2.1-F2}). Refer to Equations \eqref{eq:claim2_subcase2.1-dplus2}
        % and \eqref{eq:claim2_subcase2.1-dplus1}.
        
        % The $(d-2)$-faces in $F$ are either a $(d-a-2)$-fold pyramid over $T(m)\times T(a-m)$, a $(d-a-1)$-fold pyramid over $T(m-1)\times T(a-m)$, or a $(d-a-1)$-fold pyramid over $T(m)\times T(a-m-1)$.  These have $d-1+m(a-m)\ge d+1$, $d-1+(a-m)(m-1)\ge d+1$, and $d-1+m(a-m-1)\ge d+1$ vertices, respectively. See \cref{rmk:dplus2facets-facets}. Since $\deg_{F_2} (v_2)=d+1$, this can only occur when $m=2$ and $a=4$, so $F$ would be a $(d-5)$-fold pyramid over $T(2)\times T(2)$,  where $d\ge 6$. It follows that $f_0(P)=d+6\le 2d$ for $d\ge 6$,  a contradiction.
        \subcase If $f_k(P)=\eta_k(3d-4,d)$ in \eqref{eq:claim2_subcase2.2-m2}, then  $f_k(F)=\rho_{k}(d-1,2,d-1)$, the number of $k$-faces containing $v_1$ is exactly $\theta_{k-1}(2(d-1)-1,d-1)$, and the number of $k$-faces containing $v_2$ but not $v_1$ is exactly $\theta_{k-1}(2(d-2),d-2)$. From \cref{thm:at-most-2d}, it follows that $v_1$ lies in exactly $d+1$ facets of $P$. Moreover, there is a unique facet containing $v_2$ but not $v_1$, namely $F_2$. Hence $P$ has exactly $d+3$ facets.
      
  %      \red{Strict inequalities hold in all scenarios; see \eqref{eq:claim2_subcase2.1_f1}, \eqref{eq:claim2_subcase2.2-m2} and \eqref{eq:claim2_subcase2.2_mge3}}. 
    \end{subcases}

  \subsection*{(iii) \cref{cl:3minus1-more-2dplus1}: $P$ is not a pyramid, it has no facet omitting precisely two vertices, and it contains a vertex $v_1$ that is neither pyramidal nor simple and has the maximum degree among such vertices. Moreover,  there exists a facet $F$ not containing $v_1$ with $2d+\ell-r$ vertices, where $3\le r\le \ell+1$}

Let $X:=\set{v_1,\ldots,v_{r}}$ be the set of vertices outside $F$.  By \cref{cor:number-faces-outside-facet-practical}, there exist faces $F_1,\ldots,F_r$ in $P$ such that each $F_i$ has dimension $d-i+1$, contains $v_i$, and excludes any $v_j$ with $j<i$.  Then
\begin{equation*}
f_k(P)\ge f_k(F)+\sum_{i=1}^{r}f_{k-1}(F_i/v_i).
\end{equation*}  
Recall we have the freedom to choose the order of the vertices in the sequence $(v_1,\ldots,v_r)$.
\setcounter{case}{0}

We rely on variations on the following facts, which are inspired by \cite[Sec.~4]{Xue21} and \cite[Sec.~3.1]{PinWanYos24a}.

\begin{fact}
\label{fact:two-faces}
Let $F'$ and $F''$ be two distinct facets of a $d$-polytope that both contain a vertex $v$, and let $1\le k\le d-1$. Then there exists a $k$-face of $F'$ containing $v$ that is not a $k$-face of $F''$.
\end{fact}
% \begin{factproof}
% Let $m$ be the dimension of $F'\cap F''$. Then $v$ has at least $d-1-m$ neighbours in $F'\setminus F''$. Choose one such neighbour $w$.  Now consider any $k$-face of $F'$ containing the edge $vw$.
% \end{factproof}
\begin{fact}
\label{fact:minimum}
   If the $r$ vertices from $X$ are contained in precisely $\sum_{i=1}^{r}\binom{d+1-i}{k}$ $k$-faces of $P$, then the number of $k$-faces in $P$ that contain the vertex $v_i$ and no vertex $v_j$ with $j<i$ is precisely $\binom{d+1-i}{k}$.  
\end{fact}
\begin{factproof} While the proof is given in \cite[Lem.~20]{PinWanYos24a}, we will reproduce it here for the sake of completeness. 
The minimum number of $k$-faces within a $(d+1-i)$-face $F_i$ that contains the vertex $v_{i}$ is attained when the vertex figure $F_i/v_i$ is  a $(d-i)$-simplex.  Since $F_i$ excludes each $v_j$ with $j<i$ and the number of $k$-faces of  $P$ containing some vertex in $X$ is precisely $\sum_{i=1}^{r}\binom{d+1-i}{k}$, each vertex figure must be a minimiser. 
\end{factproof}

A possible variation of \cref{fact:minimum} is when $r-t$ vertices from $X$ are contained in precisely $\sum_{i=1}^{r-t}\binom{d+1-i}{k}$ $k$-faces of $P$.

    \case The facet $F$ has at least $d+2$ facets.
    
According to \eqref{eq:3minus1-minus-r-dplus2-2}, if $f_{k}(P)=\eta_k(2d+\ell,d)$,  then $\deg_P(v_1)=d+1$, the number of $k$-faces containing $v_1$ is precisely $\theta_{k-1}(d+1,d-1)$, and the number of $k$-faces containing some vertices in $X$ but not $v_1$ is exactly $\sum_{i=1}^{r-1}\binom{d-i}{k}$. 

 The facets of $P$ can be partitioned into the facet $F$ and the sets $X_i$ of facets containing the vertex $v_i$ but no vertex $v_j$ with $j<i$.

Since $f_{k-1}(P/v_1)=\theta_{k-1}(d+1,d-1)$, $P/v_1$ must have $d+1$ facets (\cref{cor:more-dpluss}), so $\card X_1=d+1$. %Recall that the $(2,d-3)$-triplex is the $(d-3)$-fold pyramid over $T(1)\times T(1)$.
Suppose there are two facets $F_2$ and $F_2'$ in $X_{2}$.  Then there exists a $k$-face of $F_{2}'$ that contains $v_2$ and is not a face of $F_{2}$ (\cref{fact:two-faces}). Since $F_{2}$ already contributes $\binom{d-1}{k}$ $k$-faces containing $v_{2}$ (but not $v_{1}$), it follows that the number of $k$-faces in $P$ containing $v_{2}$ but not $v_{1}$ is greater  than $\binom{d-1}{k}$, contradicting \cref{fact:minimum}. Thus $\card X_{2}=1$.  The same idea proves that  $\card X_{\ell}=0$ for $\ell\in [3, r]$, yielding that $P$ has $d+3$ facets.

%By the induction hypothesis, the facet $F$ has exactly $d+2$ $(d-2)$-faces. The vertex $v_1$ has degree $d+1$ and is contained in exactly $d+1$ facets, and there is a unique facet containing some vertices in $X\setminus \{v_1\}$. So the polytope has exactly $d+3$ facets.
    \case The facet $F$ has precisely  $d+1=(d-1)+2$ facets.
    \begin{subcases} 
        \subcase $F$ is a nonsimple polytope. Here we have strict inequalities for all $k\in [1 \ldots d-2]$; see \eqref{eq:3minus1-minus-r-dplus1-subcase2.1-k1}, \eqref{eq:3minus1-minus-r-dplus1-subcase2.1-dplus2}, and \eqref{eq:3minus1-minus-r-dplus1-subcase2.1-dplus1}.

        \subcase  $F$ is a simple polytope.  
        The inequality proof considers whether there exists a facet $J_1$ with at least two vertices outside $F \cup J_1$. 
        
        \textbf{Suppose such a facet $J_1$ exists.} Let $J_{1}\cap X=\set{v'_{1},\ldots,v'_{t}}$ and $X\setminus J_1=\set{v_1'',\dots, v_{r-t}''}$. Recall we have the freedom to choose the order of the vertices in $J_{1}\cap X$ and in $X\setminus J_1$. 
        
\subsubsection*{Consider the case $2\le t\le r-2$} 
If $f_k(P)=\eta_k(2d+\ell,d)$ in \eqref{eq:claim3_subcase2.2-large-t-m3} or \eqref{eq:claim3_subcase2.2-large-t-m2}, then  (a) every vertex in $J_1\cap X$ is simple in  $J_1$;  (b) the number of $k$-faces of $J_1$ containing a vertex in $J_1\cap X$ is exactly $S_1: =\sum_{i=1}^{2}\binom{d-i}{k}$; (c) every vertex in $X\setminus J_1$ is simple in $P$; and (d) the number of $k$-faces of $P$ containing a vertex in $X\setminus J_1$  is exactly $S_2: =\sum_{j=1}^{2}\binom{d+1-j}{k}$. 
        
The vertex $v_1''$ is contained in precisely $d$ facets. Since $S_2=\sum_{j=1}^{2}\binom{d+1-j}{k}$, there is a unique facet  $I_2$ of $P$ containing  $v_2''$ but not $v_1''$. Indeed, suppose that there are two facets $I_2$ and $I_2'$ of $P$ containing  $v_2''$ but not $v_1''$, then  $v_2''$  would contribute more than $\binom{d-1}{k}$ to $S_2$, a contradiction; see \cref{fact:two-faces,fact:minimum}. 
%If $I_2$ and $I_2'$ don't share a vertex---say $v_2''\in I_2$,  $v_3''\in I_2'$, and $v_1'',v_2''\notin I_2'$---then $v_3''$ would contribute more than $\binom{d-2}{k}$ to $S_2$, a contradiction to \cref{fact:minimum}.  } 

Additionally, since $S_1=\sum_{i=1}^{2}\binom{d-i}{k}$, if there were a facet  different from $ J_1$ that contains a vertex in $J_1\cap X$ but no vertex from $X\setminus J_1$, say $J_1'$, then we could select such a vertex as $v_1'$. Then the contribution of the $k$-faces containing $v_1'$ in $J_1\cup J_1'$  would exceed $\binom{d-1}{k}$ (see \cref{fact:two-faces}), yielding $S_1>\sum_{i=1}^{2}\binom{d-i}{k}$. Thus $P$ has exactly $d+3$ facets: the $d$ facets containing $v_1''$, together with $I_2$, $J_1$, and $F$.

 \subsubsection*{Consider the case $t=1$} If $f_k(P)=\eta_k(2d+\ell,d)$  in  \eqref{eq:claim3_subcase2.2_2}, then (a) $r=3$, (b) all the vertices in $X\setminus J_1$ are simple, (c) the number of $k$-faces containing a vertex in $X\setminus J_1$ is exactly $\sum_{j=1}^{r-1}\binom{d+1-j}{k}$, and (d) the number of $k$-faces containing the unique vertex $v_1'$ in $J_1$ is exactly $\theta_{k-1}(2(d-2),d-2)$. It is routine to verify that there are exactly $d$ facets containing the vertex $v_1''$, a unique facet containing some of the vertices in $X\setminus J_1$ but not $v_1''$, and  $J_1$ is the  unique facet of $P$ that contains the vertex $v_1'$  but no vertex from $X\setminus J_1$. Hence  $P$ has exactly $d+3$ facets.

    \textbf{Suppose no such facet $J_1$ exists.} %There exists a facet $U_1$ that is a two-fold pyramid over a face of $F$ with dimension at most $d-3$. 
        % According to \eqref{eq:claim3_subcase2.nofacetJ1-r4}, when $r\ge 4$ and $m\ge 3$, we have strict inequality unless $k=d-2$ \blue{or $k=d-3$}. When $r\ge 4$ and $m=2$, equality also occurs in  \eqref{eq:claim3_subcase2.2_3}. 
        If $f_k(P)=\eta_{k}(2d+\ell,d)$ in \eqref{eq:claim3_subcase2.nofacetJ1-r4} or $f_k(P)=\eta_{k}(3d-6+r,d)$  in \eqref{eq:claim3_subcase2.2_3}, then the number of $k$-faces containing $v_1$ is exactly $\theta_{k-1}(d+2,d-1)$, while the number of $k$-faces containing some vertex in $X\setminus v_1$ but not $v_1$ is exactly $\sum_{j=1}^{r-1}\binom{d-j}{k}$. By \cref{cor:more-dpluss}, the vertex $v_1$ lies in exactly $d+1$ facets of $P$.  Moreover, because the contribution $\sum_{j=1}^{r-1}\binom{d-j}{k}$ is minimal, there must exist a unique facet containing a vertex in $X\setminus \set{v_1}$ but not $v_1$ (\cref{fact:minimum,fact:two-faces}). Hence $P$ has exactly $d+3$ facets.

        Suppose $r=3$. The the relevant equations  are \eqref{eq:claim3_subcase2.2_4} and \eqref{eq:claim3_r3_m2}, which are both  strict.
        % Then $X=\set{v_1,v_2,v_3}$, and there is a facet $I_1$ that contains both $v_2$ and $v_3$ but not $v_1$, in which $v_2$ or $v_3$---say $v_2$---has degree at least $2d-5$.  If $f_k(P)=\eta_{k}(2d+\ell,d)$ for some $k$ in \eqref{eq:claim3_subcase2.2_4}, then the number of $k$-faces containing $v_1$ is exactly $\theta_{k-1}(2d-4,d-1)$, and the number of $k$-faces containing $v_2$ or $v_3$ but not $v_1$ is exactly $\theta_{k-1}(2d-5,d-2)+\binom{d-2}{k}=f_{k-1}(I_1/v_2)+f_{k-1}(R/v_3)$, where $R$ is a $(d-2)$-face of $P$ containing $v_3$ but not $v_1,v_2$. It follows that $v_1$ is contained in exactly $d+1$ facets of $P$ (\cref{thm:at-most-2d}). If there existed a facet $I_1'$, different from $I_1$, that contains $v_2$ or $v_3$ but not $v_1$, then $I_1'$ must contain both $v_2$ and $v_3$, which would give that the number of $k$-faces containing $v_2$  is greater than $\theta_{k-1}(2d-5,d-2)$ (\cref{fact:two-faces}), a contradiction. Hence $P$ has exactly $d+3$ facets. When $r=3$ and $m=2$, the relevant equation is \eqref{eq:claim3_r3_m2}, which is strict.
    \end{subcases}

    % $m\ge 3$, $d\ge 7$, then $f_k(F)=\rho_{k}(d-1,3,d-1)$ and $\binom{d-\ell}{k+1}+\binom{d-3}{k+1}+\sum_{i=1}^{t-2}\binom{d-2-i}{k}+\sum_{j=1}^{r-t-2}\binom{d-1-j}{k}-\binom{4}{k+2}=0$, it follows that $k=d-3$ where $t=2$ and $r-t=2$, or $k=d-2$ where $r-t=2$. In both of the cases, we have $r-t=2$, i.e. the facet $J_1$ leaves out exactly vertices in $X$. By \cref{eq:claim3_subcase2.2_1}, there are exactly $d$ facets containing $v_1'$ and there is a unique facet (say $J_2$) containing $v_2'$ but not $v_1'$. Suppose there is another facet, say $J_3$, contains some of the vertices in $J_1$ but not $v_1'$ and $v_2'$, then we have strict inequality (as the number of $k$-faces containing some of the vertices in $J_1$ would be strictly larger than $\sum_{j=1}^{r-t}\binom{d+1-j}{k}$ in \cref{eq:claim3_subcase2.2_1}). So the polytope has exactly $d+3$ facets.\\
    
  \subsection*{(iv) \cref{cl:3minus1-less-2dplus1}:  $P$ is not a pyramid,  it contains a vertex $v_1$ that is neither pyramidal nor simple and has the maximum degree among such vertices. Every facet of $P$ not containing $v_1$ has at most $2d-2$ vertices; in particular, let $F$ be one such facet}

In \cref{cl:3minus1-less-2dplus1}, two cases arise. In \cref{case:3minus1-small-facet-two}, all the inequalities---namely, \eqref{eq:3minus1-less-2dplus1-case1-part1}, \eqref{eq:3minus1-less-2dplus1-case1-part2}, and \eqref{eq:3minus1-less-2dplus1-case1-part3}---hold strictly. 

In \cref{case:3minus1-small-facet-one},  the facet $F_1$  in the shelling of the boundary complex of $P$ is such that $F_1 \cup F$ contains all the vertices of $P$ except $v_1$. Four subcases are examined, according to the number of $(d-2)$-faces of $F$: $F$ has at least $d+2$ facets;  $F$ has $d$ facets; $F$ has $d+1$ facets and is nonsimple; and $F$ has $d+1$ facets and is simple. 

Suppose $F$ has at least $d+2$ facets. If $f_k(P)=\eta_k(2d+\ell,d)$ in \eqref{eq:3minus1-less-2dplus1-3}, then $f_k(F)=\zeta_k(d-1+p,d-1)$, the number of $k$-faces containing $v_1$ is exactly $\theta_{k-1}(d+1,d-1)$, and the number of $k$-faces containing vertices outside $F$ but not $v_1$ is exactly $\sum_{i=1}^{d+\ell-p}\binom{d-i}{k}$. From \cref{cor:more-dpluss}, it follows that $v_1$ is contained in exactly $d+1$ facets. Moreover, $F_1$ is the unique facet containing vertices outside $F$ but not $v_1$ (\cref{fact:two-faces,fact:minimum}). Hence $P$ has exactly $d+3$ facets.

The case where $F$ is a simplex is ruled out. 

Now suppose $F$ has exactly $d+1$ facets and is nonsimple.  If $f_k(P)=\eta_k(2d+\ell,d)$ in \eqref{eq:3minus1-less-2dplus1-4}, then the number of $k$-faces containing $v_1$ is exactly $\theta_{k-1}(d-1+p,d-1)$ ($p\le d-1$),  and the number of $k$-faces containing vertices outside $F$ but not $v_1$ is exactly $\sum_{i=1}^{d+\ell-p}\binom{d-i}{k}$. It follows from  \cref{cor:more-dpluss} that $v_1$ is contained in exactly $d+1$ facets of $P$, and from \cref{fact:two-faces,fact:minimum} that $F_1$ is the unique facet containing vertices outside $F$ but not $v_1$. Hence $P$ has exactly $d+3$ facets.

Next, suppose $F$ has exactly $d+1$ facets and is simple. Then $F$ is a simplicial prism, and we may assume that  $F_1$  is also a simplicial prism intersecting $F$ in a $(d-2)$-simplex  $R$. It follows that $f_0(P)=3d-2$.

First, assume that a vertex in $R$, say $w_1$, is not simple in $P$, and let $F_{w_1}$ be the facet  intersecting $F$ in the simplicial $(d-2)$-prism of $F$ that omits  $w_1$.

Suppose that $F_{w_1}$ is a pyramid; it suffices to consider $v_1$ is the apex of $F_{w_1}$. By \eqref{eq:cl:3minus1-less-2dplus1-pyramid}, if $f_k(P)=\eta_k(3d-2,d)$, then  $f_k(P/v_1)=\theta_{k-1}(2(d-1),d-1)$, and the number of $k$-faces containing vertices in $F_1\setminus R$ but not $v_1$ is precisely $\sum_{i=1}^{d-1}\binom{d-i}{k}$. It follows that $v_1$ lies in exactly $d+1$ facets of $P$, and $F_1$ is the unique facet containing vertices in $F_1\setminus R$ but not $v_1$ (\cref{fact:two-faces,fact:minimum}). Hence $P$ has exactly $d+3$ facets.

Now suppose that $F_{w_1}$ is not a pyramid. It suffices to consider that $v_1\in F_{w_1}$. In \eqref{eq:cl:3minus1-less-2dplus1-nonpyramid}, if $f_k(P)=\eta_k(3d-2,d)$,  then $f_k(F_{w_1})=\eta_k(2(d-1)+d-3,d-1)$, $f_{k-1} (P/w_1)=\theta_{k-1}(d+1,d-1)$, and the number of $k$-faces containing vertices outside $F_{w_1}$ but not $w_1$ is exactly $\sum_{i=1}^{2}\binom{d-i}{k}$. It follows from \cref{thm:at-most-2d} that  $w_1$ lies in exactly $d+1$ facets, and from \cref{fact:two-faces,fact:minimum} that there is a unique facet containing vertices outside $F_{w_1}$ but not $w_1$. Hence $P$ has exactly $d+3$ facets.

Consequently, we may assume that every vertex in $R$ is simple in $P$. If  there exists another nonsimple vertex in $P$, say $u_1\ne v_1$, then the vertex $w_2\in R\setminus\set{w_1}$ and its $d-1$ neighbours in $F\cup F_1$ other than $w_1$ define a facet $F_{u_1}$ that does not contain $u_1$ or $w_1$. The same computation as in \eqref{eq:cl:3minus1-less-2dplus1-nonpyramid}, with $F_{u_1}$ replacing $F_{w_1}$ and $u_1$ replacing $w_1$, yields $f_k(P)\ge \eta_k(3d-2,d)$. In this case, the same reasoning shows that if  $f_k(P)= \eta_k(3d-2,d)$, then the analysis focussing on \eqref{eq:cl:3minus1-less-2dplus1-nonpyramid} again implies that $P$ has exactly $d+3$ facets.

Hence every vertex of $P$ except $v_1$ is simple. The inequality proof then shows that $P$ must be $\Sigma(d)$ (see \eqref{eq:Sigma_d}), which has exactly $d+3$ facets. This completes the proof of the equality part.
\end{proof}
\end{claim}
The proof of the theorem is concluded. 

\section{Conclusions}
\label{sec:conclusions}

Having established a lower bound for $d$-polytopes with up to $3d - 1$ vertices, one may ask what comes next. While the bound extends to polytopes with at least $3d - 1$ vertices, it is no longer tight.   The role of $d$-polytopes with $d + 2$ facets persists in lower bound statements beyond this vertex range. As is well known, the problem becomes more interesting when restricted to polytopes with at least $d + 3$ facets.

Despite the length of the present paper, the lower bound itself arises in a relatively straightforward manner—from truncations of triplexes, which are the minimisers among polytopes with at most $2d$ vertices.   In contrast, when extending to polytopes with more vertices, say between $3d$ and $4d-4$ vertices, the situation becomes more intricate. 

Our experience suggests that truncating faces of known minimisers often yields new minimisers. For $d\ge4$ and $0\le k \le d-1$,  we identify  five potential examples.

\begin{enumerate}
    \item A $d$-polytope obtained by truncating a simple edge (an edge whose endpoints are both simple vertices) from an $(s_1,d-s_1)$-triplex  ($2\le s_1\le d$); that is,  a $(d-s_1)$-fold pyramid over $T(1)\times T(s_1-1)$.  The resulting number of $k$-faces is  
    \begin{equation*}
        \label{eq:pot-minimiser-1}
        f_1(d,s_1,k):=\binom{d+1}{k+1}+2\binom{d}{k+1}+\binom{d-1}{k+1}-\binom{1}{k+1}-\binom{d+1-s_1}{k+1}-\binom{2}{k+1}.
    \end{equation*}
    In particular, these polytopes have $d+3$ facets and between $3d-2$ and $4d-4$ vertices.  
    \item A $d$-polytope obtained by truncating a simple vertex from a $(d-s_2)$-fold pyramid over $T(2)\times T(s_2-2)$ ($3\le s_2\le d$). The resulting number of $k$-faces is 
        \begin{equation*}
        \label{eq:pot-minimiser-2}
    f_2(d,s_2,k):=\binom{d+1}{k+1}+2\binom{d}{k+1}+\binom{d-1}{k+1}-\binom{1}{k+1}-\binom{d+1-s_2}{k+1}-\binom{d+2-s_2}{k+1} 
    \end{equation*}
    In particular, these polytopes have $d+3$ facets and  between $2d+4$ and $4d-4$ vertices.  
\item A $d$-polytope obtained by truncating a simple vertex from $J(s_3+1,d)$  ($1\le s_3\le d-3$). 
\begin{equation*}
        \label{eq:pot-minimiser-3}
    f_3(d,s_3,k):=\binom{d+1}{k+1}+3\binom{d}{k+1}-\binom{d-s_3}{k+1}-\binom{2}{k+1}.
\end{equation*}
In particular, these polytopes have have $d+4$ facets and  between $3d$ and $4d-4$ vertices.
\item A $d$-polytope obtained by truncating a nonsimple vertex from an $(s_4,d-s_4)$-triplex ($2\le s_4\le d-2$). 
        \begin{equation*}
        \label{eq:pot-minimiser-4}
    f_4(d,s_4,k):=\binom{d+1}{k+1}+2\binom{d}{k+1}+\binom{d-1}{k+1}-\binom{1}{k+1}-\binom{d+1-s_4}{k+1}-\binom{d-s_4}{k+1} 
    \end{equation*}
In particular, these polytopes have $d+3$ facets and  between $2d+2$ and $4d-4$ vertices.
\item A $d$-polytope obtained by truncating a \emph{simple triangle} (a triangular face whose vertices are simple) from an $(s_5,d-s_5)$-triplex ($3\le s_5\le 5$).   
\begin{equation*}
        \label{eq:pot-minimiser-5}
f_5(d,s_5,k):=\binom{d+1}{k+1}+2\binom{d}{k+1}+\binom{d-1}{k+1}-\binom{4}{k+2}-\binom{d+1-s_5}{k+1}+\binom{d-2}{k+1}.
\end{equation*}
In particular, these polytopes have $d+3$ facets and  between $4d-6$ and $4d-4$ vertices.
\end{enumerate}

The next lemma analyses these five functions.
\begin{lemma} 
\label[lemma]{lem:function-analysis}
Let $d\ge 4$, $2\le s_1\le d$, $3\le s_2\le d$, $1\le s_3\le d-3$, $2\le s_4\le d-2$, $3\le s_5\le 5$, and $0\le k\le d-1$. The following holds.
   \begin{enumerate}[{\rm (i)}]
   \item If $f_1(d,s_1,0)=f_2(d,s_2,0)$, then $f_1(d,s_1,k)\le f_2(d,s_2,k)$  for each $k\in [1\ldots d-1]$. 
       \item If $f_2(d,s_2,0)=f_4(d,s_4,0)$, then $f_2(d,s_2,k)=f_4(d,s_4,k)$ for each $k\in [1\ldots d-1]$. 
       \item    If $f_1(d,s_1,0)=f_3(d,s_3,0)$, then $f_1(d,s_1,k)-f_3(d,s_3,k)$  can be positive, zero, or negative depending on $(d,s_3,k)$.
       \item If $f_1(d,s_1,0)=f_5(d,s_5,0)$, then, for   each $k\in [1\ldots d-2 ]$, we have the following.
       \begin{enumerate}
           \item If $s_5=3$ and $d\ge 4$, then  $f_1(d,s_1,k)-f_5(d,s_5,k)=0$.
           \item If $s_5=4,5$ and $d\ge 5$, then $f_1(d,s_1,k)\le f_5(d,s_5,k)$ .
           \item If $s_5=4,5$ and $d= 4$, then $f_1> f_5$.
       \end{enumerate}
       \item  If $f_3(4,s_3,0)=f_5(4,s_5,0)$ and $s_5=5$, then $f_3(4,s_3,k)>f_5(4,s_5,k)$ for $k=1,2$. Moreover, if $s_5=4$, then $f_3(4,s_3,0)\ne f_5(4,4,0)$.
   \end{enumerate} 
\end{lemma}
\begin{proof}
(i)  Equality at $k=0$ forces $s_1=2s_2-d$. Then \cref{lem:combinatorial-identities}(vi) implies $f_1(d,s_1,k)-f_2(d,s_2,k)\le 0$.

    (ii) Equality at $k=0$ forces $s_2=s_4+1$,  yielding the result. 

    (iii) The condition $f_1(d,s_1,0)=f_3(d,s_3,0)$ gives
$s_1=s_3+3$. For $k=1$, we already see the variation in sign of $f_1-f_3$, since  $f_1-f_3=d-2s_3-3$.

(iv)  Equality at $k=0$ forces $s_1=s_5+d-5$. For each possible value of $s_5$, consider  $ f_1(d,s_1,k)-f_5(d,s_5,k)$:
\begin{equation*}
\binom{d+1-s_5}{k+1}-\binom{d-2}{k+1}
+\binom{4}{k+2}
-\binom{6-s_5}{k+1}-\binom{1}{k+1}-\binom{2}{k+1}. 
\end{equation*}
If $s_5=3$, then  $f_1(d,s_1,k)-f_5(d,s_5,k)=0$ for all $k\in [1\ldots d-2 ]$.
If $s_5=4,5$  and $d\ge 5$, then $f_1\le f_5$ for all $k$, with equality at $k=d-2,d-1$ (and at $d=5,\,k=1,2$) and strict inequality otherwise. If instead $d= 4$ and $s_5=4,5$, then $f_1> f_5$ for $k=1,2$.

(v) The part about $s_5=5$ follows directly by checking the sign of $f_3(4,s_3,k)-f_5(4,s_5,k)$.
\end{proof}

 In light of \cref{lem:function-analysis}, we make  the following conjecture for $d$-polytopes with  at least $d+3$ facets and between $3d$ and $4d-4$ vertices. 
 \begin{itemize}
     \item For $d\ge 5$, the minimum number of faces swaps between $f_1$ and $f_3$.
     \item If $d=4$ and $4d-4=3d$ vertices, then the minimum number of faces is given by $f_5$. 
 \end{itemize}

\section{Combinatorial identities and proofs}
We start with a combinatorial lemma whose first three expressions are from \cite{Xue21}.
\begin{lemma}[Combinatorial equalities and inequalities] For all integers $d\ge 2$, $k\in [1\ldots d-1]$, the following hold
\label[lemma]{lem:combinatorial-identities}
\begin{enumerate}[{\rm (i)}] 
\item If $2\le r\le s\le d$ are integers, then $\theta_{k}(d+s-r,d-1)+\sum_{i=1}^{r}\binom{d+1-i}{k}\ge \theta_{k}(d+s,d)$, with equality only if $r=2$ or $r=s$.
\item If $n\ge c$ are positive integers, then $\binom{n}{c}=\binom{n-1}{c-1}+\binom{n-1}{c}$.
\item If $n\ge c$ and $n\ge a$  are positive integers, then  $\binom{n}{c}-\binom{n-a}{c}=\sum_{i=1}^{a}\binom{n-i}{c-1}$.
\item If $n,c$ are positive integers, then  $\binom{n}{c}=\sum_{i=1}^{n}\binom{n-i}{c-1}$.
\item (Vandermonde's identity) If $n,a,c$ are nonnegative integers, then \[\binom{n+a}{c}=
\sum_{i=0}^{c}\binom{n}{i}\binom{a}{c-i}.\]  
\item If  $a,b,c,S\ge 0$ are integers such that $a+b=S$, then
\[\binom{a}{c}+\binom{b}{c} \ge \binom{\lfloor S/2\rfloor}{c}+\binom{\lceil S/2\rceil}{c}.\]
\item If $d\ge 5$, $1\le \ell\le d-4$, $a=\floor{(d+\ell)/2}+1$, $k\in [1\ldots d-2]$. Then 
\begin{equation*}
    \binom{d-(\ell+1)}{k+1}-\binom{d-a+1}{k+1}-\binom{d-a}{k+1} \ge 0.
\end{equation*} 
Furthermore, if $k< d-\ell-1$, then the inequality is strict.
\end{enumerate}  
\end{lemma}
\begin{proof}

The proof of (i) is embedded in the proof of \cite[Thm.~3.2]{Xue21}. The identity (i) is spelled out in \cite[Claim 1]{Pin24}. Repeated applications of (ii) yield (iii) and (iv); (v) is well known. (vi) Suppose $a\ge b$. The lemma is trivial for $a-b\le 1$ and easy for $a-b\le 2$. 
%\blue{($\binom{a}{c}+\binom{b}{c}\ge \binom{a-1}{c}+\binom{b+1}{c}$ is true for $a\ge b$)} 
Complete the proof by induction on $a-b\ge 2$.

(vii) Let $C:=\binom{d-(\ell+1)}{k+1}-\binom{d-a+1}{k+1}-\binom{d-a}{k+1}$. If $d+\ell$ is even, then $d-(\ell+1)=d-a+1+d-a$, giving that $C\ge 0$, where equality holds if and only if $k\ge d-\ell-1$. If $d+\ell$ is odd, then $d-\ell-1$ is even, say $d-\ell-1=2p\ge4 $, and $d-a=(d-\ell-1)/2$. Thus, $C$ becomes
\begin{align*}
    C=\binom{2p}{k+1}-\binom{p+1}{k+1}-\binom{p}{k+1}.
\end{align*}
%If $k>p$, then $C\ge 0$. Now consider $k\le p$.

If $p\le k\le 2p-1$, then $C>0$. So assume that $k\le p-1$. By Vandermonde's identity on $\binom{2p}{k+1}$, 
\begin{align*}
    C&=\binom{2p}{k+1}-\binom{p+1}{k+1}-\binom{p}{k+1}\\
     &\ge \binom{p+1}{k+1}+\binom{p-1}{k+1}+(p+1)\binom{p-1}{k}-\binom{p+1}{k+1}-\binom{p}{k+1}=p\binom{p-1}{k}> 0.
\end{align*}
This complets the proof of the part.
\end{proof}

\begin{lemma}The following inequalities hold.
\begin{enumerate}[{\rm (i) }]
\item If  $d\ge 4$ and each $1\le k\le d-2$, then 
\begin{equation*}
\tau_{k}(3d-3,d)\le \tau_{k}(3d-2,d). 
\end{equation*}
\item If  $d\ge 5$, $\ell\ge 1$, and  $2\le k \le d-2$, then
\begin{equation*}
\begin{aligned}
\ff_{k}(2d+\ell,d)&<\hh_{k}(2(d-1)+\ell,d-1)+\ff_{k-1}(2(d-1)+\ell,d-1)+\binom{d-1}{k}\\
&\quad +\binom{d-2}{k}-\binom{d-4}{k}.
\end{aligned}
\end{equation*}
\item If  $d\ge 5$ and $\ell\ge 1$, then 
\begin{equation*}
\ff_{1}(2d+\ell,d)< \hh_{1}(2(d-1)+\ell,d-1)+2(d-1)+\ell+d+1.  
\end{equation*} 
\item If  $d\ge 5$, $\ell\ge 1$, and  $2\le k \le d-2$, then 
\begin{equation*}
\begin{aligned}
\ff_{k}(2d+\ell,d)&< \hh_{k}(2(d-1)+\ell,d-1)+\hh_{k-1}(2(d-1)+\ell,d-1)+\binom{d-1}{k}\\
&\quad +\binom{d-2}{k}-\binom{d-4}{k}.
\end{aligned}
\end{equation*}
\item If  $d\ge 5$,  $\ell\ge 3$,  and $3\le r\le \ell+1$, then  
\begin{equation*}
\ff_{1}(2d+\ell,d)< \hh_{1}(2(d-1)+\ell-r+2,d-1)+2(d-1)+\ell-r+2+\sum_{i=1}^{r-1}d-i.
\end{equation*} 

\item If  $d\ge 5$, $\ell\ge 3$, $3\le r\le \ell+1$, and  $2\le k \le d-2$, then 
\begin{equation*}
\ff_{k}(2d+\ell,d)< \hh_{k}(2(d-1)+\ell-r+2,d-1)+\eta_{k-1}(2(d-1)+\ell-r+2,d-1)+\sum_{i=1}^{r-1}\binom{d-i}{k}. 
\end{equation*}
\item If  $d\ge 5$, $\ell\ge 3$, $3\le r\le \ell+1$, and  $2\le k \le d-2$, then
\begin{equation*}
\ff_{k}(2d+\ell,d)<\hh_{k}(2(d-1)+\ell-r+2,d-1)+\tau_{k-1}(2(d-1)+\ell-r+2,d-1)+\sum_{i=1}^{r-1}\binom{d-i}{k}. 
\end{equation*}
\end{enumerate}
\label[lemma]{lem:combinatorial-ineq}
\end{lemma}
\begin{proof} (i) We state the functions $\hh_{k}(3d-3,d)$ and $\hh_{k}(3d-2,d)$ (see \cref{lem:lower-bound-dplus2-facets} and \cref{lem:lower-bound-dplus2-facets-extra}): 
\begin{align*}
\hh_{k}(3d-3,d)&=\binom{d+1}{k+1}+\binom{d}{k+1}+\binom{d-1}{k+1} -\binom{2}{k+1}-\binom{1}{k+1}\\
\hh_{k}(3d-2,d)&=\binom{d+1}{k+1}+\binom{d}{k+1}+\binom{d-1}{k+1}+\binom{d-2}{k+1} -\binom{\floor{d/3}+1}{k+1}\\
&\quad -\binom{\floor{d/3}}{k+1}-\binom{\floor{d/3}-1}{k+1}
\end{align*}
From the expressions, we find that $\hh_{d-2}(3d-3,d)=\hh_{d-2}(3d-2,d)$. We assume $1\le k\le d-3$. It is not difficult to verify the statement for $d=4,5,6,7,8$. From now on, we assume $d\ge 9$.
% \begin{align*}
% d=4:&\; 18=\tau_{1}(3d-3,d)< \tau_{1}(3d-2,d)=19;\\
% d=5:& \; 30 =\tau_{1}(3d-3,d)<\tau_{1}(3d-2,d)=33, \; 34=\tau_{2}(3d-3,d)<\tau_{2}(3d-2,d)=35;\\
%  d=6:& \; 45=\tau_{1}(3d-3,d)<\tau_{1}(3d-2,d)=48,\;  65=\tau_{2}(3d-3,d)<\tau_{2}(3d-2,d)=68, \\
%  &\; 55=\tau_{3}(3d-3,d)<\tau_{2}(3d-2,d)=56.
% \end{align*}
We compute $\f_{k}(d):=\tau_{k}(3d-2,d)-\tau_{k}(3d-3,d)$:
\begin{align*}
\f_{k}(d)=\binom{d-2}{k+1} -\binom{\floor{d/3}+1}{k+1}
 -\binom{\floor{d/3}}{k+1}-\binom{\floor{d/3}-1}{k+1}+\binom{2}{k+1}.
\end{align*}
For $d\ge 9$ and $1\le k\le d-3$, we show that  $\binom{d-2}{k+1}\ge 3\binom{\floor{d/3}+1}{k+1}$. This is trivial if $\floor{d/3}+1< k+1$, so assume otherwise.   Since $3\floor{d/3}-2\le d-2\le 3\floor{d/3}$, we find that $\binom{d-2}{k+1}\ge\binom{3\floor{d/3}-2}{k+1}$. Let $t:=\floor{d/3}$. Then $t\ge 3$.  Additionally, let
\begin{align*}
    R:=\frac{\binom{3t-2}{k+1}}{\binom{t+1}{k+1}}=\prod_{i=0}^{k}\frac{3t-2-i}{t+1-i}.
\end{align*}
We must show that $R\ge 3$. 
  Because $(3t-2-i)/(t+1-i)\ge (3t-2)/(t+1)$ for all $t\ge 3$  and $i\in [0\ldots k]$, we have $R\ge \paren*{\frac{3t-2}{t+1}}^{k+1}$. The function $g(t):=(3t-2)/(t+1)$ is strictly increasing in $t$ as $g'(t)>0$ for $t>0$. Thus, for $t\ge 3$, we find that $g(t)\ge g(3)=7/4$. Hence 
\begin{align*}
    R\ge \paren*{\frac{3t-2}{t+1}}^{k+1}\ge \paren*{\frac{7}{4}}^{k+1}\ge \paren*{\frac{7}{4}}^{2}=\frac{49}{16}>3.
\end{align*}

Also, for $d\ge 9$ and $\floor{d/3}+1\ge k+1$  we find that 
\begin{align*}
3\binom{\floor{d/3}+1}{k+1}>\binom{\floor{d/3}+1}{k+1}
 +\binom{\floor{d/3}}{k+1}+\binom{\floor{d/3}-1}{k+1}.
\end{align*}
These two inequalities establish that $\f_{k}(d)\ge 0$ for $1\le k\le d-3$.

(ii) %\blue{(By adding an extra term $-\binom{1}{k}$ in the expression of $\eta_{k-1}(2(d-1)+\ell,d-1)$, the argument seems to work for $k=1$ as well; we just need to add a discussion for $k=1$ and $\ell \ge d-3$: $A-B=\binom{d-1}{1}-\binom{d-4}{1}-\binom{1}{1}+\binom{d-(\ell+1)}{2}-\binom{2}{2}-\binom{1}{2}=1>0$. )}
Let $A:=\hh_{k}(2(d-1)+\ell,d-1)+\ff_{k-1}(2(d-1)+\ell,d-1)+\binom{d-1}{k}+\binom{d-2}{k}-\binom{d-4}{k}$ and $B:=\ff_{k}(2d+\ell,d)$. If $\ell\ge d-4$, then we use $\hh_{k}(2(d-1)+d-4)$ in $A$, while if $\ell\ge d-2$, then we used $\ff_{k-1}(2(d-1)+d-2,d-1)$ in $A$. Additionally, if $\ell\ge d-1$, then we use $\ff_{k}(2d+d-1,d)$ in $B$. 

For $\ell\le d-2$, we simplify $A-B$ using \cref{lem:combinatorial-identities}(ii):
\begin{align*}
A-B&=\brac*{\binom{d}{k+1}+\binom{d-1}{k+1}+\binom{d-2}{k+1}-\binom{d-a+1}{k+1}-\binom{d-a}{k+1}}\\
&\quad +\brac*{\binom{d}{k}+2\binom{d-1}{k}-\binom{d-(\ell+1)}{k}}+\binom{d-1}{k}+\binom{d-2}{k}-\binom{d-4}{k}\\
&\quad -\brac*{\binom{d}{k}+3\binom{d}{k+1}-\binom{d-\ell}{k+1}}\\
&=\binom{d-1}{k}-\binom{d-4}{k}+\binom{d-(\ell+1)}{k+1}-\binom{d-a+1}{k+1}-\binom{d-a}{k+1}.
\end{align*}
Here  $a=\floor{(d+\ell)/2}+1$ for $\ell\le d-4$ and $a=d-1$ for $\ell\ge d-3$.

We have that $\binom{d-1}{k}-\binom{d-4}{k}=\binom{d-2}{k-1}+\binom{d-3}{k-1}+\binom{d-4}{k-1}\ge d-2>0$. 
If $\ell\le d-4$, then $\binom{d-(\ell+1)}{k+1}-\binom{d-a+1}{k+1}-\binom{d-a}{k+1} \ge 0$ for $k\in [2\ldots d-2]$ (\cref{lem:combinatorial-identities}(vii)). If  $\ell\ge d-3$, then $a=d-1$ in $\hh_{k}(2(d-1)+\ell,d-1)$, implying that $A-B\ge \binom{d-1}{k}-\binom{d-4}{k}-\binom{2}{k+1}-\binom{1}{k+1}>0$. This completes the proof of the part.

(iii) We proceed as in Part (ii). Let $A:=\hh_{1}(2(d-1)+\ell,d-1)+2(d-1)+\ell+d+1$ and $B:=\ff_{1}(2d+\ell,d)$.  If $\ell\ge d-4$, then we use $\hh_{1}(2(d-1)+d-4,d-1)$ in $A$, while if $\ell\ge d-1$, then we use $\ff_{1}(2d+d-1,d)$ in $B$. 

For $\ell\le d-2$, we simplify $A-B$ using \cref{lem:combinatorial-identities}(ii):
\begin{align*}
A-B&=\brac*{\binom{d}{2}+\binom{d-1}{2}+\binom{d-2}{2}-\binom{d-a+1}{2}-\binom{d-a}{2}}  +2(d-1)+\ell\\
&\quad+d+1-\brac*{\binom{d}{1}+3\binom{d}{2}-\binom{d-\ell}{2}}\\
&=2+\ell-(d-1)+\binom{d-\ell}{2}-\binom{d-a+1}{2}-\binom{d-a}{2}.
\end{align*}
Here  $a=\floor{(d+\ell)/2}+1$ for $\ell\le d-4$ and $a=d-1$ for $\ell\ge d-3$.

If $\ell\le d-4$, then 
\begin{align*}
  A-B&=\binom{d-(\ell+1)}{2}-\binom{d-a+1}{2}-\binom{d-a}{2}+2>0.
  %\ge 0
\end{align*}
From \cref{lem:combinatorial-identities}(vii) for $k=1$, it follows that $\binom{d-(\ell+1)}{2}-\binom{d-a+1}{2}-\binom{d-a}{2}\ge 0$, implying that $A-B>0$.

If  $\ell\ge d-3$, then $a=d-1$ in $\hh_{1}(2(d-1)+\ell,d-1)$, in which case 
\begin{align*}
A-B&=2+\ell-(d-1)+\binom{d-\ell}{2}-\binom{2}{2}-\binom{1}{2} > 0:
\end{align*}	 
When $d-\ell\le 0$, $A-B\ge 2$. For the other values of $d-\ell$, we find that $(d-\ell,A-B)=(3,2),(2,1),(1,1)$. This  completes the proof of the part.

(iv) Let $A:=\hh_{k}(2(d-1)+\ell,d-1)+\hh_{k-1}(2(d-1)+\ell,d-1)+\binom{d-1}{k}+\binom{d-2}{k}-\binom{d-4}{k}$ and $B:=\ff_{k}(2d+\ell,d)$. If $\ell\ge d-4$, then we use $\ell=d-4$ in $A$, while if $\ell\ge d-1$, then we use $\ell=d-1$ in $B$. 

For $\ell\le d-2$, we simplify $A-B$ using \cref{lem:combinatorial-identities}(ii):
\begin{align*}
A-B&=\brac*{\binom{d}{k+1}+\binom{d-1}{k+1}+\binom{d-2}{k+1}-\binom{d-a+1}{k+1}-\binom{d-a}{k+1}}\\
&\quad +\brac*{\binom{d}{k}+\binom{d-1}{k}+\binom{d-2}{k}-\binom{d-a+1}{k}-\binom{d-a}{k}}\\
&\quad +\binom{d-1}{k}+\binom{d-2}{k}-\binom{d-4}{k} -\brac*{\binom{d+1}{k+1}+2\binom{d}{k+1}-\binom{d-\ell}{k+1}}\\
&=\binom{d-2}{k}-\binom{d-4}{k}+\binom{d-\ell}{k+1}-\binom{d-a+2}{k+1}-\binom{d-a+1}{k+1}.
\end{align*}
Here  $a=\floor{(d+\ell)/2}+1$ for $\ell\le d-4$ and $a=d-1$ for $\ell\ge d-3$.

We have that $\binom{d-2}{k}-\binom{d-4}{k}=\binom{d-3}{k-1}+\binom{d-4}{k-1}\ge 1$, where equality holds if and only if $k=d-2$. Additionally, if $\ell\le d-4$, then $\binom{d-\ell}{k+1}-\binom{d-a+2}{k+1}-\binom{d-a+1}{k+1} \ge 0$ for $k\in[2\ldots d-2]$. Indeed, for $k\in[2\ldots d-2]$, 
\begin{align*}
    \binom{d-\ell}{k+1}-\binom{d-a+2}{k+1}-\binom{d-a+1}{k+1} \ge \binom{d-\ell-1}{k+1}-\binom{d-a+1}{k+1}-\binom{d-a}{k+1} \ge 0.
\end{align*}
The two inequalities follows from \cref{lem:combinatorial-identities}(vii). Hence $A-B\ge 1$ in this case. If  $\ell\ge d-3$, then $a=d-1$, in which case $\binom{d-\ell}{k+1}-\binom{3}{k+1}-\binom{2}{k+1} \ge -1$ with equality only when $k=2$. For $k=2$ and $d\ge 5$, we find $\binom{d-2}{k}-\binom{d-4}{k}=\binom{d-3}{k-1}+\binom{d-4}{k-1}\ge 3$. Hence $A-B>0$ in this case as well. 
The proof of the part is complete.

(v) Let $A:=\hh_{1}(2(d-1)+\ell-r+2,d-1)+2(d-1)+\ell-r+2+\sum_{i=1}^{r-1} d-i$ and $B:=\ff_{1}(2d+\ell,d)$.  If $\ell-r+2\ge d-4$, then we use $\hh_{1}(2(d-1)+d-4,d-1)$ in $A$, while if $\ell\ge d-1$, then we use $\ff_{1}(2d+d-1,d)$ in $B$. 

Let  $a:=\floor{(d+\ell-r+2)/2}+1$ for $\ell-r+2\le d-4$ and $a:=d-1$ for $\ell-r+2\ge d-3$. We simplify $A-B$ using \cref{lem:combinatorial-identities}(ii):
\begin{align*}
A-B&=\brac*{\binom{d}{2}+\binom{d-1}{2}+\binom{d-2}{2}-\binom{d-a+1}{2}-\binom{d-a}{2}}  +2(d-1)\\
&\quad +\ell-r+2+\sum_{i=1}^{r-1} d-i-\brac*{\binom{d}{1}+3\binom{d}{2}-\binom{d-\ell}{2}}\\
&=1+\ell-r+\binom{d-\ell}{2}+\sum_{i=1}^{r-3}(d-2-i)-\binom{d-a+1}{2}-\binom{d-a}{2}.
\end{align*}

If $\ell-r+2\le d-4$, then 
\begin{align*}
  A-B&\ge  1+\ell-r+\binom{d-\ell+r-3}{2}-\binom{d-a+1}{2}-\binom{d-a}{2},
\end{align*}
since $\binom{d-\ell}{2}+\sum_{i=1}^{r-3}(d-2-i)\ge \binom{d-\ell+r-3}{2}$ for $3\le r\le \ell+1$. Additionally, we find that
$\binom{d-\ell+r-3}{2}-\binom{d-a+1}{2}-\binom{d-a}{2} > 0$ in this case. Indeed, setting $\ell':=\ell-r+2$ transforms the  expression into $\binom{d-\ell'-1}{2}-\binom{d-a+1}{2}-\binom{d-a}{2}$. By \cref{lem:combinatorial-identities}(vii), $\binom{d-\ell'-1}{2}-\binom{d-a+1}{2}-\binom{d-a}{2}>0$ for $1=k<d-\ell'-1=d-\ell+r-3$, which is true for $\ell-r+2\le d-4$. Hence $A-B>0$. 

Assume that $\ell-r+2\ge d-3$. Then $a=d-1$ in $\hh_{1}(2(d-1)+\ell-r+2,d-1)=\hh_{1}(2(d-1)+d-4,d-1)$, in which case, for $d\ge 5$, $A-B=\ell-r+\binom{d-\ell}{2}+\sum_{i=1}^{r-3} (d-2-i)>0$. 
This  completes the proof of the part.

(vi) Let $A:=\hh_{k}(2(d-1)+\ell-r+2,d-1)+\eta_{k-1}(2(d-1)+\ell-r+2,d-1)+\sum_{i=1}^{r-1}\binom{d-i}{k}$ and $B:=\ff_{k}(2d+\ell,d)$.  If $\ell-r+2\ge d-4$, then we use $\ell-r+2=d-4$ in $\hh_{k}(2(d-1)+\ell-r+2,d-1)$, while if $\ell-r+2\ge d-2$, then we use $\ell-r+2=d-2$ in $\eta_{k-1}(2(d-1)+\ell-r+2,d-1)$. Additionally, if $\ell\ge d-1$, then we use $\ell=d-1$ in $B$. 

Let $a:=\floor{(d+\ell-r+2)/2}+1$ for $\ell-r+2\le d-4$  and $a: =d-1$ for $\ell-r+2\ge d-3$. For $\ell-r+2\le d-4$, we simplify $A-B$ using \cref{lem:combinatorial-identities}(ii):  
\begin{align*}
A-B&=\brac*{\binom{d}{k+1}+\binom{d-1}{k+1}+\binom{d-2}{k+1}-\binom{d-a+1}{k+1}-\binom{d-a}{k+1}}\\
&\quad +\brac*{\binom{d}{k}+2\binom{d-1}{k}-\binom{d-(\ell-r+3)}{k}}+\sum_{i=1}^{r-1}\binom{d-i}{k}\\
&\quad -\brac*{\binom{d}{k}+3\binom{d}{k+1}-\binom{d-\ell}{k+1}}\\
&=-\binom{d-a+1}{k+1}-\binom{d-a}{k+1}+\binom{d-1}{k}-\binom{d-\ell+r-3}{k}+\sum_{i=1}^{r-3}\binom{d-2-i}{k} +\binom{d-\ell}{k+1} \\
%\end{align*}
%If $\ell-r+2\le d-4$, then
%\blue{\begin{align*}
%    A-B& = \brac*{\binom{d-\ell}{k+1}-\binom{d-a+1}{k+1}-\binom{d-a}{k+1}}+\brac*{\binom{d-1}{k}-\binom{d-\ell+r-3}{k}} \\
%    & \quad +\sum_{i=1}^{r-3}\binom{d-2-i}{k}-\blue{\binom{1}{k}} \\
%    & \ge \binom{d-1}{k}-\binom{d-\ell+r-3}{k}-\blue{\binom{1}{k}}, \\
%    & \ge \binom{d-1}{k}-\binom{d-2}{k}-\blue{\binom{1}{k}} >0.
%\end{align*}
%}
&\ge -\binom{d-a+1}{k+1}-\binom{d-a}{k+1}+\binom{d-1}{k}-\binom{d-\ell+r-3}{k} \\
  &\quad+\brac*{\binom{d-2}{k+1}-\binom{d-r+1}{k+1}}+\binom{d-\ell}{k+1}\\
  &\ge \binom{d-\ell+r-3}{k+1}-\binom{d-a+1}{k+1}-\binom{d-a}{k+1}+\binom{d-1}{k}-\binom{d-\ell+r-3}{k}.
\end{align*}
%\blue{For $k\ge 2$, $\binom{d-2}{k+1}-\binom{d-r+1}{k+1}+\binom{d-\ell}{k+1}\ge \binom{d-\ell+r-3}{k+1}$ for $3\le r\le \ell+1$. Additionally, we find that $\binom{d-\ell+r-3}{k+1}-\binom{d-a+1}{k+1}-\binom{d-a}{k+1} \ge 0$ and $\binom{d-1}{k}-\binom{d-\ell+r-3}{k}> 0$, concluding that $A-B > 0$. }
 %For $k=1$, $\binom{d-2}{2}-\binom{d-r+1}{2}+\binom{d-\ell}{2}\ge \binom{d-\ell+r-3}{2}$ for $3\le r\le \ell+1$; $\binom{d-\ell+r-3}{2}-\binom{d-a+1}{2}-\binom{d-a}{2} > 0$; and $\binom{d-1}{1}-\binom{d-\ell+r-3}{1}-\binom{1}{1}\ge 0$. Hence $A-B>0$.
Since $\binom{d-2}{k+1}-\binom{d-r+1}{k+1}+\binom{d-\ell}{k+1}\ge \binom{d-\ell+r-3}{k+1}$ for $3\le r\le \ell+1$. We find that $\binom{d-1}{k}-\binom{d-\ell+r-3}{k}> 0$ for $r\le \ell+1$. Additionally, with $\ell':=\ell-r+2$, the expression 
$\binom{d-\ell+r-3}{k+1}-\binom{d-a+1}{k+1}-\binom{d-a}{k+1}$ becomes $\binom{d-\ell'-1}{k+1}-\binom{d-a+1}{k+1}-\binom{d-a}{k+1}$. By \cref{lem:combinatorial-identities}(vii), $\binom{d-\ell'-1}{k+1}-\binom{d-a+1}{k+1}-\binom{d-a}{k+1}\ge 0$. Hence  $A-B > 0$.
 
 If  $\ell-r+2\ge d-3$, then $a=d-1$ in $\hh_{k}(2(d-1)+\ell-r+2,d-1)$, in which case 
\begin{align*}
A-B&%=-\binom{2}{k+1}-\binom{1}{k+1}+\binom{d-1}{k}-\binom{d-\ell+r-3}{k} \\
%&\quad +\sum_{i=1}^{r-3}\binom{d-2-i}{k}+\binom{d-\ell}{k+1}-\binom{1}{k}.\\
& = \binom{d-1}{k}-\binom{d-\ell+r-3}{k}-\binom{2}{k+1}-\binom{1}{k+1}+\sum_{i=1}^{r-3}\binom{d-2-i}{k}+\binom{d-\ell}{k+1}.
%&=\binom{d-1}{k}-\binom{d-\ell+r-3}{k} +\sum_{i=1}^{r-3}\binom{d-2-i}{k}+\binom{d-\ell}{k+1}-\blue{\binom{1}{k}}>0. 
\end{align*}
For $k\ge 2$, since $d-1>d-\ell+r-3$, it follows that $A-B\ge \binom{d-1}{k}-\binom{d-\ell+r-3}{k}>0$. 
%For $k=1$, $A-B=\ell-r+\binom{d-\ell}{2}+\sum_{i=1}^{r-3}(d-2-i)\ge 0$  (since $\ell-r\ge d-5$, $d\ge 5$), with equality holds if and only if $d=5$, $\ell=r \ge d-1$ and $r=3$. Hence $A-B>0.$
%Note that $d-\ell+r-3\le 2$, implying that $\binom{d-1}{k}-\binom{d-\ell+r-3}{k} \ge \binom{d-1}{k}-\binom{2}{k} >0$ . 
This completes the proof of the part.

(vii)  Let $A:=\hh_{k}(2(d-1)+\ell-r+2,d-1)+\hh_{k-1}(2(d-1)+\ell-r+2,d-1)+\sum_{i=1}^{r-1} \binom{d-i}{k}$ and $B:=\ff_{k}(2d+\ell,d)$. If $\ell-r+2\ge d-4$, then we use $\ell-r+2=d-4$ in $A$, while if $\ell\ge d-1$, then we use $\ell=d-1$ in $B$. 

Let $a:=\floor{(d+\ell-r+2)/2}+1$ for $\ell-r+2\le d-4$ and $a=d-1$ for $\ell-r+2\ge d-3$.  We simplify $A-B$ using \cref{lem:combinatorial-identities}(ii):
\begin{equation}
    \label{eq:combinatorial-ineq-7-1}
    \begin{aligned}
A-B&=\brac*{\binom{d}{k+1}+\binom{d-1}{k+1}+\binom{d-2}{k+1}-\binom{d-a+1}{k+1}-\binom{d-a}{k+1}}\\
&\quad +\brac*{\binom{d}{k}+\binom{d-1}{k}+\binom{d-2}{k}-\binom{d-a+1}{k}-\binom{d-a}{k}}+\sum_{i=1}^{r-1}\binom{d-i}{k} \\
&\quad -\brac*{\binom{d+1}{k+1}+2\binom{d}{k+1}-\binom{d-\ell}{k+1}}\\
&=\sum_{i=1}^{r-2}\binom{d-1-i}{k}+\binom{d-\ell}{k+1}-\binom{d-a+2}{k+1}-\binom{d-a+1}{k+1}.
\end{aligned}
\end{equation}

\textbf{Suppose that $\ell-r+2\le d-4$ in \eqref{eq:combinatorial-ineq-7-1}.} For $k=d-2$, \eqref{eq:combinatorial-ineq-7-1} becomes $A-B=\sum_{i=1}^{r-2}\binom{d-1-i}{k}$, implying $A-B\ge 1$ for $r\ge 3$. Thus assume $2\le k\le d-3$. For  $3\le r\le \ell+1$, by \cref{lem:combinatorial-identities}  we find that
\begin{align*}
  \sum_{i=1}^{r-2}\binom{d-1-i}{k}+\binom{d-\ell}{k+1}&\ge   \sum_{i=1}^{r-2}\binom{d-\ell+r-2-i}{k}+\binom{d-\ell}{k+1}=\binom{d-\ell+r-2}{k+1},
\end{align*}
where the first inequality holds with equality if and only if $r=\ell+1$.  It suffices to show that $\binom{d-\ell+r-2}{k+1}-\binom{d-a+2}{k+1}-\binom{d-a+1}{k+1}\ge 0$ for $3\le r\le \ell$ and $\binom{d-\ell+r-2}{k+1}-\binom{d-a+2}{k+1}-\binom{d-a+1}{k+1}>0$ for $r=\ell+1$.

For $k\in[2\ldots d-3]$, setting $\ell':=\ell-r+2$ transforms  the expression $\binom{d-\ell+r-2}{k+1}-\binom{d-a+2}{k+1}-\binom{d-a+1}{k+1}$  into $\binom{d-\ell'}{k+1}-\binom{d-a+2}{k+1}-\binom{d-a+1}{k+1}$. By \cref{lem:combinatorial-identities}(vii), for $2\le k\le d-3$, we find 
\begin{equation}
\begin{aligned}
\label{eq:combinatorial-ineq-7}
\binom{d-\ell'}{k+1}-\binom{d-a+2}{k+1}-\binom{d-a+1}{k+1} 
\ge & \binom{d-\ell'-1}{k+1}-\binom{d-a+1}{k+1}-\binom{d-a}{k+1}.
\end{aligned}
\end{equation}
%\blue{where we have equality if and only if $k\ge d-\ell'=d-\ell+r-2$. In the case that $r=\ell+1$, we have equality if and only if $k\ge d-1$.}

Also, by \cref{lem:combinatorial-identities}(vii), 
\begin{equation}
\begin{aligned}
\label{eq:combinatorial-ineq-8}
 \binom{d-\ell'-1}{k+1}-\binom{d-a+1}{k+1}-\binom{d-a}{k+1}\ge 0.
\end{aligned}
\end{equation}
Combining \eqref{eq:combinatorial-ineq-7} and \eqref{eq:combinatorial-ineq-8} we get $A-B\ge 0$ for $3\le r\le \ell+1$. Inequality \eqref{eq:combinatorial-ineq-8} is strict if and only if $k< d-\ell'-1=d-\ell+r-3$ (\cref{lem:combinatorial-identities}(vii)). If $r=\ell+1$, then   \eqref{eq:combinatorial-ineq-8} is strict if $k\le d-3$. Hence, $A-B >0$ for  $r=\ell+1$ and all $k\in [2\ldots d-3]$. 

\textbf{Suppose that $\ell-r+2\ge d-3$ in \eqref{eq:combinatorial-ineq-7-1}.} Then $a=d-1$, in which case $d-\ell+r-2\le 3$, and \eqref{eq:combinatorial-ineq-7-1} becomes \begin{equation*}
    \sum_{i=1}^{r-2}\binom{d-1-i}{k}+\binom{d-\ell}{k+1}-\binom{3}{k+1}-\binom{2}{k+1} > 0,
\end{equation*}
for $d\ge 5$ and  $k\in[2\ldots d-2]$. The proof of the part is complete.
% The proof of the inequality $\binom{d-\ell+r-2}{k+1}-\binom{d-a+2}{k+1}-\binom{d-a+1}{k+1} \ge 0$ goes as follows. It suffices to consider when $d+\ell-r$ is odd.
% \begin{equation*}
%     \begin{aligned}
%         & \binom{d-\ell+r-2}{k+1}-\binom{d-a+2}{k+1}-\binom{d-a+1}{k+1} \\
%         &\ge \binom{d-\ell+r-2}{k+1}-\binom{\frac{d-\ell+r+1}{2}}{k+1}-\binom{\frac{d-\ell+r-1}{2}}{k+1} \\
%         & \text{let } 2p=d-\ell+r-1, \text{ then } 2p\ge 5, \\
%         & =\binom{2p-1}{k+1}-\binom{p+1}{k+1}-\binom{p}{k+1} \\
%         & =\frac{(2p-1)\cdots(2p-k-1)-(2p-k+1) p\cdots(p-k+1)}{(k+1)!}\\
%         & > 0.        
%     \end{aligned}
% \end{equation*}
% In the second last equation, we have $2p-1\ge 2p-k+1$ as $k\ge 2$ and $2p-2>p$ as $p\ge 3$.
\end{proof} 

% \section{Acknowledgments}

% \red{We would like to express our gratitude to XXX for the stimulating discussions on the topic and for reviewing an earlier version of the paper.}

\bibliographystyle{amsplain}

\end{document}